\documentclass[final]{siamart0216}
\usepackage{amsfonts}
\usepackage{todonotes}
\usepackage{amsfonts,amsmath,amssymb}
\usepackage{mathrsfs,mathtools,stmaryrd,wasysym}
\usepackage{enumerate}
\usepackage{xspace,mydef} 
\usepackage{esint}
\usepackage{graphicx,psfrag}
\usepackage{hyperref}
\usepackage{pgf,tikz,pgfplots}
\usepackage{pstricks-add}
\usepackage{enumitem}
\usepackage{yfonts}

\usetikzlibrary{arrows}
\usepackage{rotating}
\DeclareMathAlphabet{\mathpzc}{OT1}{pzc}{m}{it}

\usepackage[notcite,notref]{showkeys} 

\DeclareMathOperator{\supp}{supp}

\newsiamremark{remark}{Remark}
\newsiamremark{assumption}{Assumption}

\newcommand{\norm}[1]{{\left\vert\kern-0.25ex\left\vert\kern-0.25ex\left\vert #1 
\right\vert\kern-0.25ex\right\vert\kern-0.25ex\right\vert}}

\newcommand{\TheTitle}{
Finite element discretizations of a convective Brinkman--Forchheimer model under singular forcing}
\newcommand{\ShortTitle}{A Brinkman--Darcy--Forchheimer model under singular forcing}
\newcommand{\TheAuthors}{A.~Allendes, G.~Campa\~na and E.~Ot\'arola}

\headers{\ShortTitle}{\TheAuthors}

\title{{\TheTitle}\thanks{AA is partially supported by ANID through FONDECYT project 1210729. GC is partially supported by ANID--Subdirecci\'on de Capital Humano/Doctorado Nacional/2020--21200920}. EO is partially supported by ANID through FONDECYT project 1220156.}

\author{Alejandro Allendes\thanks{Departamento de Matem\'atica, Universidad T\'ecnica Federico Santa Mar\'ia, Valpara\'iso, Chile.
(\email{alejandro.allendes@usm.cl}, \url{http://aallendes.mat.utfsm.cl/}).}
\and 
Gilberto Campa\~na\thanks{Departamento de Ciencias, Universidad T\'ecnica Federico Santa Mar\'ia, Valpara\'iso, Chile.
(\email{gilberto.campana@usm.cl}.}
\and
Enrique Ot\'arola\thanks{Departamento de Matem\'atica, Universidad T\'ecnica Federico Santa Mar\'ia, Valpara\'iso, Chile.
(\email{enrique.otarola@usm.cl}, \url{http://eotarola.mat.utfsm.cl/}).}   
}

\ifpdf
\hypersetup{
  pdftitle={\TheTitle},
  pdfauthor={\TheAuthors}
}
\fi

\date{Draft version of \today.}

\begin{document}
\maketitle

\begin{abstract}
In two-dimensional Lipschitz domains, we analyze a Brinkman--Darcy--Forchheimer problem on the weighted spaces $\mathbf{H}_0^1(\omega,\Omega) \times L^2(\omega,\Omega)/\mathbb{R}$, where $\omega$ belongs to the Muckenhoupt class $A_2$. Under a suitable smallness assumption, we prove the existence and uniqueness of a solution. We propose a finite element method and obtain a quasi-best approximation result in the energy norm \emph{\`a la C\'ea} under the assumption that $\Omega$ is convex. We also develop an a posteriori error estimator and study its reliability and efficiency properties. Finally, we develop an adaptive method that yields optimal experimental convergence rates for the numerical examples we perform. 
\end{abstract}




\section{Introduction}
\label{sec:intro}
In this article, we are concerned with the study of existence and finite element approximation results for a Brinkman--Darcy--Forchheimer problem under \emph{rough} or \emph{singular forcing}. Specifically, we will study the following \emph{nonlinear} system of partial differential equations (PDEs):
\begin{equation}\label{eq:model}
-\Delta\mathbf{u} 
+ 
(\mathbf{u}\cdot\nabla)\mathbf{u}
+
|\mathbf{u}|\mathbf{u}
+ 
\mathbf{u}+\nabla \mathsf{p}  =  \mathbf{f}  \text{ in }\Omega, 
\quad
\text{div}~\mathbf{u}  =  0  \text{ in }\Omega,
\quad
\mathbf{u} =  \mathbf{0}  \text{ on }\partial\Omega.
\end{equation}
Here, $\Omega$ denotes an open and bounded domain of $\mathbb{R}^2$ with Lipschitz boundary $\partial\Omega$, $\mathbf{u}$ and $\mathsf{p}$ stand for the velocity and pressure of the fluid, respectively, $\mathbf{f}$ is an externally applied force, and $|\cdot|$ denotes the Euclidean norm. Contrary to what is usually found in the literature, our main source of originality and novelty arises from the fact that $\mathbf{f}$ is singular, say a Dirac measure, so that the problem cannot be understood in the classical setting inherited by the space $\mathbf{H}_0^1(\Omega) \times L^2(\Omega)/\mathbb{R}$.

Darcy's law is a linear relationship that describes the creep flow of Newtonian fluids in porous media. This law is supported by years of experimental data and has numerous applications in engineering. It is therefore not surprising that its analysis and approximation have been studied by several authors. Nevertheless, Darcy's law can be inaccurate when modeling fluid flow through porous media with high Reynolds numbers or through media with high porosity. To overcome this inaccuracy, Forchheimer proposed a modification of Darcy's law in \cite{forchheimer1901wasserbewegung} and formulated the so-called Darcy--Forchheimer equations. Several discretization techniques for the Darcy-Forchheimer equations have been studied in the literature; for a non-exhaustive list, we refer the interested reader to \cite{MR2425154,MR3022234,MR2948707,MR4092292,MR4049400,MR4127956}. The incorporation of $-\Delta \mathbf{u}$ and the convective term $(\mathbf{u}\cdot \nabla) \mathbf{u}$ in the Darcy--Forchheimer equations leads to the so-called convective Brinkman--Forchheimer model \eqref{eq:model}. This model was derived by the authors of \cite{VAFAI1981195} as the governing momentum equation based on local volume averaging and matched asymptotic expansion, assuming a two-dimensional stationary, isotropic, incompressible, homogeneous flow through a fluid-saturated porous medium; see also \cite{lastone}. Further justifications for the inclusion of the so-called Brinkman and convective terms in \eqref{eq:model} were presented later in \cite{VAFAI199511}. We refer the interested reader to \cite{guo2005lattice,SHENOY1994101,MR3636305,MR3000454} for further insights, analysis, and applications of this model. To conclude this paragraph, we would like to mention that recently in \cite{MR3967591,MR4658588,MR4633701} discretization methods with finite elements for the system \eqref{eq:model}, but under smooth forcing, have been considered.

While it is true that the study of finite element methods for \eqref{eq:model} and similar models is mature in a standard setting, applications and models have recently appeared where the motion of a fluid is described by \eqref{eq:model} or a modification of it, but because of the singularity of the forces $\mathbf{f}$, the problem must be understood in a completely different setting, and rigorous approximation techniques are scarce. For example, the author of \cite{Lacouture2015187} models the motion of active thin structures using the Stokes equations (a linear model related to \eqref{eq:model}), with a forcing term corresponding to a linear combination of Dirac measures. A second example comes from PDE-constrained optimization theory. In \cite{MR4304887,MR3936891,MR4548586} a problem is formulated where the state is determined by the stationary Stokes/Navier--Stokes equations, but with a measure-valued control. Finally, we refer the reader to \cite{MR3582412}, where the authors study a class of asymptotically Newtonian fluids (Newtonian under large shear rates) under singular forcing. In particular, the authors provide existence and uniqueness results as well as some regularity properties for solutions; see \cite{MR4408483} for some extensions to convex polyhedral domains.

In this article, we continue our research program focused on the development of finite element solution methods for fluid models under rough forcing. The central idea we pursue is to introduce a weight and work in the corresponding weighted Lebesgue and Sobolev spaces so that singular forcing fits into our functional framework. One of the first references is \cite{MR3906341}, where we proved the well-posedness of the Stokes problem over a reduced class of weighted spaces. Later, in \cite{MR4081912} and \cite{MR3892359}, we developed a priori and a posteriori error estimates, respectively, for suitable finite element approximations of the Stokes problem analyzed in \cite{MR3906341}. The research program continued with the analysis and approximation of the Navier--Stokes equations presented in \cite{MR3998864}. A posteriori error estimates for suitable discretizations of such a nonlinear model can be found in \cite{MR4117306}. Regarding the approximation of coupled problems involving fluid flow equations and a suitable temperature equation under singular forcing we refer the interested reader to \cite{ACFO:24,MR4659334,MR4265062}. This brings us to the present article and its main contributions. Before presenting the main contributions of our work, we would like to refer to \cite{MR4658588,COCQUET2021113008,MR3967591,MR4320857}, where different solution methods for the problem \eqref{eq:model} with $\mathbf{f}$ smooth are discussed. As far as we know, this is the first paper that deals with the numerical approximation of \eqref{eq:model} when $\mathbf{f}$ is singular.

Let us comment on the main contributions of our article:

$\bullet$ \emph{Existence and uniqueness of a solution:} We introduce a notion of weak solution in $\mathbf{H}_0^1(\omega,\Omega)\times L^2(\omega,\Omega)/\mathbb{R}$ and use a fixed-point argument to show that the proposed weak problem admits a unique solution under a suitable smallness assumption on $\mathbf{f}$. To accomplish this task, we first establish the well-posedness of a Brinkman problem in the space $\mathbf{H}_0^1(\omega,\Omega)\times L^2(\omega,\Omega)/\mathbb{R}$ using the continuity method and the well-posedness of the Stokes problem from \cite[Theorem 17]{MR3906341}.

$\bullet$ \emph{Finite element discretization:} We propose a finite element scheme for the problem \eqref{eq:model} based on the following two classical inf-sup stable pairs: the mini element and the Taylor--Hood element. We show that the proposed scheme admits a unique solution. Moreover, we obtain a quasi-best approximation result in the energy norm \emph{\`a la C\'ea}. We must immediately note that since $\mathbf{f}$ is very singular, $(\mathbf{u},\mathsf{p})$ is not expected to have any regularity properties beyond those necessary for the problem to be well-posed. Consequently, convergence rates in the energy norm cannot be obtained from the derived quasi-best approximation result.

$\bullet$ \emph{A posteriori error analysis:} Due to the singularity of the force $\mathbf{f}$ in \eqref{eq:model}, no smooth solutions are to be expected. This lack of smoothness motivates the development of a posteriori error estimators and adaptive methods to efficiently solve \eqref{eq:model}. We develop an a posteriori error estimator based on residuals for the proposed finite element discretization method. We show that the developed estimator is globally reliable; see Theorem \ref{thm:globa_reliability}. In Theorem \ref{thm:efficiency}, we investigate efficiency properties for the proposed local indicators. Furthermore, we develop an adaptive finite element method based on the proposed error estimator and present numerical experiments in convex and non-convex domains.

The article is structured as follows. In Section \ref{sec:notation}, we introduce the notation and recall basic facts about weights and weighted Sobolev spaces. In Section \ref{sec:Brinkman}, we analyze a Brinkman problem on weighted spaces. In Section \ref{sec:coupled_problem}, we introduce a weak formulation for the system \eqref{eq:model} and establish a well-posedness result. A finite element method is presented in Section \ref{sec:fem}. Here we also obtain a quasi-best approximation result \emph{\`a la C\'ea}. In Section \ref{sec:a_posteriori_anal}, we propose an error estimator for suitable inf-sup stable finite element pairs and introduce a Ritz projection on weighted spaces. We prove that the energy norm of the error can be bounded by the energy norm of the Ritz projection and obtain the global reliability of the proposed estimator. We also study local efficiency estimates. We conclude the paper with Section \ref{sec:numericalexperiments}, where we present a series of numerical experiments to illustrate our results.


\section{Notation and preliminaries} 
\label{sec:notation}
Let us set the notation and specify the framework in which we will work.

\subsection{Notation}
\label{subsec:notation}

We adopt classical notation for Lebesgue and Sobolev spaces. Let $\mathcal{W}$ and $\mathcal{Z}$ be Banach function spaces. We write $\mathcal{W} \hookrightarrow\mathcal{Z}$ to indicate that $\mathcal{W}$ is continuously embedded in $\mathcal{Z}$. We denote by $\mathcal{W}'$ and $\|\cdot\|_{\mathcal{W}}$ the dual space and norm of $\mathcal{W}$, respectively. For $\mathfrak{p} \in (1, \infty)$, we denote by $\mathfrak{p}'$ its H\"older conjugate, which is such that $1/\mathfrak{p} + 1/\mathfrak{p}' = 1$. The relation $a \lesssim b$ means that there exists a positive constant independent of $a$, $b$, and the discretization parameters such that $a \leq Cb$. The value of $C$ may vary at each occurrence. Finally, we note that the spaces of vector-valued functions and their elements are indicated by boldface.

\subsection{Weighted function spaces}
A weight is a locally integrable function on $\mathbb{R}^2$ defined to be nonnegative. Let $\omega$ be a weight and let $E \subset \mathbb{R}^2$ be a measurable set. We define $\omega(E) = \int_{E} \omega \mathrm{d}\mathbf{x}$. If the measurable set $E \subset \mathbb{R}^2$ has positive Lebesgue measure, we define $\fint_{E} \omega(\mathbf{x}) \mathrm{d}\mathbf{x} = \omega(E)/|E|$.

In what follows, we turn our attention to the weights $\omega$ belonging to the Muckenhoupt class $A_p$: Let $p\in[1,\infty)$. A weight $\omega$ is said to belong to the Muckenhoupt class $A_p$ if \cite{MR2797562,MR1800316,MR293384,MR1774162}
\begin{equation}\label{eq:weight}
\begin{array}{rcl}
[\omega]_{A_1} & := & \displaystyle\sup_B \left(\fint_B\omega \right)\sup_{\mathbf{x}\in B}\dfrac{1}{\omega(\mathbf{x})}<\infty,
\qquad p=1,
\\
[10pt]
[\omega]_{A_p}& := & \displaystyle\sup_B \left(\fint_B\omega \right)\left(\fint_B\omega^{\frac{1}{1-p}}\right)^{p-1}<\infty,
\qquad p\in(1,\infty),
\end{array}
\end{equation}
where the supremum is taken over all balls $B\in\mathbb{R}^2$. The class $A_{\infty}$ is defined by $A_{\infty}:= \cup_{p < \infty} A_p$. For $p\in[1,\infty)$, $[\omega]_{A_p}$ is called the Muckenhoupt characteristic of $\omega$. When $p \in (1,\infty)$, there is some symmetry in $A_p$ with respect to H\"older conjugate exponents: $\omega'=\omega^{1/(1-p)}\in A_{p'}$ if and only if $\omega\in A_p$ \cite[Remark 1.2.4]{MR1774162}. Finally, we note that if $1 \leq p < q < \infty$, then $A_p \subset A_q$ \cite[Remark 1.2.4]{MR1774162}.

A prototypical example of an $A_p$ weight is a power weight: Let $\alpha\in\mathbb{R}$ and let $\mathbf{z}$ be an interior point in $\Omega$. For $p>1$, the weight
\begin{eqnarray}\label{eq:weight_A2}
\mathsf{d}_{\mathbf{z}}^{\alpha}(\mathbf{x}):=|\mathbf{x}-\mathbf{z}|^{\alpha}
\end{eqnarray}
belongs to $A_p$ if and only if $\alpha\in(-2,2(p-1))$ \cite[Chapter IX, Corollary 4.4]{MR869816}. For this particular weight, there is a neighborhood of $\partial\Omega$ in which the weight is strictly positive and continuous. It is, therefore, appropriate to introduce the following restricted class of Muckenhoupt weights \cite[Definition 2.5]{MR1601373}.

\begin{definition}[restricted class $A_p(D)$]
Let $p\in[1,\infty)$ and let $D\subset\mathbb{R}^2$ be a Lipschitz domain. A weight $\varpi\in A_p$ is said to belong to $A_p(D)$ if there is an open set $\mathcal{G}\subset D$ and $\epsilon,\varpi_l>0$ such that: $\{\mathbf{x}\in D:\textnormal{dist}(\mathbf{x},\partial D)<\epsilon\}\subset\mathcal{G}$, $\varpi\in C(\overline{\mathcal{G}})$, and $\varpi_l\leq \varpi(\mathbf{x})$ for all $\mathbf{x}\in \overline{\mathcal{G}}$.
\end{definition}

Let us now introduce Lebesgue and Sobolev weighted spaces. To this end, let $E \subset \mathbb{R}^2$ be an open set, $p \in [1,\infty)$, and $\omega \in A_p$. The space of Lebesgue $p$-integrable functions with respect to the measure $\omega(\mathbf{x})\mathsf{d}\mathbf{x}$ is denoted by $L^p(\omega,E)$. 
$W^{1,p}(\omega,E)$ is defined as the space of functions $v \in L^p(\omega,E)$ with derivatives $D^{\alpha} v \in \mathbf{L}^p(\omega,E)$ for $|\alpha| \leq 1$; derivatives being understood in a week sense. We equip $W^{1,p}(\omega,E)$ with the norm \cite[Section 2.1]{MR1774162}
\begin{equation}
\|  \cdot \|: W^{1,p}(\omega,E) \rightarrow \mathbb{R},
\quad
\|v\|_{W^{1,p}(\omega,E)}:= \left( \|v\|^{p}_{L^p(\omega,E)}+\|\nabla v\|^{p}_{\mathbf{L}^p(\omega,E)} \right)^{\frac{1}{p}}.
\label{eq:norm}
\end{equation}
We also define the space $W_0^{1,p}(\omega,E)$ as the closure of $C_0^{\infty}(E)$ in $W^{1,p}(\omega,E)$. When $p=2$, we set $H^{1}(\omega,E):=W^{1,p}(\omega,E)$ and $H_0^{1}(\omega,E):=W_0^{1,p}(\omega,E)$. The spaces $L^p(\omega,E)$, $W^{1,p}(\omega,E)$, and $W_0^{1,p}(\omega,E)$ are Banach spaces \cite[Proposition 2.1.2]{MR1774162} and smooth functions are dense \cite[Corollary 2.1.6]{MR1774162}; see also \cite[Theorem 1]{MR2491902}. Moreover, given the weighted Poincar\'e inequality of \cite[Theorem 1.3]{MR643158}, we have that $\| \nabla v \|_{\mathbf{L}^p(\omega,E)}$ is an equivalent norm to the norm defined in \eqref{eq:norm} on the space $W^{1,p}_0(\omega,E)$.

\section{A Brinkman problem under singular forcing}
\label{sec:Brinkman}
In this section, we study the well-posedness of the following Brinkman problem: Find $(\mathbf{u},\mathsf{p})$ such that
\begin{equation}\label{eq:brinkman}
-\Delta\mathbf{u} +\mathbf{u}+ \nabla \mathsf{p}  =  \mathbf{f}\text{ in }\Omega, \quad \text{div}~\mathbf{u}  =  g  \text{ in }\Omega, \quad\mathbf{u} =  \mathbf{0}  \text{ on }\partial\Omega,
\end{equation}
where we allow the data $\mathbf{f}$ and $g$ to be singular. The analysis of problem \eqref{eq:brinkman} is a key step to establish the well-posedness of the Brinkman--Darcy--Forchheimer model \eqref{eq:model}.

We begin our studies by proposing a weak formulation for \eqref{eq:brinkman}. Given $\omega\in A_2$, $\mathbf{f}\in\mathbf{H}^{-1}(\omega,\Omega)$, and $g\in L^2(\omega,\Omega)/\mathbb{R}$, find $(\mathbf{u},\mathsf{p})\in \mathbf{H}_0^1(\omega,\Omega)\times L^2(\omega,\Omega)/\mathbb{R}$ such that
\begin{equation}\label{eq:brinkman_problem}
\displaystyle\int_\Omega(\nabla\mathbf{u}:\nabla\mathbf{v}+\mathbf{u}\cdot\mathbf{v}-\mathsf{p}~\textnormal{div }\mathbf{v}+\mathsf{q}\textnormal{ div }\mathbf{u}) = \langle\mathbf{f},\mathbf{v}\rangle+ \int_{\Omega} g \mathsf{q},
\end{equation}
for all $\mathbf{v}\in \mathbf{H}_0^1(\omega^{-1},\Omega)$ and $\mathsf{q}\in L^2(\omega^{-1},\Omega)/\mathbb{R}$. Here, $\langle\cdot,\cdot\rangle$ stands for the duality pairing between $\mathbf{H}^{-1}(\omega,\Omega) := \mathbf{H}_0^1(\omega^{-1},\Omega)'$ and $\mathbf{H}_0^1(\omega^{-1},\Omega)$. Note that because of the boundary conditions for $\mathbf{u}$ in problem \eqref{eq:brinkman}, we must necessarily have $\int_{\Omega}g=0$.

Let us introduce
\begin{equation}
 \label{eq:XandY}
 \mathcal{X}:=\mathbf{H}_0^1(\omega,\Omega)\times L^2(\omega,\Omega)/\mathbb{R},
 \qquad
 \mathcal{Y}:=\mathbf{H}_0^1(\omega^{-1},\Omega)\times L^2(\omega^{-1},\Omega)/\mathbb{R}.
\end{equation}
 
The well-posedness of the Brinkman problem is established in the following result.

\begin{theorem}[well-posedness of the Brinkman problem]
\label{eq:theorem_brinkman}
Let $d \in \{2,3\}$ and let $\Omega \subset \mathbb{R}^{d}$ be a bounded Lipschitz domain. Let $\omega\in A_2(\Omega)$. If $\mathbf{f}\in\mathbf{H}^{-1}(\omega,\Omega)$ and $g\in L^2(\omega,\Omega)/\mathbb{R}$, then there exists a unique solution $(\mathbf{u},\mathsf{p})\in \mathbf{H}_0^1(\omega,\Omega)\times L^2(\omega,\Omega)/\mathbb{R}$ of problem \eqref{eq:brinkman_problem}, which satisfies the following estimate
\begin{equation}\label{eq:brinkman_estimate}
\|\nabla \mathbf{u}\|_{\mathbf{L}^2(\omega,\Omega)} +\|\mathsf{p}\|_{L^2(\omega,\Omega)}
\leq 
C_{\mathcal{B}} \left(
\|\mathbf{f}\|_{\mathbf{H}^{-1}(\omega,\Omega)}
+
\|g\|_{L^{2}(\omega,\Omega)} \right),
\qquad
C_{\mathcal{B}}>0.
\end{equation}
\end{theorem}

\begin{proof}
Inspired by the proof of \cite[Theorem 6.8]{MR1814364}, we proceed on the basis of the \emph{method of continuity} presented in \cite[Theorem 5.2]{MR1814364}. We split the proof in four steps.

\emph{Step 1.} \emph{A bounded linear map $\mathcal{S}$ associated to a Stokes problem.} We define
\begin{eqnarray}\label{eq:L0_map}
\mathcal{S}:  \mathcal{X}\to \mathcal{Y}',
\quad
\langle\mathcal{S}(\mathbf{u},\mathsf{p}),(\mathbf{v},\mathsf{q})\rangle:=\displaystyle\int_\Omega
(\nabla\mathbf{u}:\nabla\mathbf{v}-\mathsf{p}~\textnormal{div }\mathbf{v}+\mathsf{q}~\textnormal{div }\mathbf{u}).
\end{eqnarray}
We notice that $\mathcal{S}$ is a bounded linear operator. In fact, we have the bound
\begin{align*}
\|\mathcal{S}(\mathbf{u},\mathsf{p}) \|_{\mathcal{Y}'}
=
\displaystyle
\sup_{(\mathbf{0},\mathsf{0})\neq(\mathbf{v},\mathsf{q})\in \mathcal{Y}}\dfrac{\langle\mathcal{S}(\mathbf{u},\mathsf{p}),(\mathbf{v},\mathsf{q})\rangle}{\|(\mathbf{v},\mathsf{q})\|_{\mathcal{Y}}}
\lesssim
\|\nabla \mathbf{u}\|_{\mathbf{L}^2(\omega,\Omega)} + \|\mathsf{p}\|_{L^2(\omega,\Omega)},
\end{align*}
where $\|(\mathbf{v},\mathsf{q})\|_{\mathcal{Y}}:=\|\nabla \mathbf{v}\|_{\mathbf{L}^2(\omega^{-1},\Omega)}+\|\mathsf{q}\|_{L^2(\omega^{-1},\Omega)}$. We now introduce the following weak formulation associated with the Stokes operator $\mathcal{S}$. Given ${\mathbf{g}} \in \mathbf{H}^{-1}(\omega,\Omega)$ and $h \in L^{2}(\omega,\Omega)/\mathbb{R}$, find $(\boldsymbol{\varphi},\psi)\in\mathcal{X}$ such that $\langle\mathcal{S}(\boldsymbol{\varphi},\psi),(\mathbf{v},\mathsf{q})\rangle=\langle\mathbf{g},\mathbf{v}\rangle+(h,\mathsf{q})_{L^2(\Omega)}$ for all $(\mathbf{v},\mathsf{q})\in\mathcal{Y}$. The well-posedness of this Stokes system follows from \cite[Theorem 17]{MR3906341}.

\emph{Step 2.} \emph{A bounded linear map $\mathcal{B}$ associated to a Brinkman problem.} We define 
\begin{eqnarray}\label{eq:L1_map}
\mathcal{B}:  \mathcal{X}\!\to\! \mathcal{Y}',
\quad
\langle\mathcal{B}(\mathbf{u},\mathsf{p}),(\mathbf{v},\mathsf{q})\rangle:=\displaystyle\int_\Omega(\nabla\mathbf{u}:\nabla\mathbf{v}\!+\!\mathbf{u}\cdot\mathbf{v}\!-\!\mathsf{p}~\textnormal{div }\mathbf{v}\!+\!\mathsf{q}~\textnormal{div }\mathbf{u}).
\end{eqnarray}
The map $\mathcal{B}$ is linear and bounded. In particular, we have $\|\mathcal{B}(\mathbf{u},\mathsf{p}) \|_{\mathcal{Y}'}\lesssim \|\nabla \mathbf{u}\|_{\mathbf{L}^2(\omega,\Omega)}+\|\mathsf{p}\|_{L^2(\omega,\Omega)}$. With $\mathcal{B}$ at hand, problem \eqref{eq:brinkman_problem} can be equivalently written as follows: Find $(\mathbf{u},\mathsf{p}) \in \mathcal{X}$ such that $\langle \mathcal{B}(\mathbf{u},\mathsf{p}),(\mathbf{v},\mathsf{q}) \rangle = \langle\mathbf{f},\mathbf{v}\rangle+(g,\mathsf{q})_{L^2(\Omega)}$ for all $(\mathbf{v},\mathsf{q})\in\mathcal{Y}$.

\emph{Step 3.} \emph{The a priori estimate \eqref{eq:brinkman_estimate}}. Let us introduce, for $t\in[0,1]$, the operator
\begin{equation}\label{eq:Lt_map}
\mathcal{L}_t: \mathcal{X}\to \mathcal{Y}',
\qquad
\mathcal{L}_t :=(1-t)\mathcal{S} + t\mathcal{B}.
\end{equation}
Observe that $\mathcal{L}_0 = \mathcal{S}$, $\mathcal{L}_1 = \mathcal{B}$, and that $\mathcal{L}_t$ is a linear and bounded operator from $\mathcal{X}$ into $\mathcal{Y}'$. Let us consider the following family of equations: Find $(\mathbf{u},\mathsf{p})\in\mathcal{X}$ such that $\langle\mathcal{L}_t(\mathbf{u},\mathsf{p}),(\mathbf{v},\mathsf{q})\rangle=\langle\mathbf{f},\mathbf{v}\rangle+(g,\mathsf{q})_{L^2(\Omega)}$ for all $(\mathbf{v},\mathsf{q})\in\mathcal{Y}$, where $t \in [0,1]$. For $t \in [0,1]$, the solvability of this problem is then equivalent to the invertibility of the map $\mathcal{L}_t$. Let $(\mathbf{u}_t,\mathsf{p}_t) \in \mathcal{X}$ be a solution to such a problem. In what follows, we prove
\begin{align}\label{eq:estimate_tproblem}
\|\nabla \mathbf{u}_t\|_{\mathbf{L}^2(\omega,\Omega)}+\|\mathsf{p}_t\|_{L^2(\omega,\Omega)}\lesssim \|\mathbf{f}\|_{\mathbf{H}^{-1}(\omega,\Omega)}+\|g\|_{L^{2}(\omega,\Omega)},
\end{align}
which is equivalent to $\| (\mathbf{u}_t,\mathsf{p}_t) \|_{\mathcal{X}} \lesssim \| \mathcal{L}_t(\mathbf{u},\mathsf{p}) \|_{\mathcal{Y}'}$. An important observation is that $(\mathbf{u}_t,\mathsf{p}_t)$ can be seen as a solution to the following Stokes problem:
Find $(\mathbf{u}_t,\mathsf{p}_t) \in \mathcal{X}$ such that $\langle\mathcal{S}(\mathbf{u}_t,\mathsf{p}_t),(\mathbf{v},\mathsf{q})\rangle =\langle\mathbf{f},\mathbf{v}\rangle+(g,\mathsf{q})_{L^2(\Omega)} - t (\mathbf{u}_t,\mathbf{v})_{\mathbf{L}^2(\Omega)}$ for all $(\mathbf{v},\mathsf{q})\in\mathcal{Y}$. We can thus apply the estimate in \cite[Theorem 17]{MR3906341} to arrive at
\begin{align}
\|\nabla \mathbf{u}_t\|_{\mathbf{L}^2(\omega,\Omega)}+\|\mathsf{p}_t\|_{L^2(\omega,\Omega)}\lesssim \|\mathbf{f}\|_{\mathbf{H}^{-1}(\omega,\Omega)}+\|g\|_{L^{2}(\omega,\Omega)}+\|\mathbf{u}_t\|_{\mathbf{L}^2(\omega,\Omega)}.
\label{eq:Garding}
\end{align}

To obtain \eqref{eq:estimate_tproblem}, we proceed by a contradiction argument. Assuming that \eqref{eq:estimate_tproblem} is false, it is possible to find sequences $\{(\mathbf{u}_k,\mathsf{p}_k)\}_{k\in\mathbb{N}}\subset \mathbf{H}_0^1(\omega,\Omega)\times L^2(\omega,\Omega)/\mathbb{R}$ and $\{(\mathbf{f}_k,g_k)\}_{k\in\mathbb{N}}\subset \mathbf{H}^{-1}(\omega,\Omega)\times L^2(\omega,\Omega)/\mathbb{R}$ such that $(\mathbf{u}_k,\mathsf{p}_k,\mathbf{f}_k,g_k)$ satisfies, for $k \in \mathbb{N}$, $\langle\mathcal{L}_t(\mathbf{u}_k,\mathsf{p}_k),(\mathbf{v},\mathsf{q})\rangle=\langle\mathbf{f}_k,\mathbf{v}\rangle+(g_k,\mathsf{q})_{L^2(\Omega)}$ for all $(\mathbf{v},\mathsf{q})\in\mathcal{Y}$ and $\|\nabla\mathbf{u}_k\|_{\mathbf{L}^2(\omega,\Omega)}+\|\mathsf{p}_k\|_{L^2(\omega,\Omega)}=1$, but $\|\mathbf{f}_k\|_{\mathbf{H}^{-1}(\omega,\Omega)}+\|g_k\|_{L^{2}(\omega,\Omega)} \rightarrow 0$ as $k\uparrow \infty$. Since $\{(\mathbf{u}_k,\mathsf{p}_k)\}_{k\in\mathbb{N}}$ is uniformly bounded in $\mathbf{H}_0^1(\omega,\Omega)\times L^2(\omega,\Omega)/\mathbb{R}$, we deduce the existence of a nonrelabelared subsequence $\{(\mathbf{u}_k,\mathsf{p}_k)\}_{k\in\mathbb{N}}$ such that $\mathbf{u}_k \rightharpoonup \mathbf{u}$ in $\mathbf{H}_0^1(\omega,\Omega)$ and $\mathsf{p}_k \rightharpoonup \mathsf{p}$ in $L^2(\omega,\Omega)/\mathbb{R}$ as $k\uparrow \infty$. The limit $(\mathbf{u},\mathsf{p})$ satisfies $\mathcal{L}_t(\mathbf{u},\mathsf{p})=\mathbf{0}$ in $\mathcal{Y}'$. Consequently, $(\mathbf{u},\mathsf{p}) = (\mathbf{0},0)$. On the other hand, the compact embedding $\mathbf{H}_0^1(\omega,\Omega) \hookrightarrow \mathbf{L}^2(\omega,\Omega)$ \cite[Theorem 4.12]{MR2797702}, \cite[Proposition 2]{MR3998864} shows that $\mathbf{u}_k \to \mathbf{0}$ in $\mathbf{L}^2(\omega,\Omega)$. We can thus invoke the G\r{a}rding-like inequality \eqref{eq:Garding} to deduce that
\begin{eqnarray*}
1 \!=\! \|\nabla\mathbf{u}_k\|_{\mathbf{L}^2(\omega,\Omega)}
\!+\!
\|\mathsf{p}_k\|_{L^2(\omega,\Omega)}
\lesssim \|\mathbf{f}_k\|_{\mathbf{H}^{-1}(\omega,\Omega)}\!+\!\|g_k\|_{L^{2}(\omega,\Omega)}\!+\!\|\mathbf{u}_k\|_{\mathbf{L}^2(\omega,\Omega)}\!\to \!0, \,\, k \uparrow\infty,
\end{eqnarray*}
which is a contradiction. We have thus obtained the desired estimate \eqref{eq:estimate_tproblem}.

\emph{Step 4.} \emph{The method of continuity and the well-posedness of \eqref{eq:brinkman_problem}.} With the estimate \eqref{eq:estimate_tproblem} at hand, we invoke \cite[Theorem 5.2]{MR1814364} and the fact that $\mathcal{L}_0 = \mathcal{S}$ maps $\mathcal{X}$ onto $\mathcal{Y}'$ to deduce that $\mathcal{L}_1 = \mathcal{B}$ maps $\mathcal{X}$ onto $\mathcal{Y}'$ as well, i.e., problem \eqref{eq:brinkman_problem} admits a solution. Since problem \eqref{eq:brinkman_problem} is linear, estimate \eqref{eq:estimate_tproblem} guarantees the uniqueness of solutions. We have thus proved that problem \eqref{eq:brinkman_problem} is well-posed. 
\end{proof}


\section{A Brinkman--Darcy--Forchheimer model}
\label{sec:coupled_problem}
In this section, we show the existence of solutions to the system \eqref{eq:model}. Before doing so, recall that the convective term $(\mathbf{v}\cdot \nabla)\mathbf{v}$ can be rewritten as $\text{div}(\mathbf{v}\otimes\mathbf{v})$ if $\mathbf{v}$ is sufficiently smooth and solenoidal. This property is used to propose a weak formulation for the system \eqref{eq:model}.


\subsection{Weak formulation}
\label{sec:weak_solutions}
For a given weight $\omega$ in the class $A_2$, we define the bilinear forms $a_{0}:\mathbf{H}_0^1(\omega,\Omega)\times \mathbf{H}_0^1(\omega^{-1},\Omega)\to\mathbb{R}$, $a_{1}:
\mathbf{L}^2(\omega,\Omega)\times \mathbf{L}^2(\omega^{-1},\Omega)\to\mathbb{R}$, and $b_{\pm}:\mathbf{H}_0^1(\omega^{\pm 1},\Omega)\times L^2(\omega^{\mp1},\Omega)/\mathbb{R}\to\mathbb{R}$ by
\begin{multline*}\label{eq:form_bilinear}
a_0(\mathbf{w},\mathbf{v}):=\int_\Omega\nabla \mathbf{w}:\nabla\mathbf{v},
\qquad 
a_1(\mathbf{w},\mathbf{v}):=\int_{\Omega}\mathbf{w}\cdot\mathbf{v}, 
\qquad b_{\pm}(\mathbf{v},\mathsf{q}):=-\int_\Omega\mathsf{q}~\text{div }\mathbf{v},
\end{multline*}
respectively. With $a_0$ and $a_1$ at hand, we define 
$
a :\mathbf{H}_0^1(\omega,\Omega)\times \mathbf{H}_0^1(\omega^{-1},\Omega)\to\mathbb{R}
$
by $a(\mathbf{w},\mathbf{v}) := a_0(\mathbf{w},\mathbf{v})+a_1(\mathbf{w},\mathbf{v})$. We now introduce forms associated to the nonlinear terms $(\mathbf{u}\cdot\nabla)\mathbf{u}$ and $|\mathbf{u}|\mathbf{u}$ in \eqref{eq:model}. We define $c:[\mathbf{H}_0^1(\omega,\Omega)]^2 \times \mathbf{H}_0^1(\omega^{-1},\Omega)\to\mathbb{R}$ and $d: [\mathbf{H}_0^1(\omega,\Omega)]^2 \times \mathbf{H}_0^1(\omega^{-1},\Omega)\to\mathbb{R}$, respectively, by
\begin{align*}\label{eq:form_nonlinear}
c(\mathbf{u},\mathbf{w};\mathbf{v}):=-\int_\Omega \mathbf{u}\otimes\mathbf{w}:\nabla\mathbf{v},\qquad d(\mathbf{u},\mathbf{w};\mathbf{v}):=\int_\Omega |\mathbf{u}|\mathbf{w}\cdot\mathbf{v}.
\end{align*}

With these ingredients, let us consider the following weak formulation for the system \eqref{eq:model}: Find $(\mathbf{u},\mathsf{p})\in \mathcal{X}$ such that
\begin{eqnarray}\label{eq:modelweak}
a(\mathbf{u},\mathbf{v}) 
+
b_{-}(\mathbf{v},\mathsf{p})
+
c(\mathbf{u},\mathbf{u};\mathbf{v})
+
d(\mathbf{u},\mathbf{u};\mathbf{v})
=
\langle \mathbf{f},\mathbf{v}\rangle,
\quad
b_{+}(\mathbf{u},\mathsf{q})=0,
\end{eqnarray}
for all $(\mathbf{v},q)\in\mathcal{Y}$. Here, $\langle\cdot,\cdot\rangle$ denotes the duality pairing between $\mathbf{H}^{-1}(\omega,\Omega)$ and $\mathbf{H}_0^1(\omega^{-1},\Omega)$. We recall that the spaces $\mathcal{X}$ and $\mathcal{Y}$ are defined in \eqref{eq:XandY}.

In the following, we will use the following inf-sup condition on weighted spaces, which follows directly from the existence of a right inverse of the divergence; see \cite[Theorem 2.8]{MR3618122}, \cite[Theorem 3.1]{MR2731700}, \cite[Lemma 6.1]{MR4081912}, and \cite[Theorem 1]{MR2548872}:
\begin{eqnarray}\label{eq:infsup}
\|\mathsf{p}\|_{L^2(\omega,\Omega)}\lesssim \sup_{\mathbf{0}\neq \mathbf{v}\in \mathbf{H}_0^1(\omega^{-1},\Omega)}\dfrac{b_{-}(\mathbf{v},\mathsf{p})}{\|\nabla\mathbf{v}\|_{\mathbf{L}^2(\omega^{-1},\Omega)}}\quad \forall \mathsf{p}\in L^2(\omega,\Omega)/\mathbb{R}.
\end{eqnarray}
We will also make use of the following weighted inf-sup condition for $a_0$ \cite{MR3906341}:
\begin{multline}
\label{eq:infsup_a0}
\inf_{\mathbf{0}\neq \mathbf{v}\in \mathbf{H}_0^1(\omega,\Omega)} \sup_{\mathbf{0}\neq \mathbf{w} \in \mathbf{H}_0^1(\omega^{-1},\Omega)}\dfrac{a_0(\mathbf{v},\mathbf{w})}{\|\nabla\mathbf{v}\|_{\mathbf{L}^2(\omega,\Omega)}\|\nabla\mathbf{w}\|_{\mathbf{L}^2(\omega^{-1},\Omega)}}
\\
=
\inf_{\mathbf{0}\neq \mathbf{w}\in \mathbf{H}_0^1(\omega^{-1},\Omega)} \sup_{\mathbf{0}\neq \mathbf{v}\in \mathbf{H}_0^1(\omega,\Omega)}\dfrac{a_0(\mathbf{v},\mathbf{w})}{\|\nabla\mathbf{v}\|_{\mathbf{L}^2(\omega,\Omega)}\|\nabla\mathbf{w}\|_{\mathbf{L}^2(\omega^{-1},\Omega)}}
>0,
\end{multline}
which holds under the further restriction that $\omega \in A_2(\Omega)$.

The following result guarantees the boundedness of the convective and Forchheimer terms on weighted spaces.

\begin{lemma}[boundedness of the convective and Forchheimer terms]\label{eq:lemma01}
If $\omega \in A_2$, $\mathbf{u},\mathbf{w} \in \mathbf{H}_0^1(\omega,\Omega)$, and $\mathbf{v}\in \mathbf{H}_0^1(\omega^{-1},\Omega)$, then
\begin{equation}\label{eq:estimate_NL}
\begin{aligned}
\left|c(\mathbf{u},\mathbf{w};\mathbf{v})\right|
& \leq  
C_{4\to 2}^2 \|\nabla \mathbf{u}\|_{\mathbf{L}^2(\omega,\Omega)} \|\nabla \mathbf{w}\|_{\mathbf{L}^2(\omega,\Omega)}\|\nabla \mathbf{v}\|_{\mathbf{L}^2(\omega^{-1},\Omega)},
\\
\left|d(\mathbf{u},\mathbf{w};\mathbf{v})\right|
& \leq 
C_{4\to 2}^2 C_{2\to 2} \|\nabla \mathbf{u}\|_{\mathbf{L}^2(\omega,\Omega)} \|\nabla \mathbf{w}\|_{\mathbf{L}^2(\omega,\Omega)}\|\nabla \mathbf{v}\|_{\mathbf{L}^2(\omega^{-1},\Omega)}.
\end{aligned}
\end{equation}
Here, $C_{4\to 2}$ and $C_{2\to 2}$ denote the best constants in the embeddings $\mathbf{H}_0^{1}(\omega,\Omega)\hookrightarrow \mathbf{L}^{4}(\omega,\Omega)$ and $\mathbf{H}_0^{1}(\omega^{-1},\Omega)\hookrightarrow \mathbf{L}^{2}(\omega^{-1},\Omega)$, respectively.
\end{lemma}
\begin{proof}
Since we are in two dimensions and $\omega$ and $\omega^{-1}$ belong to $A_2$, \cite[Theorem 1.3]{MR643158} shows that there exists $\zeta>0$ such that $\mathbf{H}_0^{1}(\omega^{\pm 1},\Omega)\hookrightarrow \mathbf{L}^{2k}(\omega^{\pm 1},\Omega)$ for every $k \in [1,2+\zeta]$. Consequently,
\begin{eqnarray*}
|c(\mathbf{u},\mathbf{w};\mathbf{v})|
& \leq & \| \mathbf{u}\|_{\mathbf{L}^4(\omega,\Omega)} \| \mathbf{w}\|_{\mathbf{L}^4(\omega,\Omega)} \|\nabla \mathbf{v}\|_{\mathbf{L}^2(\omega^{-1},\Omega)}
\\
&\leq & C_{4\to 2}^2 
\|\nabla \mathbf{u}\|_{\mathbf{L}^2(\omega,\Omega)} \|\nabla \mathbf{w}\|_{\mathbf{L}^2(\omega,\Omega)}\|\nabla \mathbf{v}\|_{\mathbf{L}^2(\omega^{-1},\Omega)}.
\end{eqnarray*}
Similarly, $|d(\mathbf{u},\mathbf{w};\mathbf{v})| \leq C_{4\to 2}^2 C_{2\to 2} \|\nabla \mathbf{u}\|_{\mathbf{L}^2(\omega,\Omega)} \|\nabla \mathbf{w}\|_{\mathbf{L}^2(\omega,\Omega)}\|\nabla \mathbf{v}\|_{\mathbf{L}^2(\omega^{-1},\Omega)}$.
\end{proof}


\subsection{Well-posedness for small data}
\label{sec:existence_solutions}
We begin this section with a redefinition of the mapping $\mathcal{B}:\mathcal{X} \to \mathcal{Y}'$ and the definition of $\mathcal{N}_{\mathcal{L}}:\mathcal{X}\to \mathcal{Y}'$ and $\mathcal{F}\in\mathcal{Y}'$ as
\begin{eqnarray*}
\langle \mathcal{B}(\mathbf{u},\mathsf{p}),(\mathbf{v},\mathsf{q})\rangle &:=&a(\mathbf{u},\mathbf{v})+b_{-}(\mathbf{v},\mathsf{p})+b_{+}(\mathbf{u},\mathsf{q}),
\\
\langle \mathcal{N}_{\mathcal{L}}(\mathbf{u},\mathsf{p}),(\mathbf{v},\mathsf{q})\rangle &:=&c(\mathbf{u},\mathbf{u};\mathbf{v})+d(\mathbf{u},\mathbf{u};\mathbf{v}),
\end{eqnarray*}
and $\langle\mathcal{F},(\mathbf{v},\mathsf{q})\rangle :=\langle\mathbf{f},\mathbf{v}\rangle$, respectively. Here, $(\mathbf{v},q)\in\mathcal{Y}$. We recall that the spaces $\mathcal{X}$ and $\mathcal{Y}$ are defined in \eqref{eq:XandY}. With this functional framework, we can reformulate the problem \eqref{eq:modelweak} as an equivalent equation in $\mathcal{Y}'$:
$
\mathcal{B}(\mathbf{u},\mathsf{p})+ \mathcal{N}_{\mathcal{L}}(\mathbf{u},\mathsf{p})=\mathcal{F}.
$

The map $\mathcal{B}$ is linear and bounded; see the proof of Theorem \ref{eq:theorem_brinkman} for details. If $\Omega$ is Lipschitz and $\omega$ belongs to $A_2(\Omega)$, then Theorem \ref{eq:theorem_brinkman} guarantees that $\mathcal{B}$ has a bounded inverse. We thus introduce
\begin{eqnarray}\label{eq:T_fixedpoint}
\mathcal{T}:\mathcal{X}\to\mathcal{X},
\qquad
(\mathbf{u},\mathsf{p})=\mathcal{T}(\mathbf{w},\mathsf{r})=\mathcal{B}^{-1}[\mathcal{F}-\mathcal{N}_{\mathcal{L}}(\mathbf{w},\mathsf{r})].
\end{eqnarray}
To prove the existence of a solution to the system \eqref{eq:modelweak}, we use a fixed point argument applied to the map $\mathcal{T}$ and prove that the existence and uniqueness of solutions is guaranteed if the datum $\mathbf{f}$ is sufficiently small. We begin the analysis with a standard contraction argument; see, for instance, \cite[Theorem 3.1]{MR2413675}, \cite[Theorem 5.6]{MR2272870}, and \cite[Proposition 1]{MR3998864}. To present such a result, we define
$\mathsf{A}:=(3 C_e\|\mathcal{B}^{-1}\|)^{-1}>0$ and $\mathfrak{B}_\mathsf{A}:=\{\mathbf{w}\in \mathbf{H}_0^1(\omega,\Omega):\text{div }\mathbf{w}=0,~\|\nabla\mathbf{w}\|_{\mathbf{L}^2(\omega,\Omega)}\leq \mathsf{A}\},$ 
where $C_e:=C_{4\to 2}^2 (1+C_{2\to 2})$ and $\|\mathcal{B}^{-1}\|$ denotes the $\mathcal{Y}' \rightarrow \mathcal{X}$ norm of $\mathcal{B}^{-1}$. Let us introduce, in addition, the map $\mathcal{T}_1:\mathbf{H}_0^1(\omega,\Omega)\to\mathbf{H}_0^1(\omega,\Omega)$ defined as $\mathbf{w} \mapsto P_r \mathcal{T}(\mathbf{w},0)$, where $P_r: \mathcal{X}\to \mathbf{H}_0^1(\omega,\Omega)$ corresponds to the projection onto the velocity component.

\begin{proposition}[$\mathcal{T}_1:\mathfrak{B}_\mathsf{A}\to \mathfrak{B}_\mathsf{A}$ is a contraction]
\label{eq:prop01}
Let $\Omega$ be a bounded Lipschitz domain and $\omega\in A_2(\Omega)$. If the forcing term $\mathbf{f}$ is sufficiently small so that
\begin{equation}\label{eq:assum01}
C_e \|\mathcal{B}^{-1}\|^2 \|\mathbf{f}\|_{\mathbf{H}^{-1}(\omega,\Omega)}< \tfrac{1}{6},
\end{equation}
then $\mathcal{T}_1$ maps $\mathfrak{B}_\mathsf{A}$ to itself and $\mathcal{T}_1$ is a contraction in $\mathfrak{B}_\mathsf{A}$.
\end{proposition}
\begin{proof}
The proof follows the same arguments as in the proof of \cite[Proposition 1]{MR3998864}. For the sake of brevity, we skip the details.
\end{proof}

The existence and uniqueness of solutions for small data is as follows.

\begin{theorem}[well-posedness for small data]
\label{thm:existence_uniqueness}
Let $\Omega\subset\mathbb{R}^{2}$ be a bounded Lipschitz domain and $\omega\in A_2(\Omega)$. If $\mathbf{f}$ is such that \eqref{eq:assum01} holds, then problem \eqref{eq:modelweak} admits a unique solution $(\mathbf{u},\mathsf{p})$. Moreover, we have the bounds
\begin{align}\label{eq:est_nablaU}
\|\nabla\mathbf{u}\|_{\mathbf{L}^2(\omega,\Omega)}
&
\leq 
\tfrac{3}{2}\|\mathcal{B}^{-1}\| \|\mathbf{f}\|_{\mathbf{H}^{-1}(\omega,\Omega)},\\	\label{eq:est_press}
\|\mathsf{p}\|_{L^2(\omega,\Omega)}
&
\lesssim 
\|\nabla \mathbf{u}\|_{\mathbf{L}^2(\omega,\Omega)}+\|\nabla \mathbf{u}\|_{\mathbf{L}^2(\omega,\Omega)}^2+\|\mathbf{f}\|_{\mathbf{H}^{-1}(\omega,\Omega)},
\end{align}
where the hidden constants are independent of $\mathbf{u}$, $\mathsf{p}$, and $\mathbf{f}$.
\end{theorem}
\begin{proof}
We apply proposition \ref{eq:prop01} to deduce the existence of a unique fixed point $\mathbf{u}\in\mathfrak{B}_\mathsf{A}$ of $\mathcal{T}_1$. We now invoke the existence of a right inverse of the divergence operator over $A_2$-weighted spaces \cite[Theorem 3.1]{MR2731700} to obtain the existence and uniqueness of the pressure $\mathsf{p}$. To deduce \eqref{eq:est_nablaU} we use that $\mathbf{u}$ is the unique fixed point of $\mathcal{T}_1$:
\begin{eqnarray*}
\|\nabla \mathbf{u}\|_{\mathbf{L}^2(\omega,\Omega)} \leq \|\mathcal{B}^{-1}\| \|\mathbf{f}\|_{\mathbf{H}^{-1}(\omega,\Omega)}+\tfrac{1}{3}\|\nabla \mathbf{u}\|_{\mathbf{L}^2(\omega,\Omega)}.
\end{eqnarray*}
Finally, to obtain \eqref{eq:est_press} we utilize the weighted inf-sup condition \eqref{eq:infsup}. 
\end{proof}

\section{Finite element approximation: a priori error estimates}
\label{sec:fem}
In this section, we analyze a finite element solution technique that approximates solutions to \eqref{eq:modelweak}. To accomplish this task, we will begin the section by introducing some terminology and a few basic ingredients \cite{MR2373954,CiarletBook,Guermond-Ern}. In what follows, we operate under the assumption that $\Omega$ is a Lipschitz polytope so that it can be triangulated exactly.

\subsection{Triangulation and assumptions}
\label{eq:basic_discrete}

Let $\mathscr{T}= \{ K \}$ be a conforming partition of $\bar{\Omega}$ into closed triangles $K$. Define $h_K = \text{diam}(K)$ and $h_{\mathscr{T}} = \max \{ h_K: K \in \T\}$. We introduce $\mathbb{T}$ as a collection of shape regular conforming triangulations that are refinements of an initial mesh $\mathscr{T}_0$. We denote by $\mathscr{S}$ the set of internal interelement boundaries $\gamma$ of $\T$. For $\gamma \in \mathscr{S}$, we define $h_{\gamma}$ to be the length of $\gamma$. For $K \in \T$, we introduce the set $\mathscr{S}_K$ as the subset of $\mathscr{S}$ containing the sides of $K$. For $\gamma \in \mathscr{S}$, we introduce $\mathcal{N}_{\gamma}$ as the subset of $\T$ containing the two elements that have $\gamma$ as a side. For $K \in \T$, we define
\begin{equation}
\label{eq:patch}
\mathcal{N}_K=  \{K' \in \T: \mathscr{S}_K \cap \mathscr{S}_{K'} \neq \emptyset \},
\quad 
\mathcal{N}_K^*= \{K' \in \T: K \cap {K'} \neq \emptyset \}.
\end{equation}
Below we will indistinctively denote by $\mathcal{N}_K$, $\mathcal{N}_K^*$, and $\mathcal{N}_{\gamma}$ either the sets themselves or the union of the elements that comprise them.

Let $\mathscr{T} \in \mathbb{T}$. We denote by $\mathbf{V}(\T)$ and $\mathcal{P}(\T)$ the finite element spaces that approximate the velocity field and the pressure, respectively. We will work with the following two classical examples:

\begin{enumerate}
\item[(a)] The \emph{mini element}, which is defined by \cite[Section 4.2.4]{Guermond-Ern} 
\begin{eqnarray}\label{ME:vel_space}
\mathbf{V}(\T)&=&\{\mathbf{v}_{\T}\in\mathbf{C}(\overline{\Omega}):\forall K\in\T, \mathbf{v}_{\T}|_{K}\in[\mathbb{W}(K)]^2\} \cap \mathbf{H}_0^1(\Omega),\\ \label{ME:press_space}
\mathcal{P}(\T)&=&\{\mathsf{q}_{\T}\in L_0^2(\Omega)\cap C(\overline{\Omega}):\forall K\in\T, \mathsf{q}_{\T}|_{K}\in\mathbb{P}_1(K) \},
\end{eqnarray}
where $\mathbb{W}(K):=\mathbb{P}_1(K)\oplus \mathbb{B}(K)$ and $\mathbb{B}(K)$ denotes the space spanned by a local bubble function.

\item[(b)] The lowest order \emph{Taylor--Hood pair}, which is defined by \cite[Section 4.2.5]{Guermond-Ern}
\begin{eqnarray}\label{TH:vel_space}
\mathbf{V}(\T)&=&\{\mathbf{v}_{\T}\in\mathbf{C}(\overline{\Omega}):\forall K\in\T, \mathbf{v}_{\T}|_{K}\in[\mathbb{P}_2(K)]^2\} \cap \mathbf{H}_0^1(\Omega),\\ \label{TH:press_space}
\mathcal{P}(\T)&=&\{\mathsf{q}_{\T}\in L_0^2(\Omega)\cap C(\overline{\Omega}):\forall K\in\T, \mathsf{q}_{\T}|_{K}\in\mathbb{P}_1(K) \}.
\end{eqnarray}
\end{enumerate}

We must immediately notice that if $\omega\in A_2$, we have that the previously defined spaces are such that: $\mathbf{V}(\T)\subset \mathbf{W}_0^{1,\infty}(\Omega)\subset \mathbf{H}_0^1(\omega,\Omega)$ and $\mathcal{P}(\T)\subset L^{\infty}(\Omega)\subset L^2(\omega,\Omega)/\mathbb{R}$. In addition, these pairs of spaces satisfy the following compatibility condition \cite[Theorems 6.2 and 6.4]{MR4081912}: There exists $\beta >0$ such that
\begin{eqnarray}\label{eq:infsup_discrete}
\beta \|\mathsf{q}_{\T}\|_{L^2(\omega^{\pm 1},\Omega)} \leq 
\sup_{\mathbf{0} \neq \mathbf{v}_{\T}\in \mathbf{V}(\T)}\dfrac{b_{\mp}(\mathbf{v}_{\T},\mathsf{q}_{\T})}{\|\nabla\mathbf{v}_{\T}\|_{\mathbf{L}^2(\omega^{\mp 1},\Omega)}}
\quad
\forall \mathsf{q}_{\T} \in \mathcal{P}(\T).
\end{eqnarray}

As a final ingredient, if $K^{+}$, $K^{-} \in \T$ are such that $K^{+} \neq K^{-}$ and $\partial K^{+} \cap \partial K^{-} = \gamma$, we define the \emph{jump} or \emph{interelement residual} of a discrete tensor valued function $\mathbf{w}_{\T}$ on an internal side $\gamma \in \mathscr{S}$ by
\begin{equation}
\llbracket \mathbf{w}_{\T} \cdot \boldsymbol{\nu} \rrbracket:= \mathbf{w}_{\T} \cdot \boldsymbol{\nu}^{+} |^{}_{K^{+}} + \mathbf{w}_{\T} \cdot \boldsymbol{\nu}^{-}  |^{}_{K^{-}},
\label{eq:jump}
\end{equation}
where $\boldsymbol{\nu}^{+}$ and $\boldsymbol{\nu}^{-}$ correspond to the unit normals on $\gamma$ pointing towards $K^{+}$ and $K^{-}$, respectively.

\subsection{The scheme} Let $\omega\in A_{2}(\Omega)$ and $\mathbf{f}\in \mathbf{H}^{-1}(\omega,\Omega)$. We introduce the following discrete approximation of \eqref{eq:modelweak}: Find $(\mathbf{u}_{\T},\mathsf{p}_{\T})\in \mathbf{V}(\T)\times \mathcal{P}(\T) $ such that.
\begin{equation}\label{eq:model_discrete}
\begin{array}{rcl}
a(\mathbf{u}_{\T},\mathbf{v}_{\T}) + b_{-}(\mathbf{v}_{\T},\mathsf{p}_{\T}) + c(\mathbf{u}_{\T},\mathbf{u}_{\T};\mathbf{v}_{\T})+ d(\mathbf{u}_{\T},\mathbf{u}_{\T};\mathbf{v}_{\T})
&=&
\langle \mathbf{f},\mathbf{v}_{\T}\rangle,
\\
b_{+}(\mathbf{u}_{\T},\mathsf{q}_{\T})&=&0,
\end{array}
\end{equation}
for all $\mathbf{v}_{\T}\in \mathbf{V}(\T)$ and $\mathsf{q}_{\T}\in \mathcal{P}(\T)$, respectively. 

Let us denote by $\mathcal{B}_\T$ the discrete version of the operator $\mathcal{B}$ induced by the discretization \eqref{eq:model_discrete}. The following results are based on the following assumption:

\begin{assumption}
\label{assuption_brinkman}
The linear and bounded map $\mathcal{B}_\T$ is such that $\mathcal{B}_{\T}^{-1}$ is uniformly bounded over all partitions $\T$.
\end{assumption}

The fact that the operator $\mathcal{B}_{\T}^{-1}$ exists is not an issue. Since we are dealing with finite dimensional spaces, the existence and uniqueness of solutions to the discrete problem \eqref{eq:model_discrete} are guaranteed by the compatibility condition \eqref{eq:infsup_discrete}. The most important point in Assumption \ref{assuption_brinkman} is that $\mathcal{B}_{\T}^{-1}$ satisfies a suitable estimate with respect to the problem data that is uniform with respect to the discretization.

The existence of a unique discrete solution is the content of the following result.

\begin{theorem}[well-posedness for small data]\label{th:existen_discrete}
Let $\Omega \subset \mathbb{R}^2$ be a bounded Lipschitz polytope, and let $\omega \in A_2(\Omega)$. If $\mathbf{f}$ is such that \eqref{eq:assum01} holds with $\mathcal{B}^{-1}$ replaced by $\mathcal{B}^{-1}_{\T}$, then  \eqref{eq:model_discrete} admits a unique solution $(\mathbf{u}_{\T},\mathsf{p}_{\T})\in \mathbf{V}(\T)\times \mathcal{P}(\T)$ satisfying the stability bound
\begin{align}
\label{eq:est_nablaU_d}
\|\nabla\mathbf{u}_\T\|_{\mathbf{L}^2(\omega,\Omega)}&\leq \tfrac{3}{2}\|\mathcal{B}^{-1}_\T\| \|\mathbf{f}\|_{\mathbf{H}^{-1}(\omega,\Omega)},
\\
\label{eq:est_press_d}
\|\mathsf{p}_\T\|_{L^2(\omega,\Omega)}&\lesssim \|\nabla \mathbf{u}_\T\|_{\mathbf{L}^2(\omega,\Omega)}+\|\nabla \mathbf{u}_\T\|_{\mathbf{L}^2(\omega,\Omega)}^2+\|\mathbf{f}\|_{\mathbf{H}^{-1}(\omega,\Omega)},
\end{align}
The hidden constants are independent of $\mathbf{u}_\T$, $\mathsf{p}_{\T}$, and $\mathbf{f}$.
\end{theorem}

\begin{proof}
The proof follows from the arguments developed in the proof of Theorem \ref{thm:existence_uniqueness}. We briefly mention that instead of $\mathcal{B}^{-1}$ we use the fact that $\mathcal{B}^{-1}_{\T}$ is uniformly bounded with respect to discretization.
\end{proof}

To present the auxiliary estimate of Lemma \ref{lemma:aux_result} and the quasi-best approximation result of Theorem \ref{thm:quasi_best}, we will operate under the following assumption: Let $\Omega$ be a convex polytope, $\omega\in A_2(\Omega)$, and $(\mathbf{u},\mathsf{p})\in \mathbf{H}_0^1(\omega,\Omega)\times L^2(\omega,\Omega)/\mathbb{R}$ with $\mathbf{u}$ solenoidal. Let $(\mathscr{B}_{\mathscr{T}}\mathbf{u},\mathscr{B}_{\mathscr{T}}\mathsf{p}) \in \mathbf{V}(\T) \times \mathcal{P}(\T)$ be the \emph{Brinkman projection} of $(\mathbf{u},\mathsf{p})$, i.e., the pair $(\mathscr{B}_{\mathscr{T}}\mathbf{u},\mathscr{B}_{\mathscr{T}}\mathsf{p})$ is such that
\begin{equation}\label{eq:proj_brinkman}
\begin{array}{rcl}
a(\mathscr{B}_{\mathscr{T}}\mathbf{u},\mathbf{v}_{\T})+b_{-}(\mathbf{v}_{\T},\mathscr{B}_{\mathscr{T}}\mathsf{p})&=&a(\mathbf{u},\mathbf{v}_{\T})+b_{-}(\mathbf{v}_{\T},\mathsf{p}),\\
b_{+}(\mathscr{B}_{\mathscr{T}}\mathbf{u},\mathsf{q}_{\T})&=&0,
\end{array}
\end{equation}
for all $\mathbf{v}_{\T}\in\mathbf{V}(\T)$ and $\mathsf{q}_{\T}\in\mathcal{P}(\T)$. Then, we have
\begin{eqnarray}
\label{eq:brinkman_prj_estimate01}
\|\nabla \mathscr{B}_{\mathscr{T}}\mathbf{u}\|_{\mathbf{L}^2(\omega,\Omega)} +\|\mathscr{B}_{\mathscr{T}}\mathsf{p}\|_{L^2(\omega,\Omega)}\lesssim  \|\nabla \mathbf{u}\|_{\mathbf{L}^2(\omega,\Omega)} +\|\mathsf{p}\|_{L^2(\omega,\Omega)},
\end{eqnarray}
where the hidden constant is independent of $(\mathbf{u},\mathsf{p})$, $(\mathscr{B}_{\mathscr{T}}\mathbf{u},\mathscr{B}_{\mathscr{T}}\mathsf{p})$, and $h_{\T}$. When the Brinkman operator in \eqref{eq:proj_brinkman} is replaced by the Stokes operator, the desired estimate can be found in \cite[Theorem 4.1]{MR4081912}. We note that in view of the arguments developed in the proof of \cite[Theorem 4.1]{MR4081912}, the only missing ingredient to obtain \eqref{eq:brinkman_prj_estimate01} is the error estimate \cite[estimate (3.9)]{MR4081912} for a regularized Green's function. If this estimate were available for the Brinkman operator and the finite element pairs considered in \eqref{ME:vel_space}--\eqref{ME:press_space} and \eqref{TH:vel_space}--\eqref{TH:press_space}, then the desired estimate \eqref{eq:brinkman_prj_estimate01} would follow immediately.

\begin{lemma}[auxiliary estimate]
\label{lemma:aux_result}
Let $\Omega \subset \mathbb{R}^2$ be a convex polytope, and let $\omega \in A_2(\Omega)$. If $\mathbf{f}$ is such that \eqref{eq:assum01} holds, then
\begin{multline}
\label{eq:brinkman_prj_estimate02}
\|\nabla(\mathbf{u} - \mathscr{B}_{\mathscr{T}}\mathbf{u})\|_{\mathbf{L}^2(\omega,\Omega)}
+
\|\mathsf{p}-\mathscr{B}_{\T}\mathsf{p}\|_{L^2(\omega,\Omega)}
\\
\lesssim 
\inf_{\mathbf{w}_{\T}\in \mathbf{V}(\T)} 
\|\nabla(\mathbf{u}-\mathbf{w}_{\T})\|_{\mathbf{L}^2(\omega,\Omega)}
+
\inf_{\mathsf{q}_{\T}\in \mathcal{P}(\T)}\|\mathsf{p}-\mathsf{q}_{\T}\|_{L^2(\omega,\Omega)},
\end{multline}
where the hidden constant is independent of $(\mathbf{u},\mathsf{p})$, $(\mathscr{B}_{\mathscr{T}}\mathbf{u},\mathscr{B}_{\mathscr{T}}\mathsf{p})$, and $h_{\T}$.
\end{lemma}
\begin{proof}
The proof is rather standard; it follows, for instance, from the arguments developed in the proof of \cite[Corollary 4.2]{MR4081912}.
\end{proof}

As it is useful for the following analysis, we define, for $\mathbf{v}\in\mathbf{H}_0^1(\omega^{-1},\Omega)$,
\begin{equation}\label{eq:Theta}
\Theta(\mathbf{u},\mathbf{u}_{\T};\mathbf{v}):=c(\mathbf{u},\mathbf{u};\mathbf{v})-c(\mathbf{u}_{\T},\mathbf{u}_{\T};\mathbf{v})+d(\mathbf{u},\mathbf{u};\mathbf{v})-d(\mathbf{u}_{\T},\mathbf{u}_{\T};\mathbf{v}).
\end{equation}
Note that $\Theta(\mathbf{u},\mathbf{u}_{\T};\mathbf{v}) = c(\mathbf{u},\mathbf{e}_{\mathbf{u}};\mathbf{v})+c(\mathbf{e}_{\mathbf{u}},\mathbf{u}_{\T};\mathbf{v})+d(\mathbf{u},\mathbf{u};\mathbf{v})-d(\mathbf{u}_{\T},\mathbf{u}_{\T};\mathbf{v})$. Moreover, for any $\mathbf{v} \in \mathbf{H}_0^1(\omega^{-1},\Omega)$, the following bound can be derived from the estimates in \eqref{eq:estimate_NL} allow us to conclude the following bound:
\begin{multline*}
| \Theta(\mathbf{u},\mathbf{u}_{\T}; \mathbf{v} ) |
\leq
C_{4\to 2}^2 \left(
\|\nabla \mathbf{u}\|_{\mathbf{L}^2(\omega,\Omega)}
+
\|\nabla \mathbf{u}_{\T}\|_{\mathbf{L}^2(\omega,\Omega)}
\right)
\|\nabla \mathbf{e}_{\mathbf{u}} \|_{\mathbf{L}^2(\omega,\Omega)}\|\nabla \mathbf{v}\|_{\mathbf{L}^2(\omega^{-1},\Omega)}
\\
+C_{4\to 2}^2 C_{2\to 2} \left( \|\nabla \mathbf{u}\|_{\mathbf{L}^2(\omega,\Omega)} + \|\nabla \mathbf{u}_{\T}\|_{\mathbf{L}^2(\omega,\Omega)} \right) \|\nabla \mathbf{e}_{\mathbf{u}}\|_{\mathbf{L}^2(\omega,\Omega)}\|\nabla \mathbf{v}\|_{\mathbf{L}^2(\omega^{-1},\Omega)},
\end{multline*}
where we have written the difference $\mathfrak{D}:= d(\mathbf{u},\mathbf{u};\mathbf{v})-d(\mathbf{u}_{\T},\mathbf{u}_{\T};\mathbf{v})$ as $\mathfrak{D} = [d(\mathbf{u},\mathbf{u};\mathbf{v})-d(\mathbf{u}_{\T},\mathbf{u};\mathbf{v})] + d(\mathbf{u}_{\T},\mathbf{u} - \mathbf{u}_{\T};\mathbf{v})$ and used that $| |\mathbf{u}|\mathbf{u} \cdot \mathbf{v} - |\mathbf{u}_{\T}| \mathbf{u} \cdot \mathbf{v}| \leq | \mathbf{u} - \mathbf{u}_{\T} | | \mathbf{u} | | \mathbf{v} |$.

We now obtain the following quasi-best approximation result.

\begin{theorem}[quasi-best approximation result]
\label{thm:quasi_best}
Let $\Omega \subset \mathbb{R}^2$ be a convex polytope, and let $\omega \in A_2(\Omega)$. Let $\mathbf{f}$ be sufficiently small so that \eqref{eq:assum01} holds and that \eqref{eq:assum01} also holds, but with $\mathcal{B}^{-1}$ replaced by $\mathcal{B}^{-1}_{\T}$. Let us also assume that
\begin{align}\label{eq:condition_estimate_apriori}
\dfrac{3C}{2}\|\mathbf{f}\|_{\mathbf{H}^{-1}(\omega,\Omega)}\left[\|\mathcal{B}^{-1}\| +\|\mathcal{B}^{-1}_\T\|\right]
<
\frac{1}{2},
\end{align}
where $C :=C_{\mathscr{B}}C_{4\to 2}^2(1+C_{2\to 2})$. Here, $C_{\mathscr{B}}>0$ denotes the hidden constant in the bound \eqref{eq:brinkman_prj_estimate01} and $C_{4\to 2}$ and $C_{2\to 2}$ denote the best constants in the Sobolev embeddings $\mathbf{H}_0^{1}(\omega,\Omega)\hookrightarrow \mathbf{L}^{4}(\omega,\Omega)$ and $\mathbf{H}_0^{1}(\omega^{-1},\Omega)\hookrightarrow \mathbf{L}^{2}(\omega^{-1},\Omega)$, respectively. Then, we can establish the following quasi-best approximation result:
\begin{multline}\label{eq:quasi_best_estimate}
\|\nabla(\mathbf{u}-\mathbf{u}_{\T})\|_{\mathbf{L}^2(\omega,\Omega)}
+
\|\mathsf{p}-\mathsf{p}_{\T}\|_{L^2(\omega,\Omega)}
\\
\lesssim 
\inf_{\mathbf{w}_{\T}\in \mathbf{V}(\T)} 
\|\nabla(\mathbf{u}\!-\!\mathbf{w}_{\T})\|_{\mathbf{L}^2(\omega,\Omega)}
+ 
\inf_{\mathsf{q}_{\T}\in \mathcal{P}(\T)} 
\|\mathsf{p}\!-\!\mathsf{q}_{\T}\|_{L^2(\omega,\Omega)},
\end{multline}
where the hidden constant may depend on $\mathbf{f}$ and $\mathbf{u}$, but is independent of $h_{\T}$.
\end{theorem}
\begin{proof}
Define $\mathbf{e}_{\mathscr{T}}:=\mathscr{B}_{\mathscr{T}}\mathbf{u}-\mathbf{u}_{\mathscr{T}}$ and $\varepsilon_{\mathscr{T}}:=\mathscr{B}_{\mathscr{T}}\mathsf{p}-\mathsf{p}_{\mathscr{T}}$, where $(\mathscr{B}_{\mathscr{T}}\mathbf{u},\mathscr{B}_{\mathscr{T}}\mathsf{p})$ corresponds to the \emph{Brinkman projection} of $(\mathbf{u},\mathsf{p})$. Invoke the definition of the \emph{Brinkman projection} to infer that
\begin{equation}\label{eq:proj_01}
\begin{array}{rcl}
a(\mathbf{e}_{\mathscr{T}},\mathbf{v}_{\T})+b_{-}(\mathbf{v}_{\T},\varepsilon_{\mathscr{T}})&=&-\Theta(\mathbf{u},\mathbf{u}_{\T};\mathbf{v}_{\mathscr{T}})
\quad
\forall \mathbf{v}_{\T}\in \mathbf{V}(\T),
\\
b_{+}(\mathbf{e}_{\mathscr{T}},\mathsf{q}_{\T})&=&0
\quad
\forall \mathsf{q}_{\T}\in \mathcal{P}(\T).
\end{array}
\end{equation}
We now use the stability bound \eqref{eq:brinkman_prj_estimate01} of the \emph{Brinkman projection} and the previously derived bound for the term $\Theta$ defined in \eqref{eq:Theta} to obtain
\begin{multline}
\|\nabla\mathbf{e}_{\mathscr{T}}\|_{\mathbf{L}^2(\omega,\Omega)}
+
\| \varepsilon_{\mathscr{T}} \|_{L^2(\omega,\Omega)}
\\
\leq
C\left[\|\nabla \mathbf{u}\|_{\mathbf{L}^2(\omega,\Omega)}
+
\|\nabla \mathbf{u}_{\mathscr{T}}\|_{\mathbf{L}^2(\omega,\Omega)}\right]\|\nabla (\mathbf{u}-\mathbf{u}_{\mathscr{T}})\|_{\mathbf{L}^2(\omega,\Omega)},
\label{eq:aux_bound}
\end{multline}
where $C=C_{\mathscr{B}}C_{4\to 2}^2(1+C_{2\to 2})$ and $C_{\mathscr{B}}>0$ corresponds to the hidden constant in the estimate \eqref{eq:brinkman_prj_estimate01}. The error can therefore be controlled with a simple application of the triangle inequality, which shows that
\begin{multline*}
 \|\nabla(\mathbf{u}-\mathbf{u}_{\T})\|_{\mathbf{L}^2(\omega,\Omega)}
+
\|\mathsf{p}-\mathsf{p}_{\T}\|_{L^2(\omega,\Omega)}
\leq
\|\nabla(\mathbf{u}-\mathscr{B}_{\mathscr{T}}\mathbf{u})\|_{\mathbf{L}^2(\omega,\Omega)}
\\
+
\|\mathsf{p}-\mathscr{B}_{\mathscr{T}}\mathsf{p}\|_{L^2(\omega,\Omega)}
+
\| \nabla\mathbf{e}_{\mathscr{T}} \|_{\mathbf{L}^2(\omega,\Omega)}
+
\| \varepsilon_{\mathscr{T}}\|_{L^2(\omega,\Omega)},
\end{multline*}
combined with the estimate of Lemma \ref{lemma:aux_result}, the bound \eqref{eq:aux_bound}, and the estimates \eqref{eq:est_nablaU} and \eqref{eq:est_nablaU_d}. In fact, we have
\begin{multline*}
 \|\nabla(\mathbf{u}-\mathbf{u}_{\T})\|_{\mathbf{L}^2(\omega,\Omega)} 
+
\|\mathsf{p}-\mathsf{p}_{\T}\|_{L^2(\omega,\Omega)}
\lesssim
\inf_{\mathbf{w}_{\T}\in \mathbf{V}(\T)}
\|\nabla(\mathbf{u}-\mathbf{w}_{\T})\|_{\mathbf{L}^2(\omega,\Omega)}
\\+\inf_{\mathsf{q}_{\T}\in \mathcal{P}(\T)} 
\|\mathsf{p}-\mathsf{q}_{\T}\|_{L^2(\omega,\Omega)}
+
\dfrac{3C}{2}\|\mathbf{f}\|_{\mathbf{H}^{-1}(\omega,\Omega)}\left[\|\mathcal{B}^{-1}\| +\|\mathcal{B}^{-1}_\T\|\right]
\|\nabla (\mathbf{u}-\mathbf{u}_{\mathscr{T}})\|_{\mathbf{L}^2(\omega,\Omega)}.
\end{multline*}
To conclude the proof, we use the assumption \eqref{eq:condition_estimate_apriori} so that the term involving $\|\nabla (\mathbf{u}-\mathbf{u}_{\mathscr{T}})\|_{\mathbf{L}^2(\omega,\Omega)}$ that appears on the right-hand side of the previous estimate can be absorbed into the left.
\end{proof}



\section{Finite element approximation: a posteriori error estimates}
\label{sec:a_posteriori_anal}
In this section, we develop an a posteriori error estimator for problem \eqref{eq:model_discrete}. To do so, we assume that the external density force $\mathbf{f}$ has a certain structure, i.e., $\mathbf{f}:= \mathbf{F}\delta_{\mathbf{z}}$, where $\mathbf{F}\in\mathbb{R}^2$.
It is therefore appropriate to consider $\alpha \in (0,2)$ and the weight $\mathsf{d}_{\mathbf{z}}^{\alpha}$ as defined in \eqref{eq:weight_A2}. We note that $\mathsf{d}_{\mathbf{z}}^{\pm \alpha} \in A_2$, $\mathsf{d}_{\mathbf{z}}^{\alpha} \in A_2(\Omega)$, and $\delta_{\mathbf{z}} \in H_0^1(\mathsf{d}_{\mathbf{z}}^{-\alpha},\Omega)'$; see \cite[Lemma 7.1.3]{MR1469972} and \cite[Remark 21.19]{MR2305115}.

In what follows, we assume that $\mathbf{F} \in \mathbb{R}^2$ is such that \eqref{eq:assum01} holds. Moreover, we assume that $\mathbf{F} \in \mathbb{R}^2$ is such that \eqref{eq:assum01} holds with $\mathcal{B}^{-1}$ replaced by $\mathcal{B}_{\T}^{-1}$. We note that under these conditions, problems \eqref{eq:modelweak} and \eqref{eq:model_discrete} are well-posed; see Theorems \ref{thm:existence_uniqueness} and \ref{th:existen_discrete}.

Let us begin our studies by redefining the spaces $\mathcal{X}$ and $\mathcal{Y}$ as follows: $\mathcal{X}=\mathbf{H}_0^1(\mathsf{d}_{\mathbf{z}}^{\alpha},\Omega)\times L^2(\mathsf{d}_{\mathbf{z}}^{\alpha},\Omega)/\mathbb{R}$ and $\mathcal{Y}=\mathbf{H}_0^1(\mathsf{d}_{\mathbf{z}}^{-\alpha},\Omega)\times L^2(\mathsf{d}_{\mathbf{z}}^{-\alpha},\Omega)/\mathbb{R}$. We define the velocity error $\mathbf{e}_\mathbf{u}$ and the pressure error $e_{\mathsf{p}}$ as
\begin{equation}
\mathbf{e}_\mathbf{u}:=\mathbf{u}-\mathbf{u}_{\T}\in \mathbf{H}_0^1(\mathsf{d}_{\mathbf{z}}^{\alpha},\Omega),\qquad e_{\mathsf{p}}:=\mathsf{p}-\mathsf{p}_{\T}\in L^2(\mathsf{d}_{\mathbf{z}}^{\alpha},\Omega)/\mathbb{R}.
\end{equation}

\subsection{Ritz projection} 
As an instrumental step to perform a global reliability analysis, we study a suitable Ritz projection $(\mathbf{\Phi},\psi)$ of the residuals. $(\mathbf{\Phi},\psi)$ is defined as the solution to the system: Find $(\mathbf{\Phi},\psi)\in \mathcal{X}$ such that
\begin{equation}\label{eq:ritz}
\begin{array}{rcl}
(\nabla\mathbf{\Phi},\nabla \mathbf{v})_{\mathbf{L}^2(\Omega)}&=&a(\mathbf{e}_{\mathbf{u}},\mathbf{v})\!+\!b_{-}(\mathbf{v},e_{\mathsf{p}})\!+\!\Theta(\mathbf{u},\mathbf{u}_{\T};\mathbf{v})
\quad
\forall \mathbf{v} \in \mathbf{H}_0^1(\mathsf{d}_{\mathbf{z}}^{-\alpha},\Omega),
\\
(\psi,\mathsf{q})_{L^2(\Omega)}&=&b_{+}(\mathbf{e}_{\mathbf{u}},\mathsf{q})
\quad
\forall \mathsf{q} \in L^2(\mathsf{d}_{\mathbf{z}}^{-\alpha},\Omega)/\mathbb{R}.
\end{array}
\end{equation}
Here, $\Theta(\mathbf{u},\mathbf{u}_{\T};\mathbf{v})$ is defined as in \eqref{eq:Theta}.

As proved in the next result, the problem \eqref{eq:ritz} is well-posed.
%

\begin{theorem}[existence and uniqueness of the Ritz projection] 
There exists a unique solution $(\mathbf{\Phi},\psi)\in \mathcal{X}$ to the system \eqref{eq:ritz}. Moreover, the following estimate holds:
\begin{multline}\label{eq:ritz_estimate}
\|\nabla\mathbf{\Phi}\|_{\mathbf{L}^2(\mathsf{d}_{\mathbf{z}}^{\alpha},\Omega)}+\|\psi\|_{L^2(\mathsf{d}_{\mathbf{z}}^{\alpha},\Omega)}
\lesssim \|\nabla\mathbf{e}_{\mathbf{u}}\|_{\mathbf{L}^2(\mathsf{d}_{\mathbf{z}}^{\alpha},\Omega)}+\|\mathbf{e}_{\mathsf{p}}\|_{L^2(\mathsf{d}_{\mathbf{z}}^{\alpha},\Omega)}
\\
+\|\nabla\mathbf{e}_{\mathbf{u}}\|_{\mathbf{L}^2(\mathsf{d}_{\mathbf{z}}^{\alpha},\Omega)}\left(\|\nabla\mathbf{u}\|_{\mathbf{L}^2(\mathsf{d}_{\mathbf{z}}^{\alpha},\Omega)}+\|\nabla\mathbf{u}_{\T}\|_{\mathbf{L}^2(\mathsf{d}_{\mathbf{z}}^{\alpha},\Omega)}\right),
\end{multline}
with a hidden constant independent of $(\mathbf{\Phi},\psi)$, $(\mathbf{u},\mathsf{p})$, and $(\mathbf{u}_{\T},\mathsf{p}_{\T})$. 
\end{theorem}
\begin{proof}
We start with the introduction of the linear functional $\mathfrak{G}$ as follows:
\begin{eqnarray}\label{eq:operator_A}
\mathfrak{G}:\mathbf{H}_0^1(\mathsf{d}_{\mathbf{z}}^{-\alpha},\Omega) \to\mathbb{R},
\qquad
\mathfrak{G}(\mathbf{v}):=a(\mathbf{e}_{\mathbf{u}},\mathbf{v})+b_{-}(\mathbf{v},e_{\mathsf{p}})+\Theta(\mathbf{u},\mathbf{u}_{\T};\mathbf{v}).
\end{eqnarray}
Let us show that $\mathfrak{G}$ belongs to $\mathbf{H}_0^1(\mathsf{d}_{\mathbf{z}}^{-\alpha},\Omega)'$. To accomplish this task, we first control the nonlinear term $\Theta(\mathbf{u},\mathbf{u}_{\T};\cdot)$ defined in \eqref{eq:Theta}. Owing to the estimates of Lemma \ref{eq:lemma01}, we obtain the bound
\begin{equation}
\| \Theta(\mathbf{u},\mathbf{u}_{\T};\cdot)\|_{\mathbf{H}_0^1(\mathsf{d}_{\mathbf{z}}^{-\alpha},\Omega)'} \leq C_{4\to 2}^2(1+C_{2\to 2}) \|\nabla\mathbf{e}_{\mathbf{u}}\|_{\mathbf{L}^2(\mathsf{d}_{\mathbf{z}}^{\alpha},\Omega)}\Lambda(\mathbf{u},\mathbf{u}_{\T}),
\label{eq:estimate_Theta}
 \end{equation}
where $\Lambda(\mathbf{u},\mathbf{u}_{\T}):= \|\nabla\mathbf{u}\|_{\mathbf{L}^2(\mathsf{d}_{\mathbf{z}}^{\alpha},\Omega)}+ \|\nabla\mathbf{u}_{\T}\|_{\mathbf{L}^2(\mathsf{d}_{\mathbf{z}}^{\alpha},\Omega)}$. Consequently,
\begin{multline*}
\|\mathfrak{G}\|_{\mathbf{H}_0^1(\mathsf{d}_{\mathbf{z}}^{-\alpha},\Omega)'}
\leq 
(1+C_{2\to 2})\| \nabla \mathbf{e}_{\mathbf{u}}\|_{\mathbf{L}^2(\mathsf{d}_{\mathbf{z}}^{\alpha},\Omega)} 
+ 
\|e_{\mathsf{p}}\|_{L^2(\mathsf{d}_{\mathbf{z}}^{\alpha},\Omega)}
\\
+
C_{4\to 2}^2(1+C_{2\to 2}) \|\nabla\mathbf{e}_{\mathbf{u}}\|_{\mathbf{L}^2(\mathsf{d}_{\mathbf{z}}^{\alpha},\Omega)} \Lambda(\mathbf{u},\mathbf{u}_{\T}) =: \Gamma(\mathbf{u},\mathbf{u}_{\T},p).
\end{multline*}
Since $\mathsf{d}_{\mathbf{z}}^{\alpha}\in A_2(\Omega)$ and $\mathfrak{G}\in (\mathbf{H}_0^1(\mathsf{d}_{\mathbf{z}}^{-\alpha},\Omega))'$, we can use the results of \cite{MR3906341} to derive the existence of a unique $\mathbf{\Phi}\in \mathbf{H}_0^1(\mathsf{d}_{\mathbf{z}}^{\alpha},\Omega)$ that satisfies the first equation of the problem \eqref{eq:ritz} and the estimate
\begin{equation}\label{eq:ritz_help01}
\|\nabla\mathbf{\Phi}\|_{\mathbf{L}^2(\mathsf{d}_{\mathbf{z}}^{\alpha},\Omega)} \lesssim \Gamma(\mathbf{u},\mathbf{u}_{\T},p).
\end{equation}
Finally, since $\mathbf{e}_{\mathbf{u}}\in \mathbf{H}_0^1(\mathsf{d}_{\mathbf{z}}^{\alpha},\Omega)$, $b_{+}(\mathbf{e}_{\mathbf{u}},\cdot)$ defines a linear and continuous functional in the space $L^2(\mathsf{d}_{\mathbf{z}}^{-\alpha},\Omega)/\mathbb{R}$. As a consequence, we deduce the existence and uniqueness of $\psi\in L^2(\mathsf{d}_{\mathbf{z}}^{\alpha},\Omega)/\mathbb{R}$ satisfying the second equation in \eqref{eq:ritz} and the estimate
\begin{eqnarray}\label{eq:ritz_help02}
\|\psi\|_{L^2(\mathsf{d}_{\mathbf{z}}^{\alpha},\Omega)}\lesssim \|\text{div }\mathbf{e}_{\mathbf{u}}\|_{\mathbf{L}^2(\mathsf{d}_{\mathbf{z}}^{\alpha},\Omega)}.
\end{eqnarray}
Thus, we can derive the desired estimate \eqref{eq:ritz_estimate} by collecting the bounds \eqref{eq:ritz_help01} and \eqref{eq:ritz_help02}. This concludes the proof.
\end{proof}

\subsection{An upper bound for the error}
In this section, we derive an upper bound for the energy norm of the error in terms of the energy norm of the Ritz projection.

Let us begin the analysis by introducing the map $\mathfrak{F}:\mathbf{H}_0^1(\mathsf{d}_{\mathbf{z}}^{-\alpha},\Omega)\to\mathbb{R}$ as
\begin{eqnarray}\label{eq:operat_F}
\mathfrak{F}(\mathbf{v}):=
(\nabla\mathbf{\Phi},\nabla \mathbf{v})_{\mathbf{L}^2(\Omega)}
-
\Theta(\mathbf{u},\mathbf{u}_{\T};\mathbf{v}),
\end{eqnarray}
where $\Theta(\mathbf{u},\mathbf{u}_{\T};\mathbf{v})$ is defined in \eqref{eq:Theta}. It is clear that the map $\mathfrak{F}$ is linear; here, $\mathbf{u}$ and $\mathbf{u}_{\T}$ are given. In addition, in view of \eqref{eq:estimate_Theta}, $\mathfrak{F}$ satisfies the estimate
\begin{equation}\label{eq:est_F}
\|\mathfrak{F}\|_{\mathbf{H}_0^1(\mathsf{d}_{\mathbf{z}}^{-\alpha},\Omega)'} \leq  \|\nabla \mathbf{\Phi}\|_{\mathbf{L}^2(\mathsf{d}_{\mathbf{z}}^{\alpha},\Omega)}
+
C_{4\to 2}^2(1+C_{2\to 2}) \|\nabla\mathbf{e}_{\mathbf{u}}\|_{\mathbf{L}^2(\mathsf{d}_{\mathbf{z}}^{\alpha},\Omega)} \Lambda(\mathbf{u},\mathbf{u}_{\T}),
\end{equation}
where $\Lambda(\mathbf{u},\mathbf{u}_{\T}):= \|\nabla\mathbf{u}\|_{\mathbf{L}^2(\mathsf{d}_{\mathbf{z}}^{\alpha},\Omega)}+ \|\nabla\mathbf{u}_{\T}\|_{\mathbf{L}^2(\mathsf{d}_{\mathbf{z}}^{\alpha},\Omega)}$.

Having introduced the linear map $\mathfrak{F}$, we note that, given the equations in problem \eqref{eq:ritz}, the pair $(\mathbf{e}_{\mathbf{u}},e_{\mathsf{p}})$ can be considered as a solution to the following problem: Find $(\mathbf{e}_{\mathbf{u}},e_{\mathsf{p}}) \in \mathcal{X}$ such that, for every $\mathbf{v} \in \mathbf{H}_0^1(\mathsf{d}_{\mathbf{z}}^{-\alpha},\Omega)$ and $\mathsf{q} \in L^2(\mathsf{d}_{\mathbf{z}}^{-\alpha},\Omega)/\mathbb{R}$,
\begin{equation}\label{eq:probl_aux}
a(\mathbf{e}_{\mathbf{u}},\mathbf{v})+b_{-}(\mathbf{v},e_{\mathsf{p}})=\mathfrak{F}(\mathbf{v}),
\qquad b_{+}(\mathbf{e}_{\mathbf{u}},\mathsf{q})=(\psi,\mathsf{q})_{L^2(\Omega)}.
\end{equation}

With all these ingredients at hand, we present the following result.

\begin{proposition}[upper bound for the error] 
Let $\mathbf{F} \in \mathbb{R}^2$ be such that 
\begin{eqnarray}\label{eq:small_cond}
1-C_{\mathcal{B}}C_{4\to 2}^2(1+C_{2\to 2})(\|\nabla\mathbf{u}\|_{\mathbf{L}^2(\mathsf{d}_{\mathbf{z}}^{\alpha},\Omega)}+\|\nabla\mathbf{u}_{\T}\|_{\mathbf{L}^2(\mathsf{d}_{\mathbf{z}}^{\alpha},\Omega)})\geq \lambda > 0,
\end{eqnarray}
where $\lambda < 1$. Then, the following upper bound for the error $(\mathbf{e}_{\mathbf{u}},e_{\mathsf{p}})$ holds:
\begin{equation}\label{eq:upper_bound_ritz}
\|\nabla\mathbf{e}_{\mathbf{u}}\|_{\mathbf{L}^2(\mathsf{d}_{\mathbf{z}}^{\alpha},\Omega)}+\|e_{\mathsf{p}}\|_{L^2(\mathsf{d}_{\mathbf{z}}^{\alpha},\Omega)}\lesssim \|\nabla\mathbf{\Phi}\|_{\mathbf{L}^2(\mathsf{d}_{\mathbf{z}}^{\alpha},\Omega)}+\|\psi\|_{L^2(\mathsf{d}_{\mathbf{z}}^{\alpha},\Omega)},
\end{equation}
with a hidden constant independent of $(\mathbf{u},\mathsf{p})$, $(\mathbf{u}_\T,\mathsf{p}_\T)$, and $(\mathbf{\Phi},\psi)$.
\end{proposition}
\begin{proof}
Invoke the estimate in Theorem \ref{eq:theorem_brinkman} and the bound \eqref{eq:est_F} to deduce that
\begin{multline*}
\|\nabla\mathbf{e}_{\mathbf{u}}\|_{\mathbf{L}^2(\mathsf{d}_{\mathbf{z}}^{\alpha},\Omega)}+\|e_{\mathsf{p}}\|_{L^2(\mathsf{d}_{\mathbf{z}}^{\alpha},\Omega)}\leq C_{\mathcal{B}} 
\left(
\|\mathfrak{F}\|_{\mathbf{H}_0^1(\mathsf{d}_{\mathbf{z}}^{-\alpha},\Omega)'} + \|\psi\|_{L^2(\mathsf{d}_{\mathbf{z}}^{\alpha},\Omega)}
\right)
\\
\leq C_{\mathcal{B}}
\left(
\|\nabla \mathbf{\Phi}\|_{\mathbf{L}^2(\mathsf{d}_{\mathbf{z}}^{\alpha},\Omega)} 
+
C_{4\to 2}^2(1+C_{2\to 2}) \|\nabla\mathbf{e}_{\mathbf{u}}\|_{\mathbf{L}^2(\mathsf{d}_{\mathbf{z}}^{\alpha},\Omega)} \Lambda(\mathbf{u},\mathbf{u}_{\T})
+
\|\psi\|_{L^2(\mathsf{d}_{\mathbf{z}}^{\alpha},\Omega)}
\right).
\end{multline*}
With this bound at hand, we invoke the smallness assumption \eqref{eq:small_cond} to conclude.
\end{proof}

\subsection{A residual-type error estimator}
In what follows, we introduce an a posteriori error estimator for the finite element approximation \eqref{eq:model_discrete} of problem \eqref{eq:modelweak} based on the discrete pairs $(\mathbf{V}(\T),
\mathcal{P}(\T))$ defined in \eqref{ME:vel_space}--\eqref{ME:press_space} or \eqref{TH:vel_space}--\eqref{TH:press_space}. To present it, we first introduce, for $K \in \T$, the local distance 
\begin{equation}
 \label{eq:DK}
 D_K:=\max_{\mathbf{x} \in K}|\mathbf{x}-\mathbf{z}|.
\end{equation}
We thus define, for $K\in \T$ and $\gamma\in\mathscr{S}$, the \emph{element residual} $\mathcal{R}_K$ and the \emph{interelement residual} $\mathcal{J}_{\gamma}$ as
\begin{eqnarray}\label{eq:element_residual}
\mathcal{R}_K&:=& (\Delta \mathbf{u}_{\T}
-
\mathbf{u}_{\T}
-
(\mathbf{u}_{\T}\cdot \nabla)\mathbf{u}_{\T}
-
\mathbf{u}_{\T}\text{div }\mathbf{u}_{\T}
-
|\mathbf{u}_{\T}|\mathbf{u}_{\T}
-
\nabla\mathsf{p}_{\T})|_K. 
\\ 
\label{eq:interelement_residual}
\mathcal{J}_{\gamma}
&:=& 
\llbracket(\nabla \mathbf{u}_{\T}-\mathsf{p}_{\T}\mathbf{I})\cdot \boldsymbol{\nu}\rrbracket,
\end{eqnarray}
where $(\mathbf{u}_{\T},\mathsf{p}_{\T})$ denotes the solution to the discrete problem \eqref{eq:model_discrete} and $\mathbf{I}\in\mathbb{R}^{2\times 2}$ denotes the identity matrix. The jump $\llbracket(\nabla \mathbf{u}_{\T}-\mathsf{p}_{\T}\mathbf{I})\cdot \boldsymbol{\nu}\rrbracket$ of the discrete tensor valued function $\nabla \mathbf{u}_{\T}-\mathsf{p}_{\T}\mathbf{I}$ is defined as in \eqref{eq:jump}.
For $K\in\T$ and $\alpha\in(0,2)$, we define the \emph{element error indicator} 
\begin{multline}\label{eq:indicator_e}
\displaystyle\mathcal{E}_{\alpha}(\mathbf{u}_{\T},\mathsf{p}_{\T};K):= \left( 
h_K^2 D_K^{\alpha} \|\mathcal{R}_K\|_{\mathbf{L}^2(K)}^2
+
\|\text{div }\mathbf{u}_{\T}\|_{L^2(\mathsf{d}_{\mathbf{z}}^{\alpha},K)}^2
\right.
\\
\left.
+
h_K D_K^{\alpha} \|\mathcal{J}_{\gamma} \|_{\mathbf{L}^2(\partial K\setminus\partial \Omega)}^2
+
h_K^{\alpha}|\mathbf{F}|^2 \# (\{\mathbf{z}\} \cap  K)
\right)^{\frac{1}{2}}.
\end{multline}
By $\#(E)$ we understand the cardinality of the set $E$. We note that in \eqref{eq:indicator_e} we consider our elements $K$ to be closed sets. We define the a posteriori \emph{error estimator} as
\begin{eqnarray}\label{eq:estimator_e}
\mathcal{E}_{\alpha}(\mathbf{u}_{\T},\mathsf{p}_{\T};\T):= \left(\sum_{K\in\T}\mathcal{E}_{\alpha}^2(\mathbf{u}_{\T},\mathsf{p}_{\T};K)\right)^{\frac{1}{2}}.
\end{eqnarray}

\subsection{A quasi-interpolation operator}
\label{sec:quasi_interpolation}
To derive a posteriori error estimates, we will use the quasi-interpolation operator $\Pi_{\T}:\mathbf{L}^1(\Omega)\to \mathbf{V}(\T)$ introduced in \cite{MR3439216}. In particular, the following properties of $\Pi_{\T}$ will be relevant to our analysis \cite{MR3892359,MR3439216}.

\begin{proposition}[stability and interpolation estimates for the operator $\Pi_{\T}$]
Let $K\in\T$ and $\alpha\in(-2,2)$. Then, for every $\mathbf{v}\in \mathbf{H}^1(\mathsf{d}_{\mathbf{z}}^{\pm\alpha},\mathcal{N}_K^*)$, we have
\begin{align}
\label{eq:Pi_01}
\|\nabla \Pi_{\T}\mathbf{v}\|_{\mathbf{L}^2(\mathsf{d}_{\mathbf{z}}^{\pm\alpha},K)} &\lesssim  \|\nabla\mathbf{v}\|_{\mathbf{L}^2(\mathsf{d}_\mathbf{z}^{\pm\alpha},\mathcal{N}_K^*)},
\\
\label{eq:Pi_02}
\|\mathbf{v}- \Pi_{\T}\mathbf{v}\|_{\mathbf{L}^2(\mathsf{d}_\mathbf{z}^{\pm\alpha},K)} &\lesssim  h_{K} \|\nabla\mathbf{v}\|_{\mathbf{L}^2(\mathsf{d}_\mathbf{z}^{\pm\alpha},\mathcal{N}_K^*)}.
\end{align}
Moreover, if $\alpha\in(0,2)$, then
\begin{align}\label{eq:Pi_03}
\|\mathbf{v}- \Pi_{\T}\mathbf{v}\|_{\mathbf{L}^2(K)}&\lesssim h_{K}D_K^{\frac{\alpha}{2}} \|\nabla\mathbf{v}\|_{\mathbf{L}^2(\mathsf{d}_\mathbf{z}^{-\alpha},\mathcal{N}_K^*)}.
\end{align}
The hidden constants in both estimates are independent of $\mathbf{v}$, $K$, and $\mathscr{T}$.
\end{proposition}
\begin{proof}
See \cite[Proposition 4]{MR3892359}.
\end{proof}

\begin{proposition}[trace interpolation estimate for the operator $\Pi_{\T}$]
Let $K\in \T$, $\gamma \subset \mathscr{S}_K$, $\alpha\in(0,2)$, and $\mathbf{v}\in \mathbf{H}^{1}(\mathsf{d}_{\mathbf{z}}^{-\alpha},\mathcal{N}_K^*)$. Then,
\begin{equation}\label{eq:Pi_04}
\|\mathbf{v}-\Pi_{\T}\mathbf{v}\|_{\mathbf{L}^2(\gamma)}\lesssim h_K^{\frac{1}{2}}D_{K}^{\frac{\alpha}{2}}\|\nabla \mathbf{v}\|_{\mathbf{L}^2(\mathsf{d}_\mathbf{z}^{-\alpha},\mathcal{N}_K^*)},
\end{equation}
where the hidden constant is independent of $\mathbf{v}$, $K$, and the mesh $\T$.
\end{proposition}
\begin{proof}
See \cite[Proposition 5]{MR3892359}.
\end{proof}


\subsection{Reliablity} 
Let us now derive a global reliability bound for the estimator $\mathcal{E}_{\alpha}$ defined in \eqref{eq:estimator_e}.

\begin{theorem}[global reliability]
Let $\alpha \in (0,2)$, let the pair $(\mathbf{u},\mathsf{p})$ $\in$ $\mathbf{H}_0^1(\mathsf{d}_\mathbf{z}^{\alpha},\Omega)\times L^2(\mathsf{d}_\mathbf{z}^{\alpha},\Omega)/\mathbb{R}$ be the solution to \eqref{eq:modelweak}, and let $(\mathbf{u}_{\T},\mathsf{p}_{\T})\in \mathbf{V}(\T)\times \mathcal{P}(\T)$ be the solution to the discrete system \eqref{eq:model_discrete}. If $\mathbf{F}$ is such that \eqref{eq:small_cond} holds, then
\begin{eqnarray}
\|\nabla\mathbf{e}_{\mathbf{u}}\|_{\mathbf{L}^2(\mathsf{d}_{\mathbf{z}}^{\alpha},\Omega)}+\|e_{\mathsf{p}}\|_{L^2(\mathsf{d}_{\mathbf{z}}^{\alpha},\Omega)}\lesssim \mathcal{E}_{\alpha}(\mathbf{u}_{\T},\mathsf{p}_{\T};\T),
\label{eq:reliability_estimate} 
\end{eqnarray}
with a hidden constant independent of $(\mathbf{u},\mathsf{p})$ and $(\mathbf{u}_{\T},\mathsf{p}_{\T})$, the size of the elements in the mesh $\T$, and $\# \T$.
\label{thm:globa_reliability}
\end{theorem}
\begin{proof}
To provide the computable upper bound \eqref{eq:reliability_estimate}, we will utilize the fact that the energy norm of the error can be bounded in terms of the energy norm of the Ritz projection and proceed in three steps.

\emph{Step 1:} Let $\mathbf{v}\in \mathbf{H}_0^1({\mathsf{d}_{\mathbf{z}}^{-\alpha}},\Omega)$ be arbitrary. We utilize the first equation of problems \eqref{eq:ritz} and \eqref{eq:modelweak} to conclude that
\begin{multline}\label{eq:reliability_aux01}
(\nabla\mathbf{\Phi},\nabla \mathbf{v})_{\mathbf{L}^2(\Omega)}
=
\langle\mathbf{F}\delta_{\mathbf{z}},\mathbf{v}\rangle
-
\sum_{K\in\T}\int_K
\left(
\nabla \mathbf{u}_{\T}:\nabla\mathbf{v}
+
\mathbf{u}_{\T}\cdot\mathbf{v}
\right.
\\
\left.
-
\mathbf{u}_{\T} \otimes\mathbf{u}_{\T}:\nabla \mathbf{v}
+
|\mathbf{u}_{\T}|\mathbf{u}_{\T}\cdot\mathbf{v}
-
\mathsf{p}_{\T}\text{div }\mathbf{v}
\right).
\end{multline}
Applying a standard integration by parts argument, on the basis of the fact that, for $\gamma \in \mathscr{S}$, $(\llbracket (\mathbf{u}_{\T}\otimes \mathbf{u}_{\T})\boldsymbol{\nu}\rrbracket,\mathbf{v})_{\mathbf{L}^2(\gamma)}=0$, yields the identity
\begin{align}\label{eq:reliability_aux02}
(\nabla\mathbf{\Phi},\nabla \mathbf{v})_{\mathbf{L}^2(\Omega)}=\langle\mathbf{F}\delta_{\mathbf{z}},\mathbf{v}\rangle + \sum_{\gamma\in \mathscr{S}}\int_{\gamma}\mathcal{J}_{\gamma}\cdot \mathbf{v}
+\sum_{K\in\T}\int_{K} \mathcal{R}_{K}\cdot\mathbf{v}.
\end{align}
We recall that the element residual $\mathcal{R}_{K}$ and the interelement residual $\mathcal{J}_{\gamma}$ are defined as in \eqref{eq:element_residual} and \eqref{eq:interelement_residual}, respectively.

Let us now observe that, for every $(\mathbf{v}_{\T},\mathsf{q}_{\T}) \in \mathbf{V}(\T) \times \mathcal{P}(\T)$, we have
\[
 \langle\mathbf{F}\delta_{\mathbf{z}},\mathbf{v}_{\T}\rangle - a(\mathbf{u}_{\T},\mathbf{v}_{\T}) - b_{-}(\mathbf{v}_{\T},\mathsf{p}_{\T})  - c(\mathbf{u}_{\T},\mathbf{u}_{\T};\mathbf{v}_{\T}) - d(\mathbf{u}_{\T},\mathbf{u}_{\T};\mathbf{v}_{\T})= 0,
\]
which follows from rewriting the first equation in \eqref{eq:model_discrete}. Set $\mathbf{v}_{\T}=\Pi_{\T}\mathbf{v}$ into the previous relation, apply an integration by parts formula, and utilize the relation \eqref{eq:reliability_aux02} to arrive at
\begin{multline}\label{eq:reliability_aux03}
(\nabla\mathbf{\Phi},\nabla \mathbf{v})_{\mathbf{L}^2(\Omega)} =
\langle\mathbf{F}\delta_{\mathbf{z}},\mathbf{v}-\Pi_{\T}\mathbf{v}\rangle
+
\sum_{K\in\T}\int_{K} \mathcal{R}_K\cdot(\mathbf{v}-\Pi_{\T}\mathbf{v})
\\
+ \sum_{\gamma\in \mathscr{S}}\int_{\gamma}\mathcal{J}_{\gamma}\cdot (\mathbf{v}-\Pi_{\T}\mathbf{v}) =: \text{I}+\text{II}+\text{III}.
\end{multline}

In what follows, we control the terms $\text{I}$, $\text{II}$, and $\text{III}$ following the arguments developed in \cite{MR4117306}. Let us begin with the control of the term \text{I}. To accomplish this task, we invoke the local bound of \cite[Theorem 4.7]{MR3264365}, the interpolation error bound \eqref{eq:Pi_02}, and the stability estimate \eqref{eq:Pi_01} as follows: If $K \in \T$ is such that $\mathbf{z} \in K$, then
\begin{equation*}\label{eq:reliability_termI}
\begin{array}{rl}
\text{I}
&
\lesssim |\mathbf{F}|\left(h_{K}^{\frac{\alpha}{2}-1}\|\mathbf{v}-\Pi_{\T}\mathbf{v}\|_{\mathbf{L}^2(\mathsf{d}_{\mathbf{z}}^{-\alpha},K)}+h_{K}^{\frac{\alpha}{2}}\|\nabla(\mathbf{v}-\Pi_{\T}\mathbf{v})\|_{\mathbf{L}^2(\mathsf{d}_{\mathbf{z}}^{-\alpha},K)}\right)\\
&
\lesssim |\mathbf{F}|h_{K}^{\frac{\alpha}{2}}\|\nabla \mathbf{v}\|_{\mathbf{L}^2(\mathsf{d}_{\mathbf{z}}^{-\alpha},\mathcal{N}_K^*)}.	
\end{array}
\end{equation*}
To bound the terms $\text{II}$ and $\text{III}$, we invoke H\"older's inequality and the interpolation error estimates \eqref{eq:Pi_03} and \eqref{eq:Pi_04} to obtain 
\begin{equation*}\label{eq:reliability_termII-III}
\begin{array}{rl}
\text{II} 
& 
\lesssim \displaystyle\sum_{K\in\T}h_{K}D_{K}^{\frac{\alpha}{2}}\|\mathcal{R}_K\|_{\mathbf{L}^2(K)}\|\nabla\mathbf{v}\|_{\mathbf{L}^2(\mathsf{d}_{\mathbf{z}}^{-\alpha},\mathcal{N}_K^*)},
\\ 
\text{III} 
&
\lesssim \displaystyle\sum_{\gamma\in\mathscr{S}}h_{K}^{\frac{1}{2}}D_{K}^{\frac{\alpha}{2}}\|\mathcal{J}_{\gamma}\|_{\mathbf{L}^2(\gamma)}\|\nabla\mathbf{v}\|_{\mathbf{L}^2(\mathsf{d}_{\mathbf{z}}^{-\alpha},\mathcal{N}_K^*)}.
\end{array}
\end{equation*}

Having bounded the terms $\text{I}$, $\text{II}$, and $\text{III}$, we invoke the inf-sup condition \eqref{eq:infsup_a0} and the identity \eqref{eq:reliability_aux03} to obtain an estimate for $\|\nabla \mathbf{\Phi}\|_{\mathbf{L}^2(\mathsf{d}_{\mathbf{z}}^{\alpha},\Omega)}$:
\begin{multline}
\|\nabla \mathbf{\Phi}\|_{\mathbf{L}^2(\mathsf{d}_{\mathbf{z}}^{\alpha},\Omega)}^2
\lesssim
\left[
\sup_{\mathbf{0}\neq \mathbf{v}\in\mathbf{H}_0^1(\mathsf{d}_\mathbf{z}^{-\alpha},\Omega)}
\dfrac{(\nabla \mathbf{\Phi}, \nabla \mathbf{v})}{\|\nabla \mathbf{v}\|_{\mathbf{L}^2(\mathsf{d}_\mathbf{z}^{-\alpha},\Omega)}}
\right]^2
\\
\lesssim
\sum_{K\in\T}
\left(h_{K}D_{K}^{\alpha}\|\mathcal{J}_{\gamma}\|_{\mathbf{L}^2(\partial K\setminus\partial\Omega)}^2+h_{K}^2D_{K}^{\alpha}\|\mathcal{R}_K\|_{\mathbf{L}^2(K)}^2+h_K^{\alpha}|\mathbf{F}|^2\#(\{\mathbf{z}\}\cap K)\right),
\end{multline}
upon utilizing a finite overlapping property of stars, which guarantees that 
\[
 \left[ \sum_{K \in \T} \|\nabla\mathbf{v}\|^2_{\mathbf{L}^2(\mathsf{d}_{\mathbf{z}}^{-\alpha},\mathcal{N}_K^*)} \right]^{\tfrac{1}{2}} \lesssim \|\nabla\mathbf{v}\|_{\mathbf{L}^2(\mathsf{d}_{\mathbf{z}}^{-\alpha},\Omega)}.
\]
Consequently, we have $\|\nabla \mathbf{\Phi}\|_{\mathbf{L}^2(\mathsf{d}_{\mathbf{z}}^{\alpha},\Omega)}\lesssim \mathcal{E}_{\alpha}(\mathbf{u}_{\T},\mathsf{p}_{\T}; \T)$.

\emph{Step 2:} Let $\psi\in L^2(\mathsf{d}_{\mathbf{z}}^{\alpha},\Omega)$. A basic computation reveals that the function $\tilde{q} :=\mathsf{d}_{\mathbf{z}}^{\alpha}\psi\in L^2(\mathsf{d}_{\mathbf{z}}^{-\alpha},\Omega)$. Define $q=\tilde{q}+c$, where $c\in\mathbb{R}$ is such that $q\in L^2(\mathsf{d}_{\mathbf{z}}^{-\alpha},\Omega)/\mathbb{R}$. Substituting $q$ into the second equation of problem \eqref{eq:ritz}, we obtain
\begin{multline*}
\|\psi\|_{L^2(\mathsf{d}_{\mathbf{z}}^{\alpha},\Omega)}^2 
=
(\psi, q)_{L^2(\Omega)} = b_{+}(\mathbf{e}_{\mathbf{u}},q)
\\
=b_{+}(\mathbf{e}_{\mathbf{u}},\mathsf{d}_{\mathbf{z}}^{\alpha}\psi)=-b_{+}(\mathbf{u}_{\T},\mathsf{d}_{\mathbf{z}}^{\alpha}\psi)\leq \|\text{div }
\mathbf{u}_{\T}\|_{L^2(\mathsf{d}_{\mathbf{z}}^{\alpha},\Omega)}\|\psi\|_{L^2(\mathsf{d}_{\mathbf{z}}^{\alpha},\Omega)},
\end{multline*}
upon utilizing that $\int_{\Omega} \psi = 0$ and $\int_{\Omega} \text{div } \mathbf{e}_{\mathbf{u}} = 0$. We have thus obtained the estimate $\|\psi\|_{L^2(\mathsf{d}_{\mathbf{z}}^{\alpha},\Omega)}\leq \|\text{div }\mathbf{u}_{\T}\|_{L^2(\mathsf{d}_{\mathbf{z}}^{\alpha},\Omega)}$.

\emph{Step 3:} The desired estimate \eqref{eq:reliability_estimate} is obtained from \eqref{eq:upper_bound_ritz} and the estimates derived in steps 1 and 2. This completes the proof.
\end{proof}

\subsection{Local efficiency bounds}
We use classical residual estimation techniques based on the bubble functions constructed in \cite{MR3264365} to derive efficiency bounds for the local indicator $\mathcal{E}_{\alpha}(\mathbf{u}_{\T},\mathsf{p}_{\T};K)$ defined in \eqref{eq:indicator_e}.

Given $K\in\T$, we introduce an element bubble function $\varphi_K$ which satisfies the following properties: $0\leq \varphi_K \leq 1$,
\begin{eqnarray}\label{eq:varphi_01}
\varphi_K(\mathbf{z})=0,
\qquad |K|\lesssim \int_K \varphi_K,
\qquad \|\nabla\varphi_K\|_{\mathbf{L}^{\infty}(R_K)}\lesssim h_K^{-1},
\end{eqnarray}
and there exists a simplex $K^*\subset K$ such that $R_K:=\supp(\varphi_K)\subset K^*$. Notice that, since $\varphi_K$ satisfies \eqref{eq:varphi_01}, we have the bound
\begin{eqnarray}\label{eq:varphi_02}
\|\theta\|_{L^2(R_K)}\lesssim\|\varphi_K^{\frac{1}{2}}\theta\|_{L^2(R_K)}
\quad 
\forall\theta \in \mathbb{P}_5(R_K).
\end{eqnarray}

Second, given $\gamma\in \mathscr{S}$, we introduce a bubble function $\varphi_\gamma$ that satisfies the following properties: $0\leq \varphi_\gamma\leq 1$,
\begin{eqnarray}\label{eq:varphi_03}
\varphi_\gamma(\mathbf{z})=0,
\qquad |\gamma|\lesssim \int_\gamma \varphi_\gamma,
\qquad \|\nabla\varphi_\gamma\|_{\mathbf{L}^{\infty}(R_\gamma)}\lesssim h_\gamma^{-1},
\end{eqnarray}
and $R_{\gamma}:=\supp(\varphi_{\gamma})$ is such that, if $\mathcal{N}_{\gamma}=\{K,K'\}$, there are simplices $K_{*}\subset K$ and $K_{*}'\subset K'$ such that $R_{\gamma}\subset K_{*} \cup K_{*}'\subset  K\cup K'$.

The following estimates are relevant to the efficiency analysis to be performed.

\begin{proposition}[estimates for bubble functions]
Let $\alpha\in (0,2)$, $K\in\T$, and $\varphi_K$ be the function that satisfies \eqref{eq:varphi_01}. Then,
\begin{eqnarray}\label{eq:varphi_04}
h_K\|\nabla(\theta\varphi_K)\|_{\mathbf{L}^2(\mathsf{d}_{\mathbf{z}}^{-\alpha},K)}\lesssim D_K^{-\frac{\alpha}{2}}\|\theta\|_{L^2(K)} \qquad \forall\theta \in \mathbb{P}_{5}(K).
\end{eqnarray}
Let $\alpha\in (0,2)$, $\gamma\in\mathscr{S}$, and $\varphi_{\gamma}$ be the function that satisfies \eqref{eq:varphi_03}. Then
\begin{eqnarray}\label{eq:varphi_05}
h_K^{\frac{1}{2}}\|\nabla(\theta\varphi_\gamma)\|_{\mathbf{L}^2(\mathsf{d}_{\mathbf{z}}^{-\alpha},\mathcal{N}_{\gamma})}\lesssim D_K^{-\frac{\alpha}{2}}\|\theta\|_{L^2(\gamma)} \qquad \forall\theta \in \mathbb{P}_3(\gamma).
\end{eqnarray}
$\theta$ is extended to the elements in $\mathcal{N}_{\gamma}$ as a constant along the direction normal to $\gamma$.
\end{proposition}
\begin{proof}
See \cite[Lemma 5.2]{MR3264365}. 
\end{proof}

We are now ready to analyze efficiency bounds for the local error indicator $\mathcal{E}_{\alpha}(\mathbf{u}_{\T},\mathsf{p}_{\T};K)$ defined in \eqref{eq:indicator_e}.

\begin{theorem}[local efficiency]\label{thm:efficiency}
Let the pair $(\mathbf{u},\mathsf{p})$ $\in$ $\mathbf{H}_0^1(\mathsf{d}_{\mathbf{z}}^{\alpha},\Omega)\times L^2(\mathsf{d}_{\mathbf{z}}^{\alpha},\Omega)/\mathbb{R}$ be the solution to \eqref{eq:modelweak}, and let $(\mathbf{u}_{\T},\mathsf{p}_{\T})\in \mathbf{V}(\T)\times \mathcal{P}(\T)$ be the solution to the discrete problem \eqref{eq:model_discrete}. If $\mathbf{F}$ is such that \eqref{eq:small_cond} holds, then
\begin{multline}\label{eq:bound_efficiency}
\mathcal{E}_{\alpha}(\mathbf{u}_{\T},\mathsf{p}_{\T};K)^2
\lesssim 
\|\nabla\mathbf{e}_{\mathbf{u}}\|_{\mathbf{L}^2(\mathsf{d}_{\mathbf{z}}^{\alpha},\mathcal{N}_K)}^2+\|e_{\mathsf{p}}\|_{L^2(\mathsf{d}_{\mathbf{z}}^{\alpha},\mathcal{N}_K)}^2+h_K^{2}\| \mathbf{e}_{\mathbf{u}}\|_{\mathbf{L}^2(\mathsf{d}_{\mathbf{z}}^{\alpha},\mathcal{N}_K)}^2\\+(1+h_K^2)\| \mathbf{e}_{\mathbf{u}}\|_{\mathbf{L}^4(\mathsf{d}_{\mathbf{z}}^{\alpha},\mathcal{N}_K)}^2
+
\sum_{K'\in \mathcal{N}_K^*}h_{K'}^2D_{K'}^{\alpha}\||\mathbf{u}_{\T}|\mathbf{u}_{\T}-\mathbf{\Pi}_{K'}(|\mathbf{u}_{\T}|\mathbf{u}_{\T})\|_{\mathbf{L}^2(K')}^2.
\end{multline}
Here, $\mathbf{\Pi}_K$ corresponds to orthogonal projection operator onto $[\mathbb{P}_0(K)]^2$, $\mathcal{N}_K^*$ is defined in \eqref{eq:patch}, and the hidden constant is independent of $(\mathbf{u},\mathsf{p})$ and $(\mathbf{u}_{\T},\mathsf{p}_{\T})$, the size of the elements in the mesh $\T$, and $\# \T$. 
\end{theorem}
\begin{proof}
We bound each contribution in \eqref{eq:indicator_e} separately. In doing so, we proceed in five steps.

\emph{Step 1:} Let $K\in\T$. In a first step, we bound the term $h_K^{2}D_K^{\alpha}\|\mathcal{R}_K\|^2_{\mathbf{L}^2(K)}$. To accomplish this task, we define
\[
\tilde{\mathcal{R}}_K:= (\Delta \mathbf{u}_{\T}
-
\mathbf{u}_{\T}
-
(\mathbf{u}_{\T}\cdot \nabla)\mathbf{u}_{\T}
-
\mathbf{u}_{\T}\text{div }\mathbf{u}_{\T}
-
\Pi_{K}(|\mathbf{u}_{\T}|\mathbf{u}_{\T})
-
\nabla\mathsf{p}_{\T})|_K.
\]
Notice that $\tilde{\mathcal{R}}_K = \mathcal{R}_K+|\mathbf{u}_{\T}|\mathbf{u}_{\T}-\mathbf{\Pi}_{K}(|\mathbf{u}_{\T}|\mathbf{u}_{\T})$. A simple application of the triangle inequality yields a first estimate for 
$\|\mathcal{R}_K\|_{\mathbf{L}^2(K)}$:
\begin{eqnarray}\label{eq:bound_interp_Pi}
\|\mathcal{R}_K\|_{\mathbf{L}^2(K)} \leq \|\tilde{\mathcal{R}}_K\|_{\mathbf{L}^2(K)}+\|\mathbf{\Pi}_{K}(|\mathbf{u}_{\T}|\mathbf{u}_{\T})-|\mathbf{u}_{\T}|\mathbf{u}_{\T}\|_{\mathbf{L}^2(K)}.
\end{eqnarray}
It thus suffices to control $\|\tilde{\mathcal{R}}_K\|_{\mathbf{L}^2(K)}$. To do this, we define $\boldsymbol{\phi}_K := \varphi_K\tilde{\mathcal{R}}_K$, and observe that \eqref{eq:varphi_02} guarantees the bound 
\begin{eqnarray}\label{eq:bound_efficiency_01}
\|\tilde{\mathcal{R}}_K\|_{\mathbf{L}^2(K)}^2\lesssim \int_{R_K}|\tilde{\mathcal{R}}_K|^2 \varphi_K=\int_K \tilde{\mathcal{R}}_K\cdot \boldsymbol{\phi}_K.
\end{eqnarray}
Let us now utilize that $\varphi_K(\mathbf{z})=0$ to immediately deduce the relations $\boldsymbol{\phi}_K(\mathbf{z})=\varphi_K(\mathbf{z})\tilde{\mathcal{R}}_K(\mathbf{z}) = \mathbf{0}$. We thus set $\mathbf{v}=\boldsymbol{\phi}_K$ as a test function in the identity \eqref{eq:reliability_aux02} and use that $\boldsymbol{\phi}_K|_{\gamma} = \mathbf{0}$, for every $\gamma\in \mathscr{S}_K$, to obtain
\begin{eqnarray}\label{eq:Rk_tilde}
\int_K \tilde{\mathcal{R}}_K \cdot \boldsymbol{\phi}_K=(\nabla\mathbf{\Phi},\nabla \boldsymbol{\phi}_K)_{\mathbf{L}^2(\Omega)}
+
\int_K (|\mathbf{u}_{\T}|\mathbf{u}_{\T}-\mathbf{\Pi}_{K}(|\mathbf{u}_{\T}|\mathbf{u}_{\T}))\cdot \boldsymbol{\phi}_K. 
\end{eqnarray}
We now control $|(\nabla\mathbf{\Phi},\nabla \boldsymbol{\phi}_K)_{\mathbf{L}^2(K)}|$. For this purpose, we set $\mathbf{v}= \boldsymbol{\phi}_K$ as a test function in the first equation of problem \eqref{eq:ritz} and use the property $\supp \boldsymbol{\phi}_K\subset K$ and H\"older's inequality to obtain the bound
\begin{multline}\label{eq:bound_efficiency_02_EO}
|(\nabla\mathbf{\Phi},\nabla \boldsymbol{\phi}_K)_{\mathbf{L}^2(K)}|
\lesssim 
\|\nabla \mathbf{e}_{\mathbf{u}}\|_{\mathbf{L}^2(\mathsf{d}_{\mathbf{z}}^{\alpha},K)} \|\nabla\boldsymbol{\phi}_K\|_{\mathbf{L}^2(\mathsf{d}_{\mathbf{z}}^{-\alpha},K)}
\\
+ \|\mathbf{e}_{\mathbf{u}}\|_{\mathbf{L}^2(\mathsf{d}_{\mathbf{z}}^{\alpha},K)}\|\boldsymbol{\phi}_K\|_{\mathbf{L}^2(\mathsf{d}_{\mathbf{z}}^{-\alpha},K)}
+
\|e_{\mathsf{p}}\|_{L^2(\mathsf{d}_{\mathbf{z}}^{\alpha},K)}\|\nabla\boldsymbol{\phi}_K\|_{\mathbf{L}^2(\mathsf{d}_{\mathbf{z}}^{-\alpha},K)}
+
\|\mathbf{e}_{\mathbf{u}}\|_{\mathbf{L}^4(\mathsf{d}_{\mathbf{z}}^{\alpha},K)} 
\\
\cdot \left[
\|\mathbf{u}\|_{\mathbf{L}^4(\mathsf{d}_{\mathbf{z}}^{\alpha},K)} 
+ 
\|\mathbf{u}_{\T}\|_{\mathbf{L}^4(\mathsf{d}_{\mathbf{z}}^{\alpha},K)}
\right]
\left[\|\nabla\boldsymbol{\phi}_K\|_{\mathbf{L}^2(\mathsf{d}_{\mathbf{z}}^{-\alpha},K)}+\|\boldsymbol{\phi}_K\|_{\mathbf{L}^2(\mathsf{d}_{\mathbf{z}}^{-\alpha},K)}
\right].
\end{multline}
We now notice that, in view of \cite[Theorem 1.3]{MR643158}, we have the bounds: $\|\mathbf{u}\|_{\mathbf{L}^4(\mathsf{d}_{\mathbf{z}}^{\alpha},K)} \lesssim \|\nabla \mathbf{u}\|_{\mathbf{L}^2(\mathsf{d}_{\mathbf{z}}^{\alpha},\Omega)}$ and $\|\mathbf{u}_{\T}\|_{\mathbf{L}^4(\mathsf{d}_{\mathbf{z}}^{\alpha},K)} \lesssim \|\nabla \mathbf{u}_{\T}\|_{\mathbf{L}^2(\mathsf{d}_{\mathbf{z}}^{\alpha},\Omega)}$. On the other hand, the estimate \eqref{eq:varphi_04} and the estimate (5.6) in \cite{MR3264365} allow us to conclude
\[
\|\nabla\boldsymbol{\phi}_K\|_{\mathbf{L}^2(\mathsf{d}_{\mathbf{z}}^{-\alpha},K)}
\lesssim 
h_K^{-1}D_K^{-\frac{\alpha}{2}}\|\tilde{\mathcal{R}}_K\|_{\mathbf{L}^2(K)},
\qquad
\|\boldsymbol{\phi}_K\|_{\mathbf{L}^2(\mathsf{d}_{\mathbf{z}}^{-\alpha},K)}
\lesssim 
D_K^{-\frac{\alpha}{2}}\|\tilde{\mathcal{R}}_K\|_{\mathbf{L}^2(K)},
\]
respectively. Based on these bounds, the estimates \eqref{eq:bound_efficiency_01} and \eqref{eq:bound_efficiency_02_EO} in conjunction with the relation \eqref{eq:Rk_tilde} and the smallness assumption \eqref{eq:small_cond} allow us to derive the local a posteriori bound
\begin{multline}\label{eq:bound_efficiency_05_EO}
h_K^2 D_K^{\alpha}\|\tilde{\mathcal{R}}_K\|_{\mathbf{L}^2(K)}^2
\lesssim 
\|\nabla \mathbf{e}_{\mathbf{u}}\|_{\mathbf{L}^2(\mathsf{d}_{\mathbf{z}}^{\alpha},K)}^2+\|e_{\mathsf{p}}\|_{L^2(\mathsf{d}_{\mathbf{z}}^{\alpha},K)}^2
\\
+h_K^2 \| \mathbf{e}_{\mathbf{u}}\|_{\mathbf{L}^2(\mathsf{d}_{\mathbf{z}}^{\alpha},K)}^2 + (1+h_K^2)\| \mathbf{e}_{\mathbf{u}}\|_{\mathbf{L}^4(\mathsf{d}_{\mathbf{z}}^{\alpha},K)}^2+h_K^2D_K^{\alpha}\||\mathbf{u}_{\T}|\mathbf{u}_{\T}-\mathbf{\Pi}_{K}(|\mathbf{u}_{\T}|\mathbf{u}_{\T})\|_{\mathbf{L}^2(K)}^2.
\end{multline}
A collection of \eqref{eq:bound_interp_Pi} and \eqref{eq:bound_efficiency_05_EO} yields the desired estimate for
$h_K^2 D_K^{\alpha}\|\mathcal{R}_K\|_{\mathbf{L}^2(K)}^2$.

\emph{Step 2:} Let $K\in\T$ and $\gamma\in\mathscr{S}_K$. In what follows, we bound $h_K D_K^{\alpha}\|\mathcal{J}_{\gamma}\|_{\mathbf{L}^2(\gamma)}^2$. To do this, we use arguments similar to those leading to  \eqref{eq:bound_efficiency_05_EO} but now using the bubble function $\varphi_{\gamma}$. Define the function $\mathbf{\Lambda}_{\gamma}=\varphi_{\gamma}\mathcal{J}_{\gamma}$, where $\mathcal{J}_{\gamma}$ and  $\varphi_{\gamma}$ are defined in \eqref{eq:interelement_residual} and \eqref{eq:varphi_03}, respectively. We utilize the construction of the bubble function $\varphi_{\gamma}$ to deduce the bound
\begin{eqnarray}\label{eq:bound_efficiency_06}
\|\mathcal{J}_\gamma\|_{\mathbf{L}^2(\gamma)}^2\lesssim \int_{\gamma}|\mathcal{J}_\gamma|^2 \varphi_\gamma=\int_\gamma \mathcal{J}_\gamma\cdot \mathbf{\Lambda}_\gamma.
\end{eqnarray}
Now, set $\mathbf{v}=\mathbf{\Lambda}_\gamma$ in the identity \eqref{eq:reliability_aux02} and use that $\mathbf{\Lambda}_\gamma$ is such that $\mathbf{\Lambda}_\gamma(\mathbf{z})=0$ and $\supp(\mathbf{\Lambda}_\gamma)\subseteq R_\gamma=\supp(\varphi_\gamma)\subset K_*\cup K_*'\subset \cup\{K':K'\in\mathcal{N}_{\gamma}\}$ to arrive at
\begin{multline}\label{eq:bound_efficiency_07}
\int_\gamma \mathcal{J}_\gamma\cdot \mathbf{\Lambda}_\gamma=(\nabla\mathbf{\Phi},\nabla\mathbf{\Lambda}_\gamma)_{\mathbf{L}^2(\Omega)}-\sum_{K'\in\mathcal{N}_{\gamma}}\int_{K'}
\tilde{\mathcal{R}}_{K'}\cdot\mathbf{\Lambda}_{\gamma}
\\
+\sum_{K'\in\mathcal{N}_{\gamma}}\int_{K'}\left(|\mathbf{u}_{\T}|\mathbf{u}_{\T}-\mathbf{\Pi}_{K'}(|\mathbf{u}_{\T}|\mathbf{u}_{\T})\right)\cdot \mathbf{\Lambda}_{\gamma}.
\end{multline}
In view of this identity, similar arguments to those developed to obtain \eqref{eq:bound_efficiency_02_EO} yield
\begin{multline}
\int_\gamma \mathcal{J}_{\gamma}\cdot\mathbf{\Lambda}_\gamma \leq \left|(\nabla\mathbf{\Phi},\nabla\mathbf{\Lambda}_\gamma)_{\mathbf{L}^2(\mathcal{N}_{\gamma})}\right|
\\
+\sum_{K'\in\mathcal{N}_\gamma} \left( 
\|\tilde{\mathcal{R}}_{K'}\|_{\mathbf{L}^2(K')}+\||\mathbf{u}_{\T}|\mathbf{u}_{\T}-\mathbf{\Pi}_{K'}(|\mathbf{u}_{\T}|\mathbf{u}_{\T})\|_{\mathbf{L}^2(K')}
\right)
\|\mathbf{\Lambda}_\gamma\|_{\mathbf{L}^2(K')}
\\
\lesssim 
\sum_{K'\in\mathcal{N}_\gamma}
\left(\|\nabla \mathbf{e}_{\mathbf{u}}\|_{\mathbf{L}^2(\mathsf{d}_{\mathbf{z}}^{\alpha},K')}+\|e_{\mathsf{p}}\|_{L^2(\mathsf{d}_{\mathbf{z}}^{\alpha},K')}\right)\!\|\nabla\mathbf{\Lambda}_\gamma\|_{\mathbf{L}^2(\mathsf{d}_{\mathbf{z}}^{-\alpha},K')}
+
\|\mathbf{e}_{\mathbf{u}}\|_{\mathbf{L}^2(\mathsf{d}_{\mathbf{z}}^{\alpha},K')}
\\
\cdot\|\mathbf{\Lambda}_\gamma\|_{\mathbf{L}^2(\mathsf{d}_{\mathbf{z}}^{-\alpha},K')}
+
\|\mathbf{e}_{\mathbf{u}}\|_{\mathbf{L}^4(\mathsf{d}_{\mathbf{z}}^{\alpha},K')}\left(\|\nabla\mathbf{\Lambda}_\gamma\|_{\mathbf{L}^2(\mathsf{d}_{\mathbf{z}}^{-\alpha},K')}+\|\mathbf{\Lambda}_\gamma\|_{\mathbf{L}^2(\mathsf{d}_{\mathbf{z}}^{-\alpha},K')}\right)
\\
+\sum_{K'\in\mathcal{N}_\gamma} \left(
\|\tilde{\mathcal{R}}_{K'}\|_{\mathbf{L}^2(K')}+\||\mathbf{u}_{\T}|\mathbf{u}_{\T}-\mathbf{\Pi}_{K'}(|\mathbf{u}_{\T}|\mathbf{u}_{\T})\|_{\mathbf{L}^2(K')}
\right)
\|\mathbf{\Lambda}_\gamma\|_{\mathbf{L}^2(K')}.
\label{eq:eff_aux}
\end{multline}
The terms $\|\nabla\mathbf{\Lambda}_\gamma\|_{\mathbf{L}^2(\mathsf{d}_{\mathbf{z}}^{-\alpha},K')}$ and $\|\mathbf{\Lambda}_\gamma\|_{\mathbf{L}^2(\mathsf{d}_{\mathbf{z}}^{-\alpha},K')}$ can be controlled in view of \eqref{eq:varphi_05} and \cite[estimate (5.8)]{MR3264365}, respectively. In fact, we have
\begin{equation}\label{eq:bound_efficiency_13}
\|\nabla\mathbf{\Lambda}_\gamma\|_{\mathbf{L}^2(\mathsf{d}_{\mathbf{z}}^{-\alpha},K')}
\lesssim 
h_{K'}^{-\frac{1}{2}}D_{K'}^{-\frac{\alpha}{2}}\|\mathcal{J}_{\gamma}\|_{\mathbf{L}^2(\gamma)},
\,\,
\|\mathbf{\Lambda}_\gamma\|_{\mathbf{L}^2(\mathsf{d}_{\mathbf{z}}^{-\alpha},K')}
\lesssim 
h_{K'}^{\frac{1}{2}}D_{K'}^{-\frac{\alpha}{2}}\|\mathcal{J}_{\gamma}\|_{\mathbf{L}^2(\gamma)}.
\end{equation}
We also observe that
$\|\mathbf{\Lambda}_\gamma\|_{\mathbf{L}^2(K')}\approx |K'|^{\frac{1}{2}}|\gamma|^{-\frac{1}{2}}\|\mathbf{\Lambda}_\gamma\|_{\mathbf{L}^2(\gamma)}\approx h_{K'}^{\frac{1}{2}}\|\mathbf{\Lambda}_\gamma\|_{\mathbf{L}^2(\gamma)}$, as a consequence of $|K'|\approx h_{K'}^2$, $|\gamma|\approx h_{K'}$, and standard arguments. With these ingredients at hand, the inequalities in \eqref{eq:eff_aux} show that
\begin{multline*}
\int_\gamma \mathcal{J}_{\gamma}\cdot\mathbf{\Lambda}_\gamma 
\lesssim 
\sum_{K'\in\mathcal{N}_\gamma} \left(\|\nabla \mathbf{e}_{\mathbf{u}}\|_{\mathbf{L}^2(\mathsf{d}_{\mathbf{z}}^{\alpha},K')}+\|e_{\mathsf{p}}\|_{L^2(\mathsf{d}_{\mathbf{z}}^{\alpha},K')}\right)h_{K'}^{-\frac{1}{2}}D_{K'}^{-\frac{\alpha}{2}}\|\mathbf{\Lambda}_\gamma\|_{\mathbf{L}^2(\gamma)}
\\
+\sum_{K'\in\mathcal{N}_\gamma} \left[
\|\mathbf{e}_{\mathbf{u}}\|_{\mathbf{L}^2(\mathsf{d}_{\mathbf{z}}^{\alpha},K')}h_{K'}^{\frac{1}{2}}
+
\|\mathbf{e}_{\mathbf{u}}\|_{\mathbf{L}^4(\mathsf{d}_{\mathbf{z}}^{\alpha},K')}(h_{K'}^{-\frac{1}{2}}+h_{K'}^{\frac{1}{2}})
\right]
D_{K'}^{-\frac{\alpha}{2}}
\|\mathbf{\Lambda}_\gamma\|_{\mathbf{L}^2(\gamma)}
\\
+\sum_{K'\in\mathcal{N}_\gamma}h_{K'}^{\frac{1}{2}}(\|\tilde{\mathcal{R}}_{K'}\|_{\mathbf{L}^2(K')}+\||\mathbf{u}_{\T}|\mathbf{u}_{\T}-\mathbf{\Pi}_{K'}(|\mathbf{u}_{\T}|\mathbf{u}_{\T})\|_{\mathbf{L}^2(K')})\|\mathbf{\Lambda}_\gamma\|_{\mathbf{L}^2(\gamma)}.
\end{multline*}
The desired control for the term $h_K D_K^{\alpha}\|\mathcal{J}_\gamma\|_{\mathbf{L}^2(\gamma)}^2$ follows from replacing the previous estimate in \eqref{eq:bound_efficiency_06}:
\begin{multline}\label{eq:bound_efficiency_08}
h_K D_K^{\alpha}\|\mathcal{J}_\gamma\|_{\mathbf{L}^2(\gamma)}^2\lesssim	\sum_{K'\in\mathcal{N}_\gamma}
\bigg(\|\nabla \mathbf{e}_{\mathbf{u}}\|_{\mathbf{L}(\mathsf{d}_{\mathbf{z}}^{\alpha},K')}^2\!+\!\|e_{\mathsf{p}}\|_{L^2(\mathsf{d}_{\mathbf{z}}^{\alpha},K')}^2
+h_{K'}^{2}\| \mathbf{e}_{\mathbf{u}}\|_{\mathbf{L}^2(\mathsf{d}_{\mathbf{z}}^{\alpha},K')}^2
\\
+
(1+h_{K'}^2)
\| \mathbf{e}_{\mathbf{u}}\|_{\mathbf{L}^4(\mathsf{d}_{\mathbf{z}}^{\alpha},K')}^2+h_{K'}^{2}D_{K'}^{\alpha}\||\mathbf{u}_{\T}|\mathbf{u}_{\T}-\mathbf{\Pi}_{K'}(|\mathbf{u}_{\T}|\mathbf{u}_{\T})\|_{\mathbf{L}^2(K')}^2
\bigg).
\end{multline}

\emph{Step 3:} Let $K\in\T$. The control of $\|\text{div }\mathbf{u}_{\T}\|_{L^2(\mathsf{d}_{\mathbf{z}}^{\alpha},K)}$ follows easily from the mass conservation equation $\text{div }\mathbf{u} = 0$. In fact,
\begin{align}\label{eq:bound_efficiency_09}
	\|\text{div }\mathbf{u}_{\T}\|_{L^2(\mathsf{d}_{\mathbf{z}}^{\alpha},K)}=\|\text{div }\mathbf{e}_{\mathbf{u}}\|_{L^2(\mathsf{d}_{\mathbf{z}}^{\alpha},K)}\lesssim \|\nabla\mathbf{e}_{\mathbf{u}}\|_{\mathbf{L}^2(\mathsf{d}_{\mathbf{z}}^{\alpha},K)}.
\end{align}

\emph{Step 4:} Let $K\in\T$. We now control $h_K^{\alpha}|\mathbf{F}|^2 \# (\{\mathbf{z}\} \cap  K)$.
Let us first note that if $K \cap \{ \mathbf{z} \}=\emptyset$, then the desired bound \eqref{eq:bound_efficiency} follows directly from the estimates obtained in the Steps 1, 2, and 3. On the other hand, if $K \cap \{ \mathbf{z} \} = \{ \mathbf{z} \}$, we must obtain a bound for $h_K^{\alpha}|\mathbf{F}|^2$ in \eqref{eq:indicator_e}. To accomplish this task, we invoke the smooth function $\mu$ introduced in the proof of \cite[Theorem 5.3]{MR3264365}, which is such that
\begin{align}\label{eq:mu_function}
\mu(\mathbf{z})=1,
\qquad \|\mu\|_{L^{\infty}(\Omega)}=1,
\qquad \|\nabla \eta\|_{\mathbf{L}^{\infty}(\Omega)}
\lesssim 
h_K^{-1},
\qquad \supp(\mu)\subset \mathcal{N}_K^*.
\end{align}
With $\mu$ at hand, we define $\mathbf{v}_{\mu}:=\mathbf{F}\mu\in \mathbf{H}_0^1(\mathsf{d}_{\mathbf{z}}^{-\alpha},\Omega)$. Let us now invoke the fact that $(\mathbf{u},\mathsf{p})$ and $(\mathbf{\Phi},\psi)$ solve problems \eqref{eq:modelweak} and \eqref{eq:ritz}, respectively, to obtain
\begin{multline}\label{eq:bound_efficiency_10}
|\mathbf{F}|^2
=
\langle \mathbf{F}\delta_{\mathbf{z}},\mathbf{v}_{\mu}\rangle
=
a(\mathbf{u},\mathbf{v}_{\mu})
+
b_{-}(\mathbf{v}_{\mu},\mathsf{p})
+
c(\mathbf{u},\mathbf{u};\mathbf{v}_{\mu})
+
d(\mathbf{u},\mathbf{u};\mathbf{v}_{\mu})
\\
=(\nabla\mathbf{\Phi},\nabla \mathbf{v}_{\mu})_{\mathbf{L}^2(\Omega)}
+
a(\mathbf{u}_{\T},\mathbf{v}_{\mu})
+
b_{-}(\mathbf{v}_{\mu},\mathsf{p}_{\T})
+
c(\mathbf{u}_{\T},\mathbf{u}_{\T};\mathbf{v}_{\mu})
+
d(\mathbf{u}_{\T},\mathbf{u}_{\T};\mathbf{v}_{\mu}).
\end{multline}
Since $\supp(\mu)\subset \mathcal{N}_K^*$, similar arguments to the ones utilized to obtain \eqref{eq:bound_efficiency_02_EO} yields
\begin{multline}\label{eq:bound_efficiency_11}
|(\nabla\mathbf{\Phi},\nabla \mathbf{v}_{\mu})_{\mathbf{L}^2(\Omega)}|
\lesssim 
\|\mathbf{e}_{\mathbf{u}}\|_{\mathbf{L}^2(\mathsf{d}_{\mathbf{z}}^{\alpha},\mathcal{N}_K^*)}
\|\mathbf{v}_{\mu}\|_{\mathbf{L}^2(\mathsf{d}_{\mathbf{z}}^{-\alpha},\mathcal{N}_K^*)}
\\
+
\left[
\|\nabla \mathbf{e}_{\mathbf{u}}\|_{\mathbf{L}^2(\mathsf{d}_{\mathbf{z}}^{\alpha},\mathcal{N}_K^*)}+\|e_{\mathsf{p}}\|_{L^2(\mathsf{d}_{\mathbf{z}}^{\alpha},\mathcal{N}_K^*)}
\right]
\|\nabla\mathbf{v}_{\mu}\|_{\mathbf{L}^2(\mathsf{d}_{\mathbf{z}}^{-\alpha},\mathcal{N}_K^*)}
\\
\|\mathbf{e}_{\mathbf{u}}\|_{\mathbf{L}^4(\mathsf{d}_{\mathbf{z}}^{\alpha},\mathcal{N}_K^*)}
\left[
\|\nabla\mathbf{v}_{\mu}\|_{\mathbf{L}^2(\mathsf{d}_{\mathbf{z}}^{-\alpha},\mathcal{N}_K^*)}
+
\|\mathbf{v}_{\mu}\|_{\mathbf{L}^2(\mathsf{d}_{\mathbf{z}}^{-\alpha},\mathcal{N}_K^*)}
\right].
\end{multline} 
In view of the identity \eqref{eq:bound_efficiency_10}, the bound \eqref{eq:bound_efficiency_11}, and basic estimates on the basis of an integrations by parts arguments, we obtain
\begin{multline*}
|\mathbf{F}|^2\lesssim 
\left[\|\nabla \mathbf{e}_{\mathbf{u}}\|_{\mathbf{L}^2(\mathsf{d}_{\mathbf{z}}^{\alpha},\mathcal{N}_K^*)}+\|e_{\mathsf{p}}\|_{L^2(\mathsf{d}_{\mathbf{z}}^{\alpha},\mathcal{N}_K^*)}\right]
\|\nabla\mathbf{v}_{\mu}\|_{\mathbf{L}^2(\mathsf{d}_{\mathbf{z}}^{-\alpha},\mathcal{N}_K^*)}
+ \|\mathbf{e}_{\mathbf{u}}\|_{\mathbf{L}^2(\mathsf{d}_{\mathbf{z}}^{\alpha},\mathcal{N}_K^*)}
\\
\cdot \|\mathbf{v}_{\mu}\|_{\mathbf{L}^2(\mathsf{d}_{\mathbf{z}}^{-\alpha},\mathcal{N}_K^*)}
+
\|\mathbf{e}_{\mathbf{u}}\|_{\mathbf{L}^4(\mathsf{d}_{\mathbf{z}}^{\alpha},\mathcal{N}_K^*)}
\left[\|\nabla\mathbf{v}_{\mu}\|_{\mathbf{L}^2(\mathsf{d}_{\mathbf{z}}^{-\alpha},\mathcal{N}_K^*)} + \|\mathbf{v}_{\mu}\|_{\mathbf{L}^2(\mathsf{d}_{\mathbf{z}}^{-\alpha},\mathcal{N}_K^*)}\right]
\\
+\sum_{K'\in\T: K'\subset \mathcal{N}_K^*}
\left(\|\tilde{\mathcal{R}}_{K'}\|_{\mathbf{L}^2(K')}+\||\mathbf{u}_{\T}|\mathbf{u}_{\T}-\mathbf{\Pi}_{K'}(|\mathbf{u}_{\T}|\mathbf{u}_{\T})\|_{\mathbf{L}^2(K')}
\right)
\|\mathbf{v}_{\mu}\|_{\mathbf{L}^2(K')}
\\+\sum_{K'\in\T: K'\subset \mathcal{N}_K^*}\sum_{\gamma\in\mathscr{S}_{K'}: \gamma\not\subset \partial\mathcal{N}_K}\|\mathcal{J}_{\gamma}\|_{\mathbf{L}^2(\gamma)}\|\mathbf{v}_{\mu}\|_{\mathbf{L}^2(\gamma)}.
\end{multline*}
We now use the estimates
\begin{equation*}\label{eq:mu_estimates}
\|\mu\|_{L^2(\gamma)}\lesssim h_K^{\frac{1}{2}}, 
\quad
\|\mu\|_{L^2(\mathcal{N}_K^*)}\lesssim h_K,
\quad
\|\mu\|_{L^{2}(\mathsf{d}_{\mathbf{z}}^{-\alpha},\mathcal{N}_K^*)}\lesssim h_K^{1-\frac{\alpha}{2}},
\end{equation*}
and $\|\nabla\mu\|_{\mathbf{L}^{2}(\mathsf{d}_{\mathbf{z}}^{-\alpha},\mathcal{N}_K^*)}\lesssim h_K^{-\frac{\alpha}{2}}$ together with the fact that, since $\mathbf{z}\in K$, we have $h_{K}\approx D_K$, to conclude that
\begin{multline}\label{eq:bound_efficiency_12}
|\mathbf{F}|^2
\lesssim 
h_K^{-\frac{\alpha}{2}}
|\mathbf{F}| 
\bigg[\|\nabla \mathbf{e}_{\mathbf{u}}\|_{\mathbf{L}^2(\mathsf{d}_{\mathbf{z}}^{\alpha},\mathcal{N}_K^*)}^2\!+\!\|e_{\mathsf{p}}\|_{L^2(\mathsf{d}_{\mathbf{z}}^{\alpha},\mathcal{N}_K^*)}^2
+
h_K^{2} \|\mathbf{e}_{\mathbf{u}}\|_{\mathbf{L}^2(\mathsf{d}_{\mathbf{z}}^{\alpha},\mathcal{N}_K^*)}^2
\\
+
(1+h_K^2) \|\mathbf{e}_{\mathbf{u}}\|_{\mathbf{L}^4(\mathsf{d}_{\mathbf{z}}^{\alpha},\mathcal{N}_K^*)}^2
\bigg]^{\frac{1}{2}}+h_K^{-\frac{\alpha}{2}}|\mathbf{F}|\Bigg[\sum_{K'\in\T: K'\subset \mathcal{N}_K^*}\sum_{\gamma\in\mathscr{S}_{K'}: \gamma\not\subset \partial\mathcal{N}_K^*}h_{K'}^{\frac{1}{2}}D_{K'}^{\frac{\alpha}{2}}\|\mathcal{J}_{\gamma}\|_{\mathbf{L}^2(\gamma)}
\\
+\sum_{K'\in\T: K'\subset \mathcal{N}_K^*}
h_{K'}D_{K'}^{\frac{\alpha}{2}}\left(\|\tilde{\mathcal{R}}_{K'}\|_{\mathbf{L}^2(K')}+\||\mathbf{u}_{\T}|\mathbf{u}_{\T}-\mathbf{\Pi}_{K'}(|\mathbf{u}_{\T}|\mathbf{u}_{\T})\|_{\mathbf{L}^2(K')}\right)\Bigg].
\end{multline}
Replacing the estimates \eqref{eq:bound_efficiency_05_EO} and \eqref{eq:bound_efficiency_08} in the previous bound allows us to conclude.

\emph{Step 5:} By combining the estimates derived in the previous steps, i.e., estimates \eqref{eq:bound_efficiency_05_EO}, \eqref{eq:bound_efficiency_08}, \eqref{eq:bound_efficiency_09} and \eqref{eq:bound_efficiency_12}, we obtain the desired local efficiency estimate \eqref{eq:bound_efficiency}. This completes the proof.
\end{proof}

\section{Numerical experiments}
\label{sec:numericalexperiments}

In this section, we present a series of numerical examples that illustrate the performance of the estimator $\mathcal{E}_{\alpha}$.

The numerical examples were carried out with a \texttt{C++} code implemented by us. All matrices were assembled exactly, and the global linear systems were solved with the multifrontal massively parallel sparse direct solver (MUMPS) \cite{MUMPS1,MUMPS2}. A quadrature formula was used to compute the right-hand sides, the local indicators, and the error estimator,
which guarantees accuracy by using polynomials of degree 19.  ParaView \cite{Ayachit2015ThePG} was used to visualize suitable finite element approximations.

For a given partition $\T$ we solve the discrete system \eqref{eq:model_discrete} with the lowest order Taylor--Hood pair \eqref{TH:vel_space}--\eqref{TH:press_space} using the iterative strategy described in \textbf{Algorithm 1}. Once we obtain a discrete solution, for each $K \in \mathscr{T}$ we compute  the local error indicator $\mathcal{E}_{\alpha}(\mathbf{u}_{\T},\mathsf{p}_{\T};K)$, defined in \eqref{eq:indicator_e}, to drive the adaptive procedure described in \textbf{Algorithm 2}. In this way, a sequence of adaptively refined meshes is generated from the initial meshes shown in Figure \ref{fig:mesh}.
 
\begin{figure}[!ht]
\centering
\hspace{-1.0cm}
\begin{minipage}[b]{0.35\textwidth}\centering
\includegraphics[width=2cm,height=1.8cm,scale=0.66]{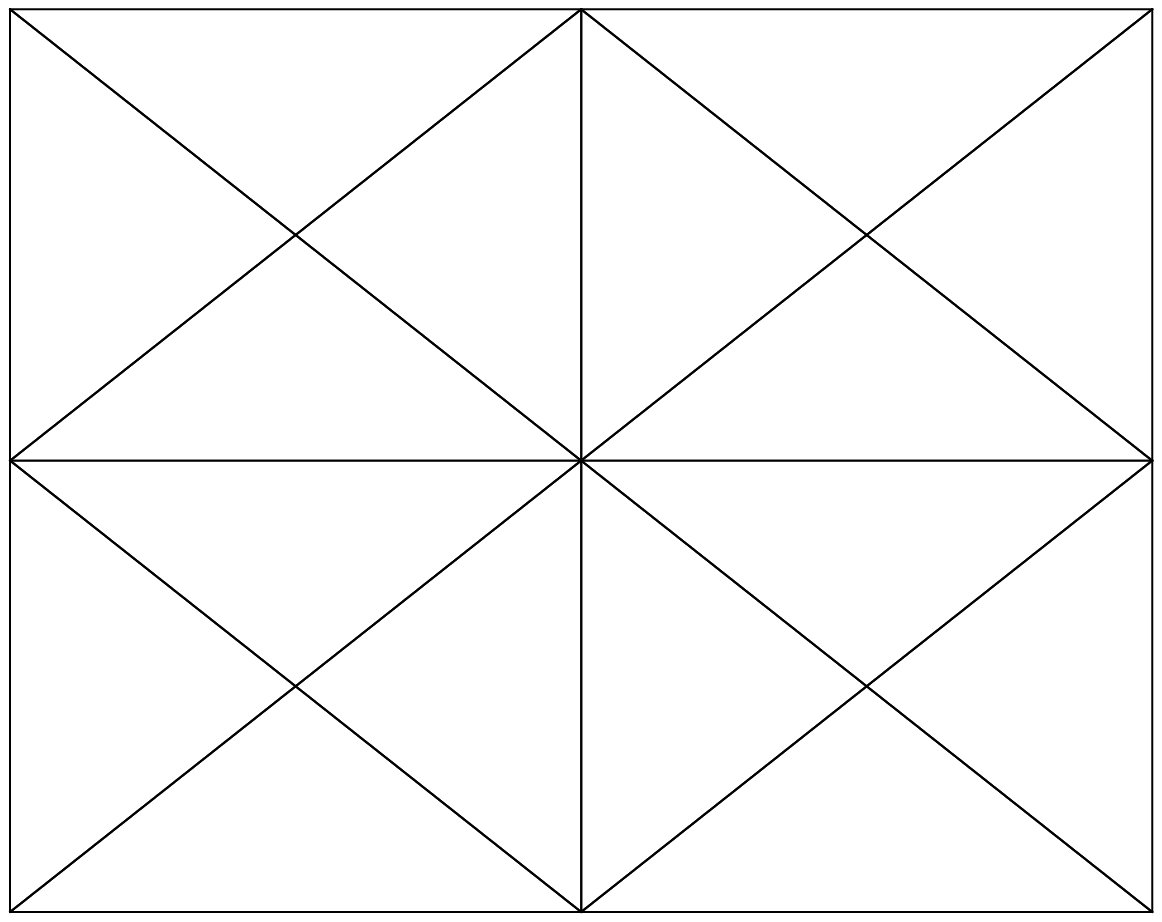} \\
\tiny{(A.1)}
\end{minipage}
\begin{minipage}[b]{0.35\textwidth}\centering
\includegraphics[width=2cm,height=1.8cm,scale=0.66]{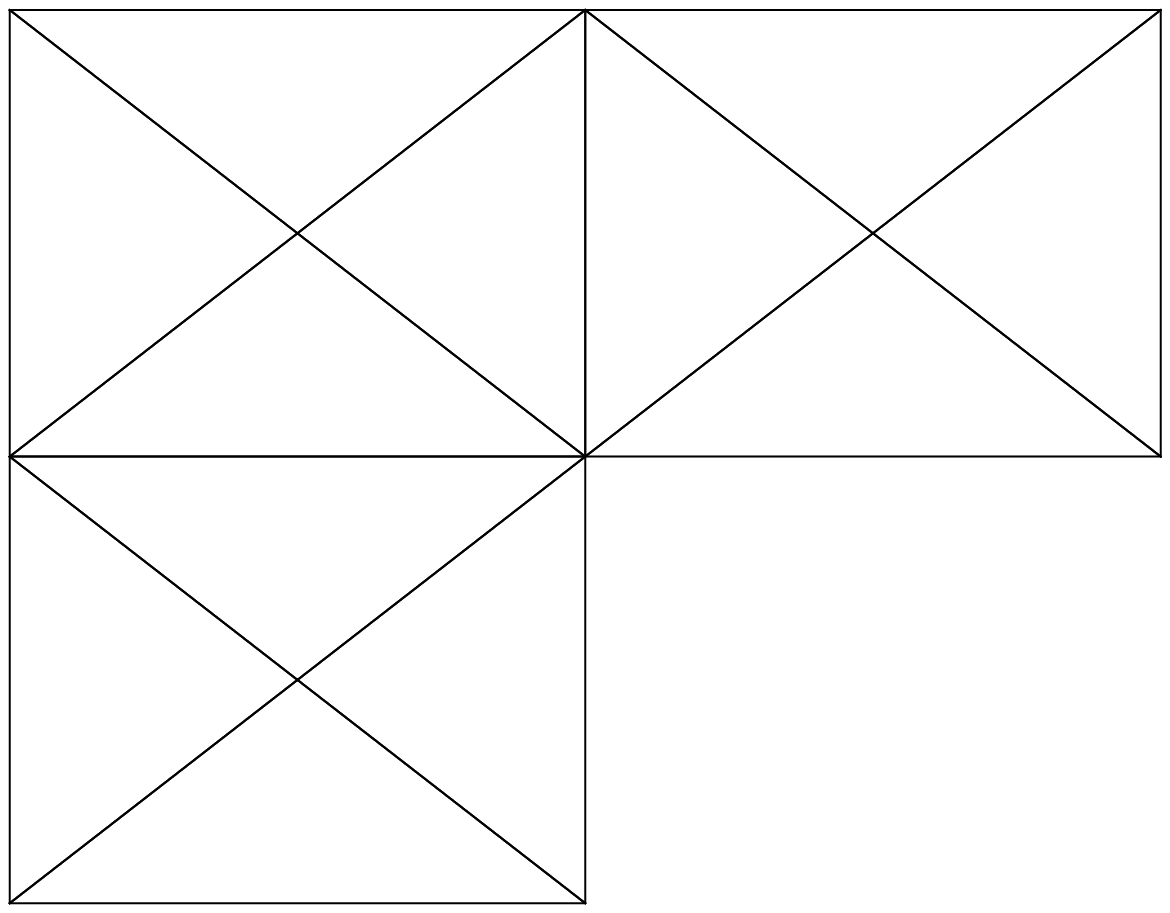} \\
\qquad \tiny{(A.2)}
\end{minipage}
\begin{minipage}[b]{0.35\textwidth}\centering
\includegraphics[width=4cm,height=1.8cm,scale=0.66]{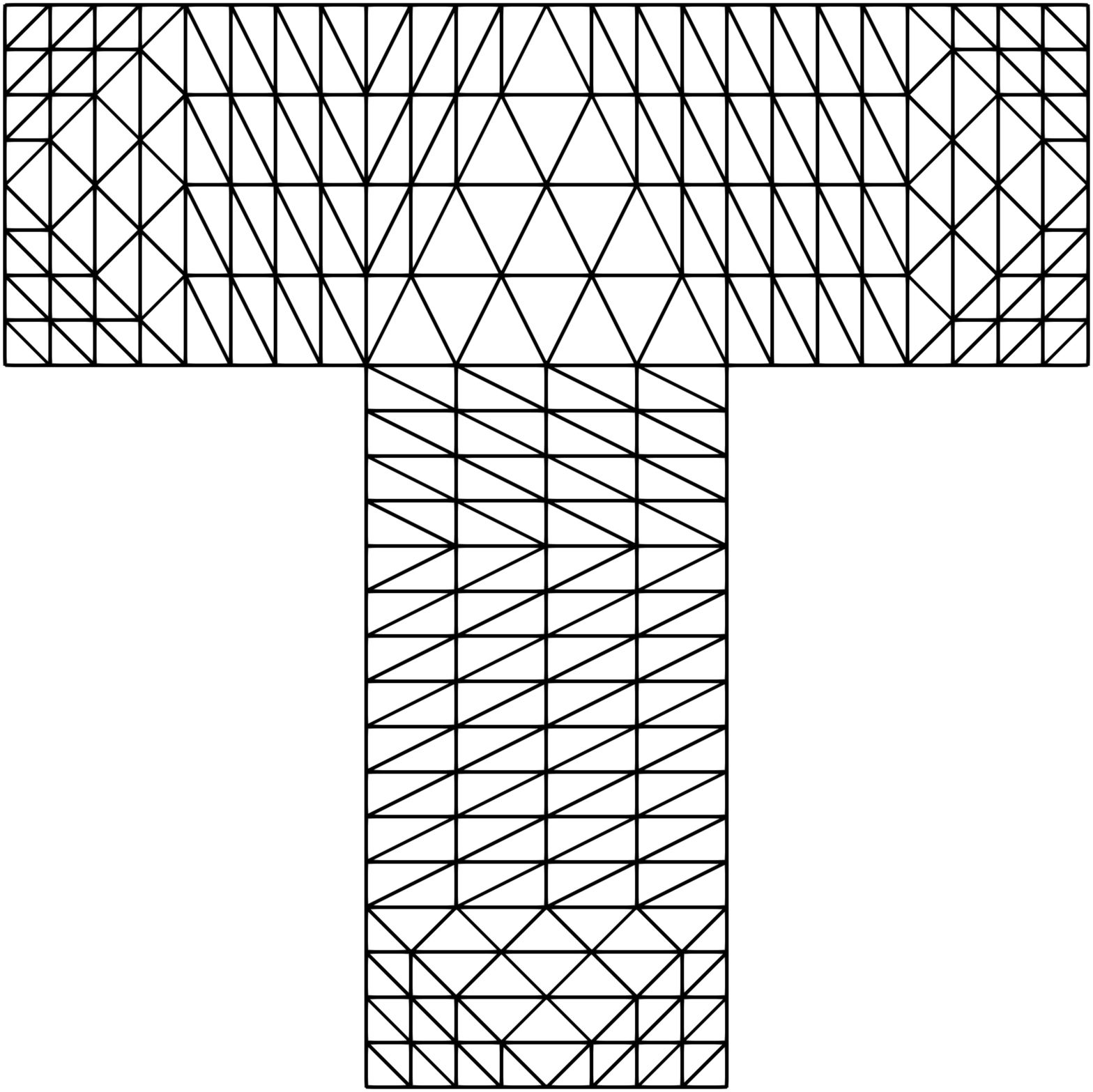} \\
\qquad \tiny{(A.3)}
\end{minipage}
\caption{The initial meshes used in the adaptive algorithm, \emph{Algorithm} 2, when \textrm{(A.1)} $\Omega = (0,1)^2$, \textrm{(A.2)} $\Omega=(-1,1)^2 \setminus[0,1)\times (-1,0]$, and \textrm{(A.3)} $\Omega=((-1.5,1.5)\times(0,1))\cup((-0.5,0.5)\times(-2,1))$.}
\label{fig:mesh}
\end{figure}

Finally, we define the total number of degrees of freedom as $\textnormal{Ndof}:=\textnormal{dim }\mathbf{V}(\T)+\textnormal{dim }\mathcal{P}(\T)$. 
 
\begin{algorithm}[ht]
\caption{\textbf{Iterative Scheme}.}
\label{Algorithm1}
\textbf{Input:} Initial guess $(\mathbf{u}_{\T}^{0},\mathsf{p}_{\T}^{0}) \in \mathbf{V}(\T)\times \mathcal{P}(\T)$, interior point $\mathbf{z} \in \Omega$, $\mathbf{F} \in \mathbb{R}^2$, and tol=$10^{-8}$. Set $i$=1;
\\
$\boldsymbol{1}$: Find $(\mathbf{u}_{\T}^{i},\mathsf{p}_{\T}^{i}) \in \mathbf{V}(\T)\times \mathcal{P}(\T)$ such that
\begin{equation*}
\begin{array}{rcl}
a(\mathbf{u}_{\T}^{i},\mathbf{v}_{\T})+b_{-}(\mathbf{v}_{\T},\mathsf{p}_{\T}^{i})+c(\mathbf{u}_{\T}^{i-1},\mathbf{u}_{\T}^{i};\mathbf{v}_{\T})+d(\mathbf{u}_{\T}^{i-1},\mathbf{u}_{\T}^{i};\mathbf{v}_{\T})&=&\langle \mathbf{F}\delta_{\mathbf{z}},\mathbf{v}_{\T}\rangle ,\\
b_{+}(\mathbf{u}_{\T}^{i},\mathsf{q}_{\T})&=&0,
\end{array}
\end{equation*}
\\
for all $\mathbf{v}_{\T}\in \mathbf{V}(\T)$ and $\mathsf{q}_{\T}\in \mathcal{P}(\T)$, respectively.

$\boldsymbol{2}$: If $|(\mathbf{u}_{\T}^{i}, \mathsf{p}_{\T}^{i})-(\mathbf{u}_{\T}^{i-1}, \mathsf{p}_{\T}^{i-1})|>$ tol, set $i \leftarrow i + 1$ and go to step $\boldsymbol{1}$. Otherwise, \textbf{return} $(\mathbf{u}_{\T}, \mathsf{p}_{\T}) = (\mathbf{u}_{\T}^{i}, \mathsf{p}_{\T}^{i})$. Here, $|\cdot|$ denotes the Euclidean norm.
\end{algorithm}

\begin{algorithm}[ht]
\caption{\textbf{Adaptive Algorithm.}}
\label{Algorithm2}
\textbf{Input:} Initial mesh $\mathscr{T}_0$, interior point $\mathbf{z} \in \Omega$, $\alpha\in(0,2)$, and $\mathbf{F} \in \mathbb{R}^2$;
\\
$\boldsymbol{1}$: Utilize \textbf{Algorithm} \ref{Algorithm1} to solve the discrete problem \eqref{eq:model_discrete};
\\
$\boldsymbol{2}$: For each $K\in\mathscr{T}$ compute the local error indicator $\mathcal{E}_{\alpha}(\mathbf{u}_{\T},\mathsf{p}_{\T};K)$ defined in \eqref{eq:indicator_e};
\\
$\boldsymbol{3}$: Mark an element $K\in\mathscr{T}$ for refinement if;
\begin{equation*}
\mathcal{E}_{\alpha}(\mathbf{u}_{\T},\mathsf{p}_{\T};K)>\tfrac{1}{2}\max_{K'\in \mathscr{T}} \mathcal{E}_{\alpha}(\mathbf{u}_{\T},\mathsf{p}_{\T};K');
\end{equation*}
$\boldsymbol{4}$: From step $\boldsymbol{3}$ construct a new mesh using a longest edge bisection algorithm. Set $i \leftarrow i + 1$ and go to step $\boldsymbol{1}$.
\end{algorithm}

\subsection{Convex and non-convex domains} 
We investigate the performance of the developed a posteriori error estimator in problems posed on convex and non-convex domains with homogeneous Dirichlet boundary conditions. We recall that we are considering the discrete problem \eqref{eq:model_discrete} in the discrete framework defined by the spaces \eqref{TH:vel_space}--\eqref{TH:press_space}. This framework is called the \emph{Taylor–Hood approximation}.

\subsubsection{Convex domain}
\label{subsec:convex_domain}
We investigate the performance of the a posteriori error estimator $\mathcal{E}_{\alpha}(\mathbf{u}_{\T},\mathsf{p}_{\T};\T)$ when used to guide the adaptive procedure of \textbf{Algorithm 2}. In particular, we study the effects of varying the exponent $\alpha$ in the Muckenhoupt weight. For this purpose, we consider $\Omega=(0,1)^2$, $\mathbf{z}=(0.5,0.5)^{\mathsf{T}}$, $\mathbf{F}=(1,1)^{\mathsf{T}}$, and $\alpha=\{0.25,0.5,0.75,1.0,1.25,1.5,1.75\}$.

Figure \ref{fig:test_01} shows the results obtained for Example 1. We note that the devised a posteriori error estimator $\mathcal{E}_{\alpha}$ achieves optimal computational convergence rates for all considered values of the parameter $\alpha$. We also note that most of the refinement focus on the singular source point.

\begin{figure}[!ht]
\centering
\psfrag{Ndof(-1)}{\Large $\text{Ndof}^{-1}$}
\psfrag{Est p=0.25}{\Large $\alpha=0.25$}
\psfrag{Est p=0.5}{\Large $\alpha=0.5$}
\psfrag{Est p=0.75}{\Large $\alpha=0.75$}
\psfrag{Est p=1.0}{\Large $\alpha=1.0$}
\psfrag{Est p=1.25}{\Large $\alpha=1.25$}
\psfrag{Est p=1.5}{\Large $\alpha=1.5$}
\psfrag{Est p=1.75}{\Large $\alpha=1.75$}

\begin{minipage}[b]{0.24\textwidth}\centering
\scriptsize{\qquad $\mathcal{E}_{\alpha}(\mathbf{u}_{\T},\mathsf{p}_{\T};\T)$}
\includegraphics[trim={0 0 0 0},clip,width=3.1cm,height=3.2cm,scale=0.5]{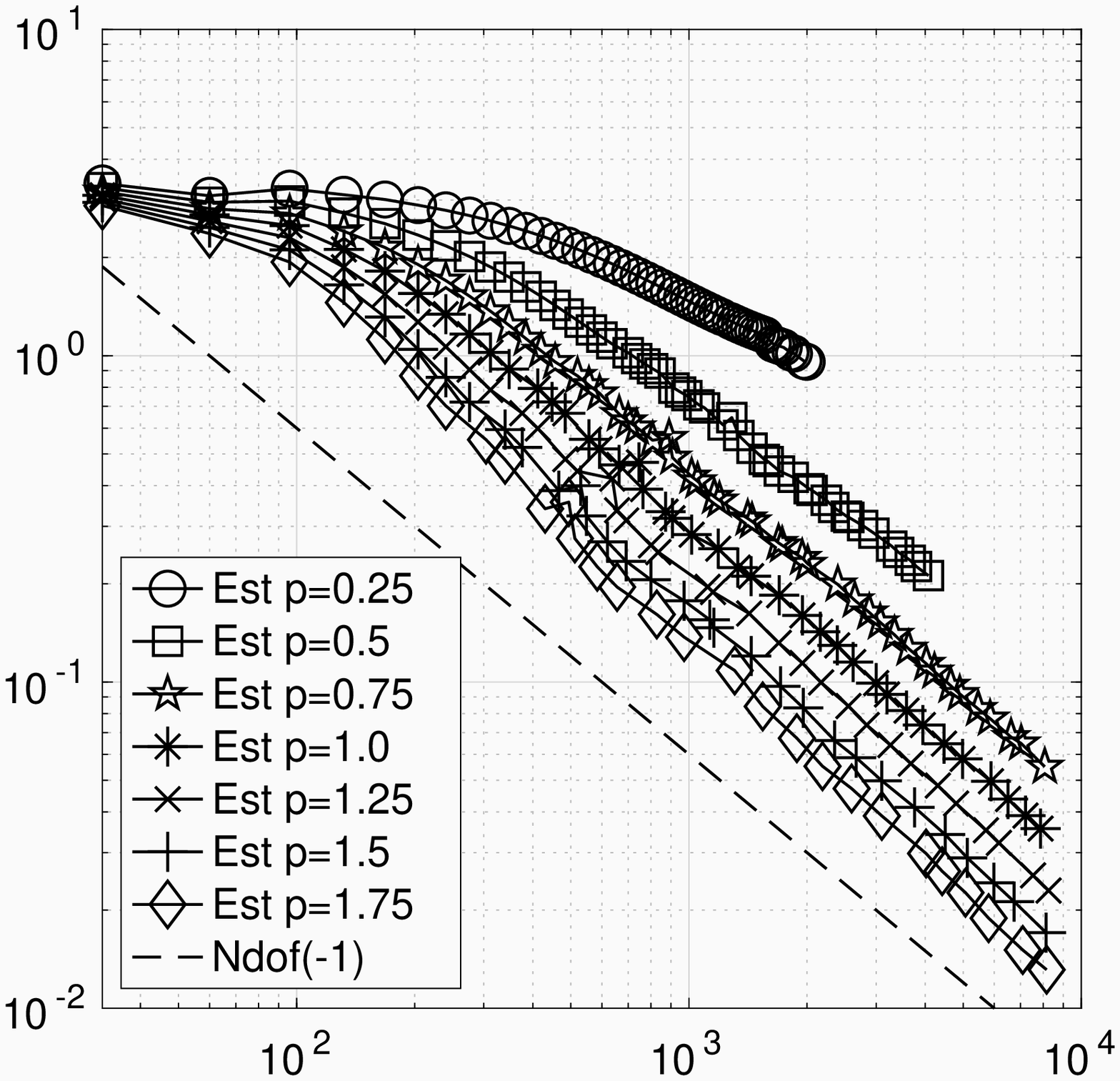} \\
\qquad \tiny{(B.1)}
\end{minipage}
\begin{minipage}[b]{0.24\textwidth}\centering
\includegraphics[trim={25cm 0 25cm 0},clip,width=3.0cm,height=3.4cm,scale=0.5]{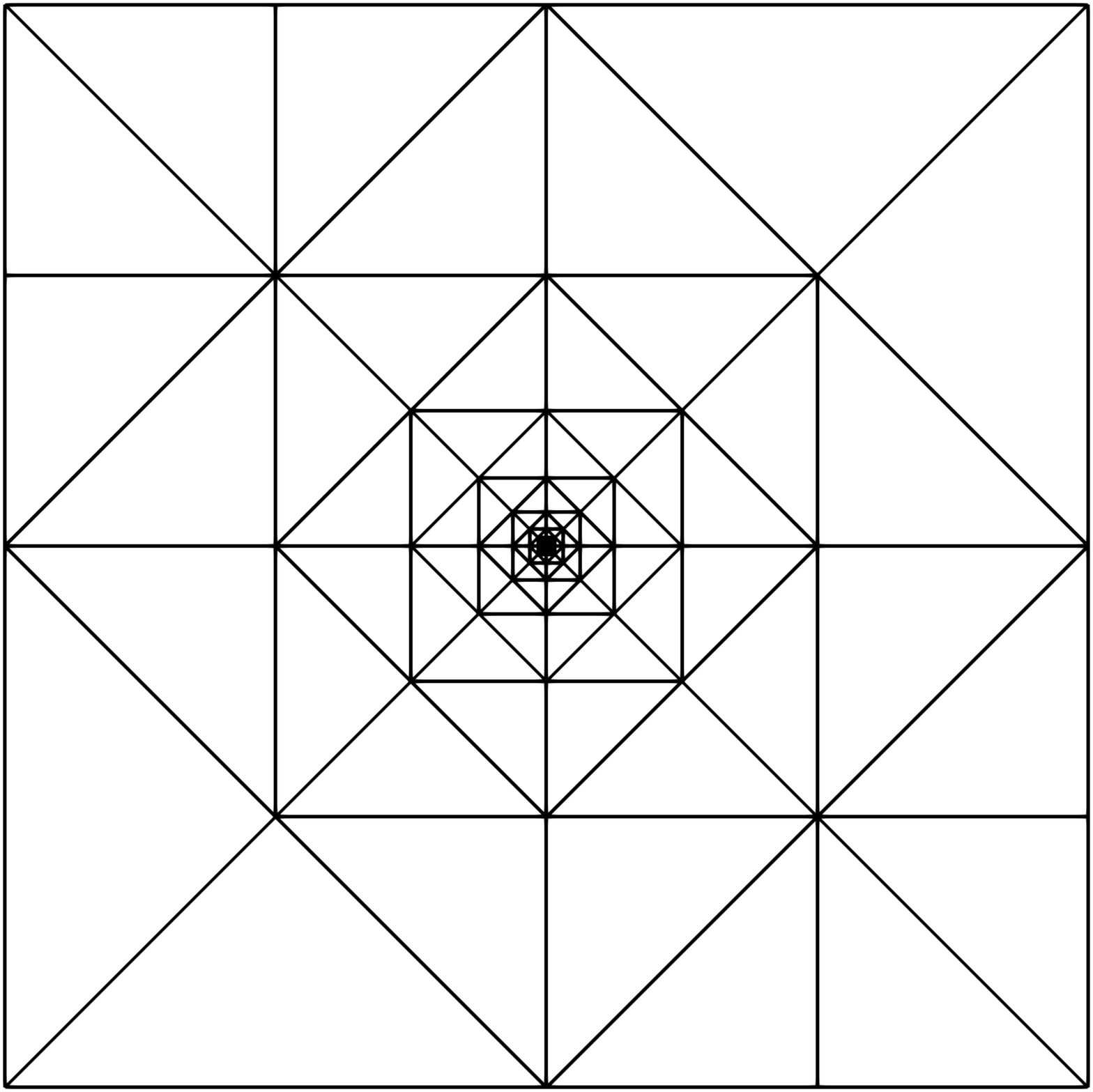} \\
\qquad \tiny{(B.2)}
\end{minipage}
\begin{minipage}[b]{0.24\textwidth}\centering
\includegraphics[trim={25cm 0 25cm 0},clip,width=3.0cm,height=3.4cm,scale=0.5]{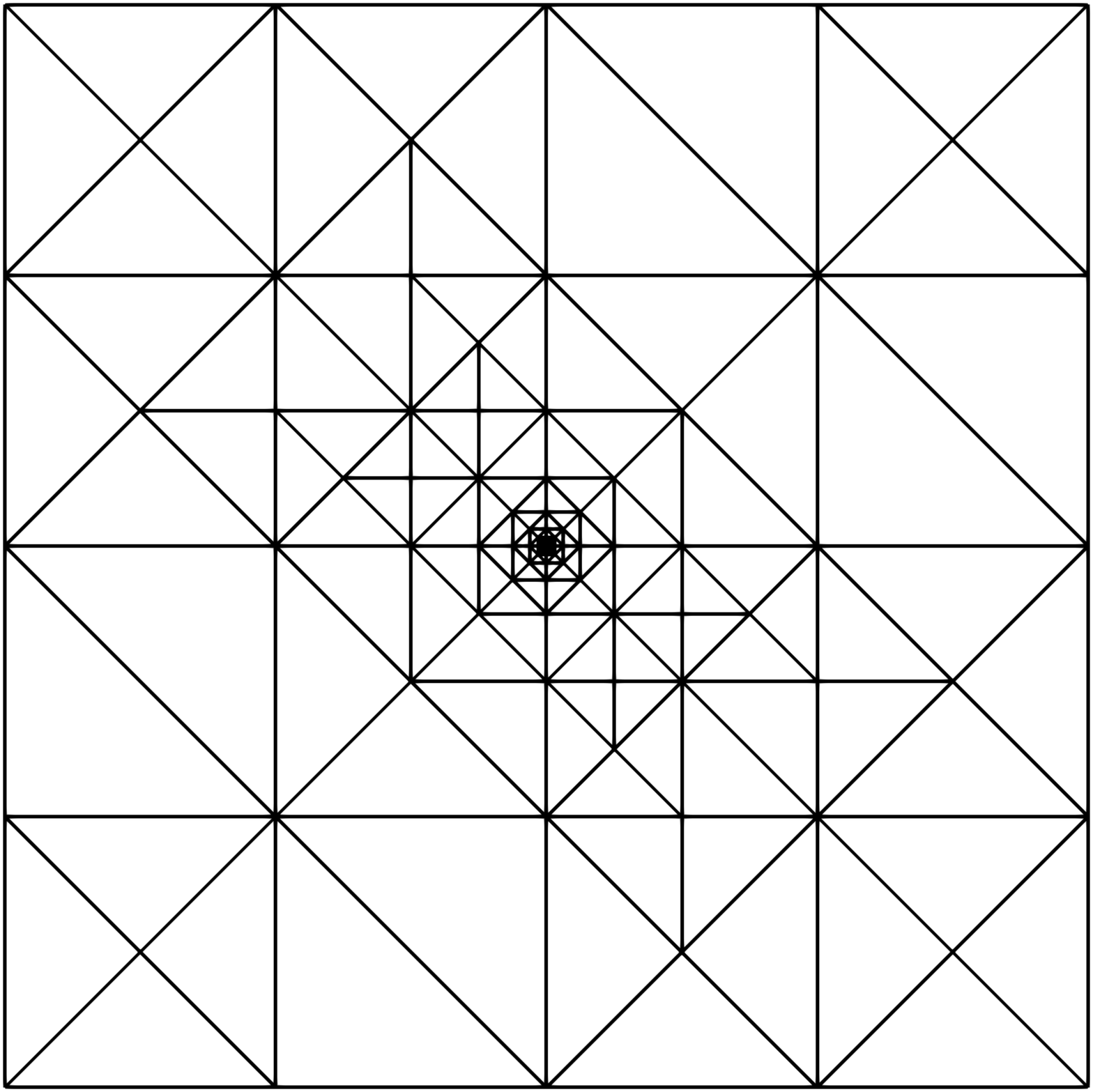} \\
\qquad \tiny{(B.3)}
\end{minipage}
\begin{minipage}[b]{0.24\textwidth}\centering
\includegraphics[trim={25cm 0 25cm 0},clip,width=3.0cm,height=3.4cm,scale=0.5]{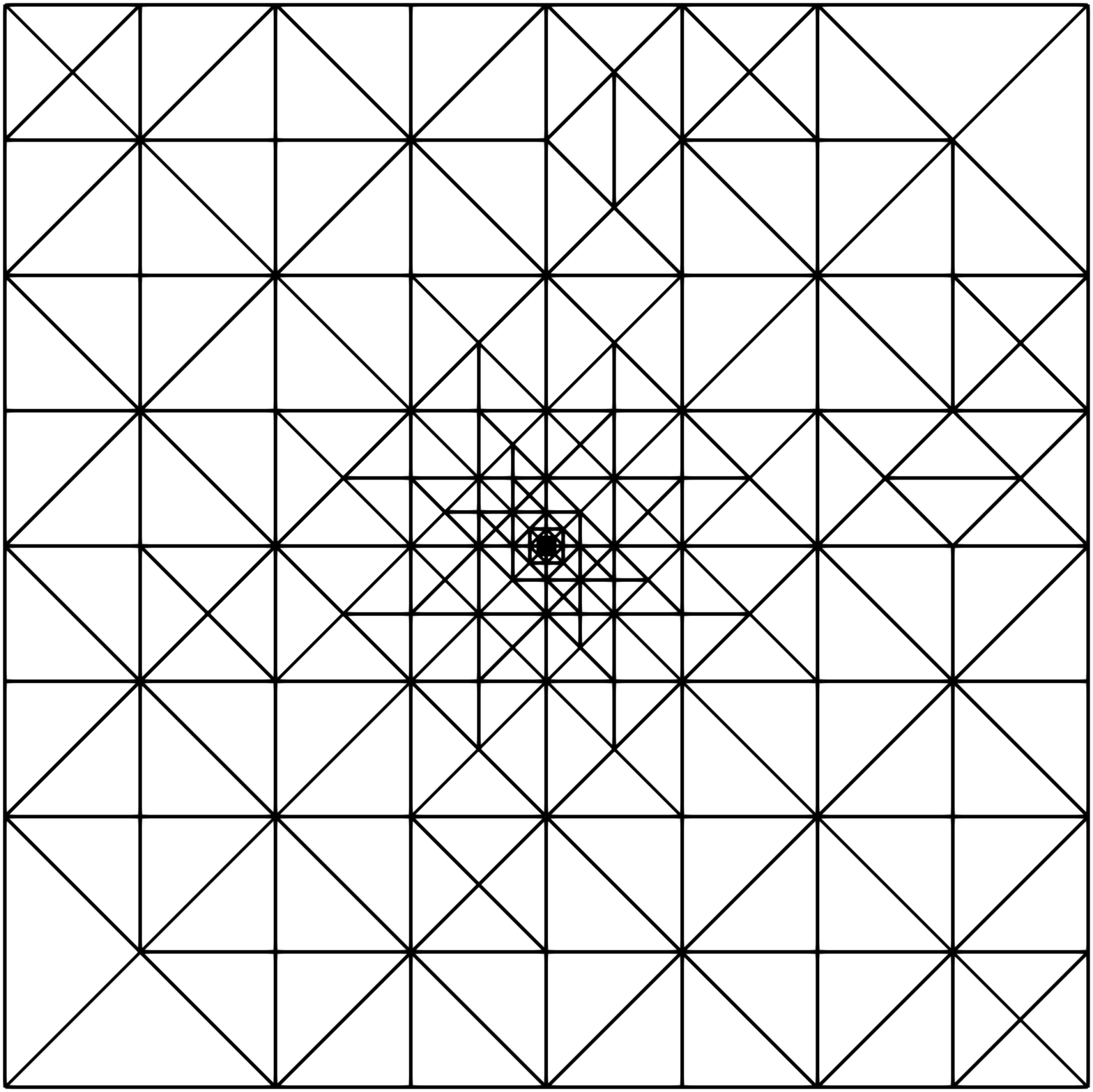} \\
\qquad \tiny{(B.4)}
\end{minipage}
\caption{Example 1: Computational rates of convergence for $\mathcal{E}_{\alpha}(\mathbf{u}_{\T},\mathsf{p}_{\T};\T)$ considering $\alpha \in \{0.25, 0.5, 0.75, 1.0, 1.25, 1.5, 1.75\}$ (B.1) and the meshes obtained after 20 adaptive refinements for $\alpha=0.5$ (156 elements and 85 vertices) (B.2); $\alpha=1.0$ (192 elements and 105 vertices) (B.3);  and $\alpha=1.5$ (304 elements and 167 vertices) (B.4).}
\label{fig:test_01}
\end{figure}

\subsubsection{Non-convex domain}
We consider $\Omega=(-1,1)^2 \setminus[0,1)\times (-1,0]$, $\mathbf{z}=(0.5,0.5)^{\mathsf{T}}$, and $\mathbf{F}=(1,1)^{\mathsf{T}}$. Figure \ref{fig:test_02} shows the results obtained for Example 2. We note that optimal computational convergence rates are obtained for all considered values of the parameter $\alpha$: $\alpha \in \{0.25,0.5,0.75,1.0,1.25,1.5,1.75\}$. We also note that most of the refinement is concentrated around the singular source point and that the geometric singularity for $\alpha \geq 1$ is quickly noticed

\begin{figure}[!ht]
\centering
\psfrag{Ndof(-1)}{\Large $\text{Ndof}^{-1}$}
\psfrag{Est p=0.25}{\Large $\alpha=0.25$}
\psfrag{Est p=0.5}{\Large $\alpha=0.5$}
\psfrag{Est p=0.75}{\Large $\alpha=0.75$}
\psfrag{Est p=1.0}{\Large $\alpha=1.0$}
\psfrag{Est p=1.25}{\Large $\alpha=1.25$}
\psfrag{Est p=1.5}{\Large $\alpha=1.5$}
\psfrag{Est p=1.75}{\Large $\alpha=1.75$}

\begin{minipage}[b]{0.24\textwidth}\centering
\scriptsize{\qquad $\mathcal{E}_{\alpha}(\mathbf{u}_{\T},\mathsf{p}_{\T};\T)$}
\includegraphics[trim={0 0 0 0},clip,width=3.1cm,height=3.2cm,scale=0.5]{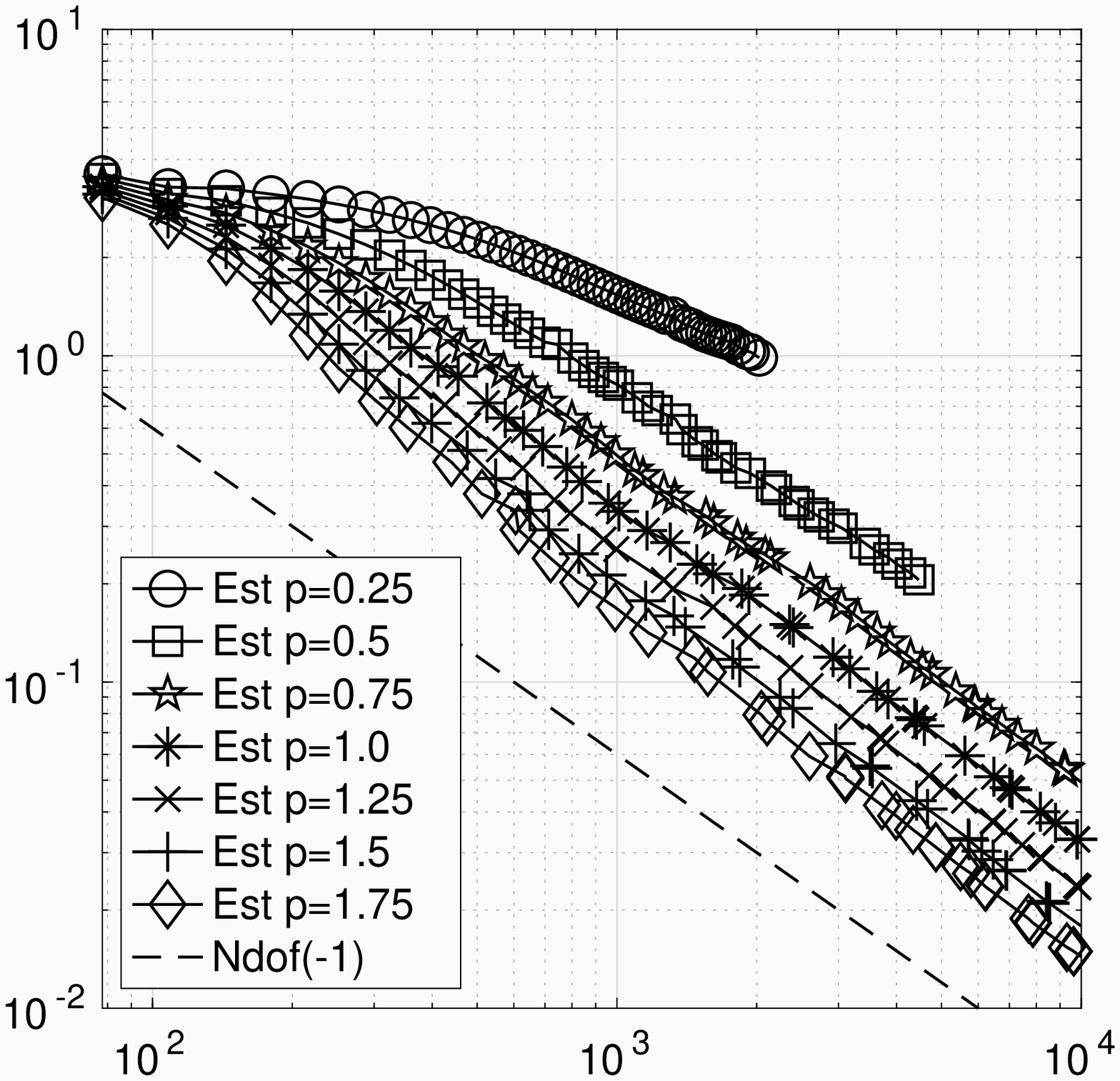} \\
\qquad \tiny{(B.1)}
\end{minipage}
\begin{minipage}[b]{0.24\textwidth}\centering
\includegraphics[trim={25cm 0 25cm 0},clip,width=3.0cm,height=3.4cm,scale=0.5]{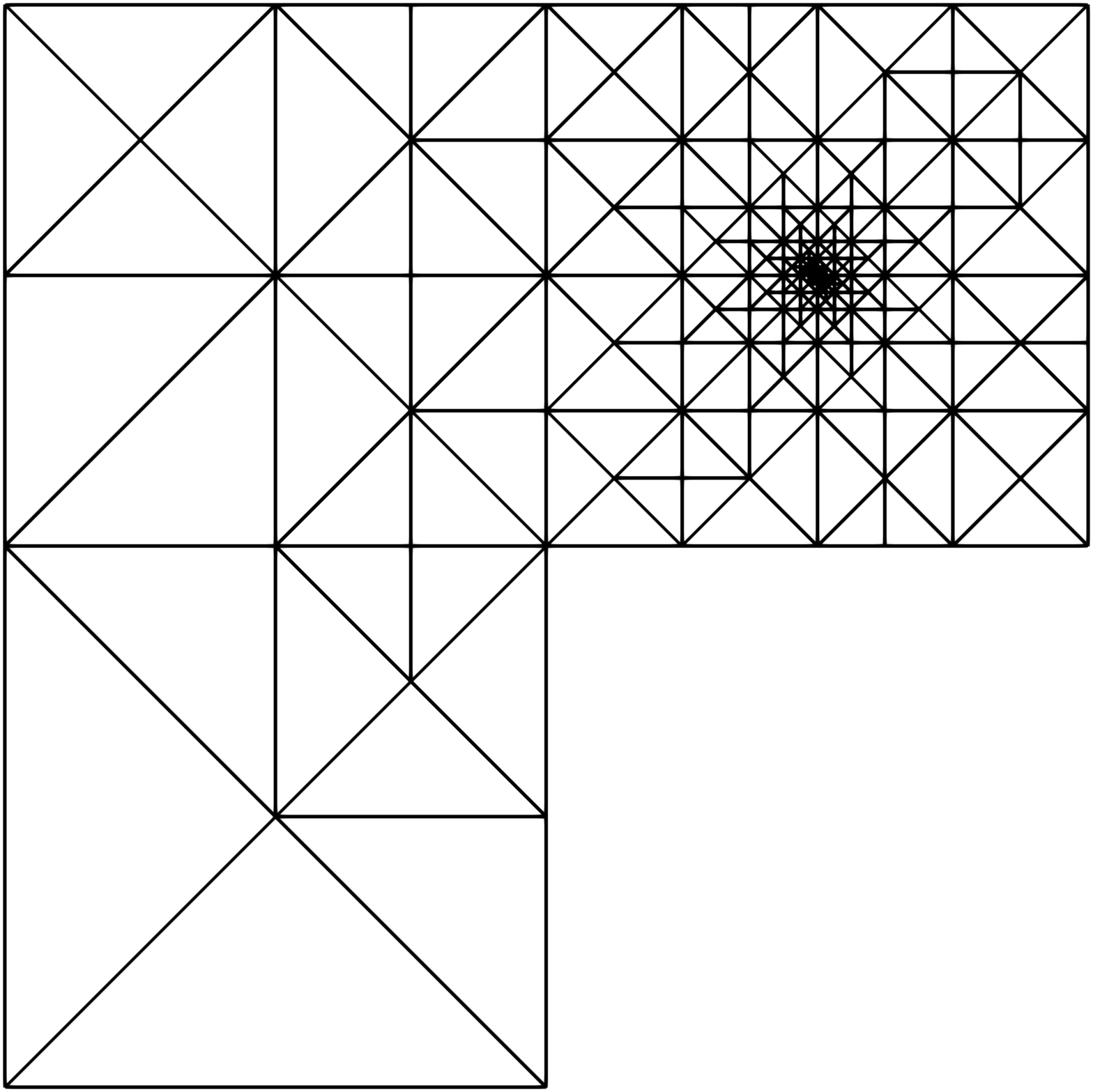} \\
\qquad \tiny{(B.2)}
\end{minipage}
\begin{minipage}[b]{0.24\textwidth}\centering
\includegraphics[trim={25cm 0 25cm 0},clip,width=3.0cm,height=3.4cm,scale=0.5]{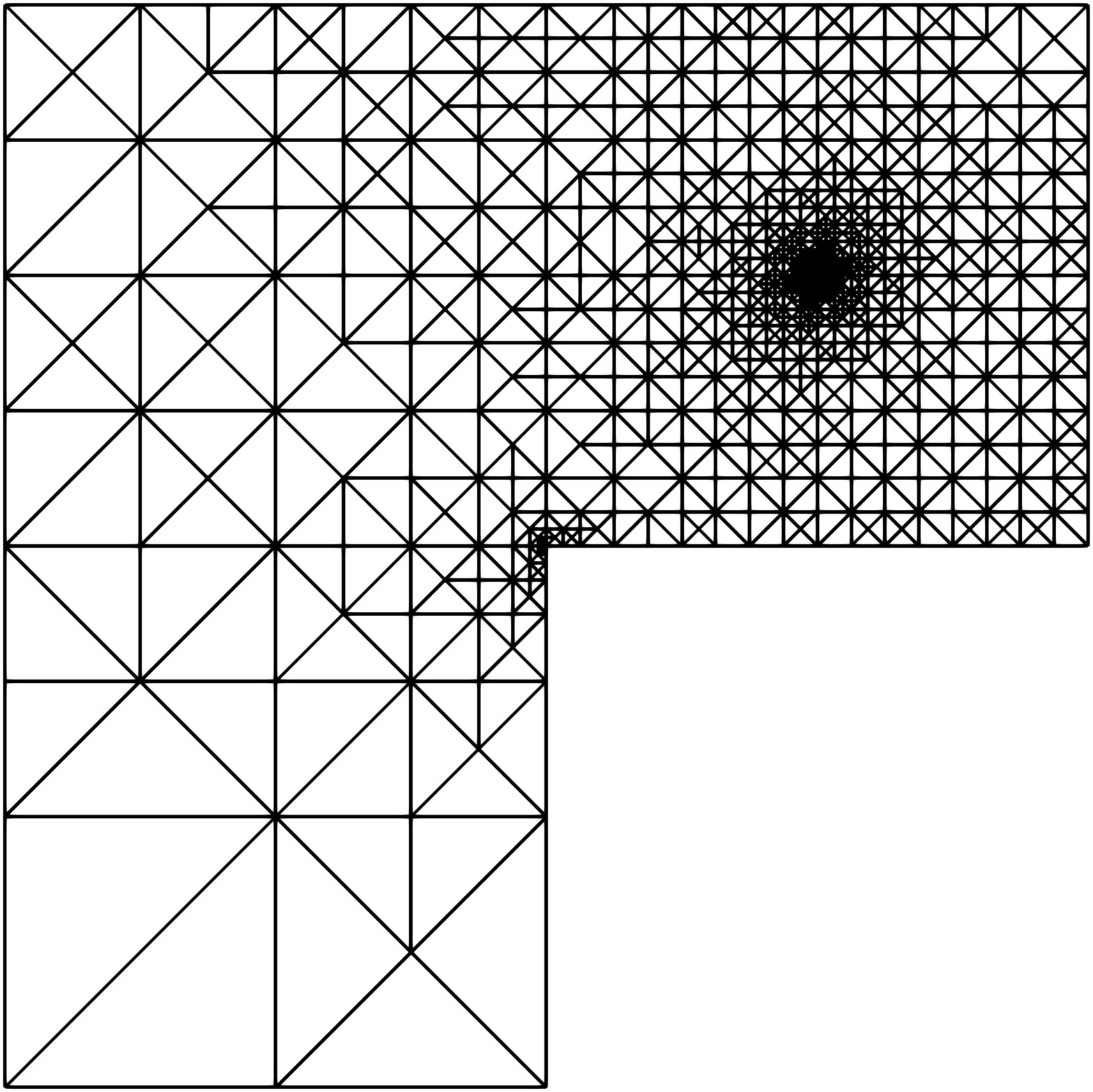} \\
\qquad \tiny{(B.3)}
\end{minipage}
\begin{minipage}[b]{0.24\textwidth}\centering
\includegraphics[trim={25cm 0 25cm 0},clip,width=3.0cm,height=3.4cm,scale=0.5]{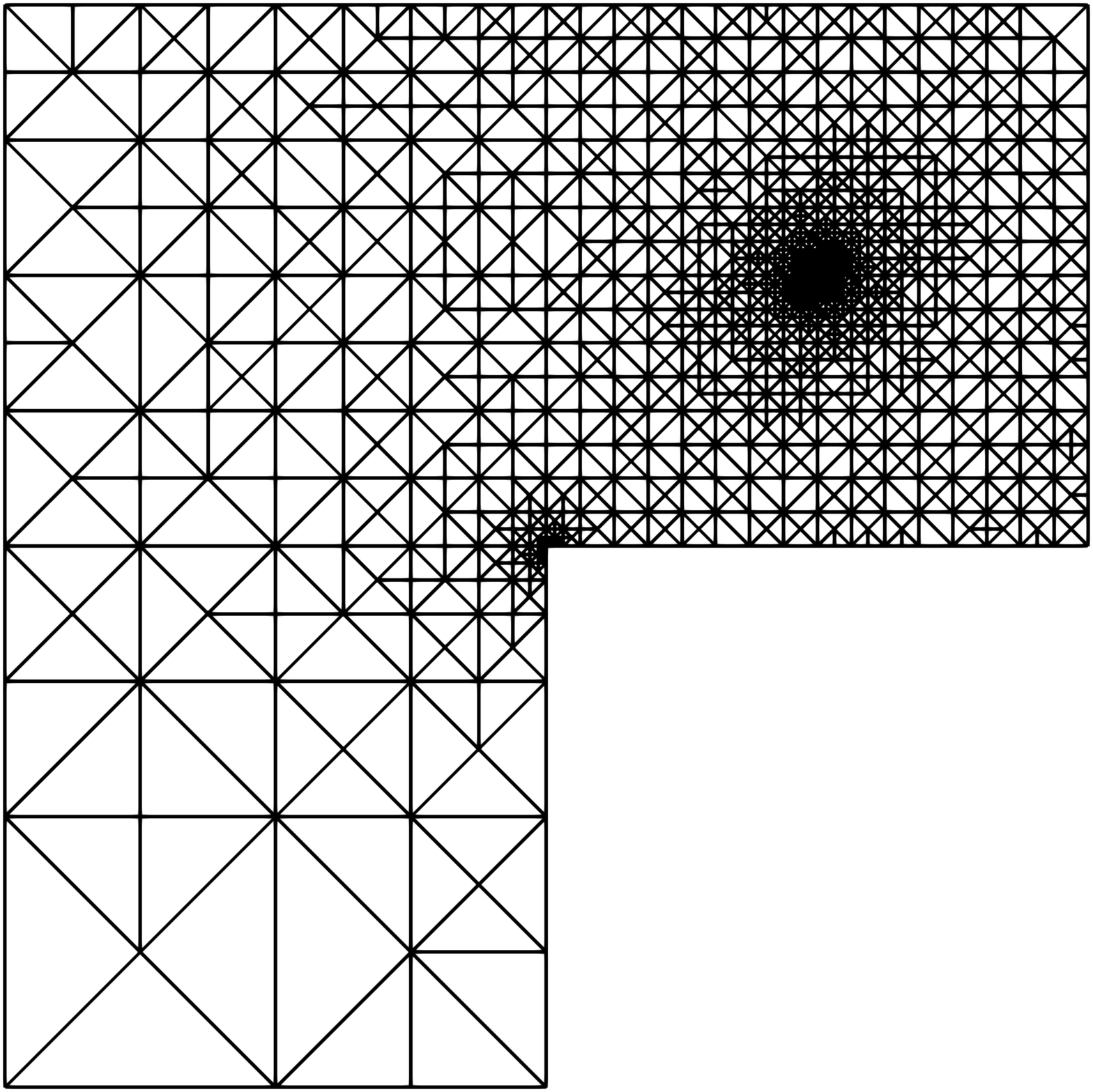} \\
\qquad \tiny{(B.4)}
\end{minipage}
\caption{Example 2: Computational rates of convergence for $\mathcal{E}_{\alpha}(\mathbf{u}_{\T},\mathsf{p}_{\T};\T)$ considering $\alpha \in \{0.25, 0.5, 0.75, 1.0, 1.25, 1.5, 1.75\}$ (B.1) and the meshes obtained after 40 adaptive refinements for $\alpha=0.5$ (534 elements and 280 vertices) (B.2); $\alpha=1.0$ (1917 elements and 994 vertices) (B.3);  and $\alpha=1.5$ (2401 elements and 1247 vertices) (B.4).}
\label{fig:test_02}
\end{figure}

\subsection{A series of Dirac delta points} 
We consider $\Omega=((-1.5,1.5)\times(0,1))\cup((-0.5,0.5)\times(-2,1))$ and go beyond the theory by considering nonhomogeneous Dirichlet boundary conditions and a series of Dirac delta sources on the right-hand side:
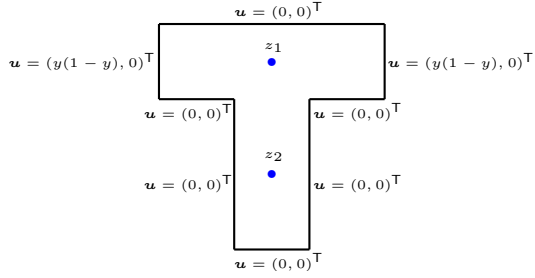
\begin{figure}[!ht]
\centering
\newrgbcolor{xdxdff}{0.49 0.49 1}
\psset{xunit=1.0cm,yunit=1.0cm,algebraic=true,dotstyle=o,dotsize=3pt 0,linewidth=0.8pt,arrowsize=3pt 2,arrowinset=0.25}
\begin{pspicture*}(-3.5,-2.28)(3.5,1.3)
\psline(0.5,-2)(0.5,0)
\psline(0.5,0)(1.5,0)
\psline(1.5,0)(1.5,1)
\psline(1.5,1)(-1.5,1)
\psline(-1.5,1)(-1.5,0)
\psline(-1.5,0)(-0.5,0)
\psline(-0.5,0)(-0.5,-2)
\psline(-0.5,-2)(0.5,-2)
\rput[tl](-3.5,0.62){\tiny $\boldsymbol{u}=(y(1-y),0)^{\mathsf{T}}$}
\rput[tl](-1.7,-0.06){\tiny $\boldsymbol{u}=(0,0)^{\mathsf{T}}$}
\rput[tl](-1.7,-1){\tiny $\boldsymbol{u}=(0,0)^{\mathsf{T}}$}
\rput[tl](-0.5,-2.05){\tiny $\boldsymbol{u}=(0,0)^{\mathsf{T}}$}
\rput[tl](0.55,-1){\tiny $\boldsymbol{u}=(0,0)^{\mathsf{T}}$}
\rput[tl](0.55,-0.05){\tiny $\boldsymbol{u}=(0,0)^{\mathsf{T}}$}
\rput[tl](1.54,0.62){\tiny $\boldsymbol{u}=(y(1-y),0)^{\mathsf{T}}$}
\rput[tl](-0.5,1.3){\tiny $\boldsymbol{u}=(0,0)^{\mathsf{T}}$}
\rput[tl](-0.1,0.75){\tiny $z_1$}
\rput[tl](-0.1,-0.7){\tiny $z_{2}$}
\begin{scriptsize}
\psdots[dotstyle=*,linecolor=blue](0,0.5)
\psdots[dotstyle=*,linecolor=blue](0,-1)
\end{scriptsize}
\end{pspicture*}
\caption{Example 3: T--shaped domain with Dirac delta source points located at $\mathbf{z}_{1}=(0,0.5)$ and $\mathbf{z}_{2}=(0,-1)$.}
\label{Fig:T}
\end{figure}
\begin{equation}\label{eq:modelnew}
-\Delta\mathbf{u} +(\mathbf{u}\cdot\nabla)\mathbf{u}+|\mathbf{u}|\mathbf{u}+ \mathbf{u}+\nabla \mathsf{p}  = \sum_{\mathbf{z}\in\mathcal{Z}} \mathbf{F}_{\mathbf{z}}\delta_{\mathbf{z}}  \text{ in }\Omega,
\end{equation}
where $\mathcal{Z}\subset\Omega$ denotes a finite set with $\#\mathcal{Z}>1$ and $\{\mathbf{F}_{\mathbf{z}}\}_{\mathbf{z}\in\mathcal{Z}}\subset\mathbb{R}^{2}$. In particular, we consider $\mathbf{F}_\mathbf{z}=(1,1)^{\mathsf{T}}$ for all $\mathbf{z}\in \mathcal{Z}$. Let us introduce the weight
\begin{eqnarray}\label{new_weight}
\rho(\mathbf{x})=\left\{
\begin{array}{cc}
\mathsf{d}_{\mathbf{z}}^{\alpha}(\mathbf{x}),
& \exists~\mathbf{z}\in\mathcal{Z}:~|\mathbf{x}-\mathbf{z}|< \frac{d_{\mathcal{Z}}}{2},
\\
1, & |\mathbf{x}-\mathbf{z}|\geq\frac{d_{\mathcal{Z}}}{2}~\forall~\mathbf{z}\in\mathcal{Z},
\end{array}
\right.
\end{eqnarray}
where $d_{\mathcal{Z}}=\min\{\textrm{dist}(\mathcal{Z},\partial\Omega),\min\{|\mathbf{z}-\mathbf{z}'|:\mathbf{z},\mathbf{z}'\in\mathcal{Z},\mathbf{z}\neq \mathbf{z}'\}\}$. With this weight at hand, we modify the definition of the spaces $\mathcal{X}$ and $\mathcal{Y}$ as follows: $\mathcal{X}=\mathbf{H}_{0} ^{1}(\rho,\Omega)\times L^{2}(\rho,\Omega)\setminus\mathbb{R},$ and $\mathcal{Y}=\mathbf{H}_{0} ^{1}(\rho^{-1},\Omega)\times L^{2}(\rho^{-1},\Omega)\setminus\mathbb{R}$. The weight $\rho$ belongs to the Muckenhoupt class $A_2$ (see \cite[Theorem 6]{MR3215609}) and also to the restricted class $A_2(\Omega)$. Define $D_{K,\mathcal{Z}}:=\min_{\mathbf{z}\in\mathcal{Z}}\left\{\max_{\mathbf{x}\in K}|\mathbf{x}-\mathbf{z}|\right\}$. With all these ingredients, we propose the following a posteriori error estimator when considering the Taylor–Hood scheme:
\begin{eqnarray}
\label{eq:new_estimator_e}
\mathcal{D}_{\alpha}(\mathbf{u}_{\T},\mathsf{p}_{\T};\T):= \left(\sum_{K\in\T}\mathcal{D}_{\alpha}^2(\mathbf{u}_{\T},\mathsf{p}_{\T};K)\right)^{\frac{1}{2}},
\end{eqnarray}
where the local errors indicators are such that
\begin{multline}\label{eq:new_indicator_e}
\displaystyle\mathcal{D}_{\alpha}^2(\mathbf{u}_{\T},\mathsf{p}_{\T};K):=h_K^2 D_{K,\mathcal{Z}}^{\alpha} \|\mathcal{R}_K\|_{\mathbf{L}^2(K)}^2
+
h_K D_{K,\mathcal{Z}}^{\alpha} \|\mathcal{J}_{\gamma} \|_{\mathbf{L}^2(\partial K\setminus\partial \Omega)}^2\!\\+\!\|\text{div }\mathbf{u}_{\T}\|_{L^2(\rho,K)}^2
+
\sum_{\mathbf{z}\in\mathcal{Z}\cap K}h_K^{\alpha}|\mathbf{F}_{\mathbf{z}}|^2 .
\end{multline}

Figure \ref{fig:test_03} shows the results obtained for Example 3. It shows the adaptive mesh obtained after 60 iterations, the streamlines associated with the velocity field $\mathbf{u}_{\T}$, the pressure contours, and the velocity and pressure elevations. It can be observed that the developed a posteriori error estimator achieves an optimal computational convergence rate and that most of the refinement is concentrated on the singular sources and the geometric singularities involved.

\begin{figure}[!ht]
\centering
\psfrag{Ndof(-1)}{\Large $\text{Ndof}^{-1}$}
\psfrag{Est p=1.0}{\Large $\alpha=1.0$}

\begin{minipage}[b]{0.25\textwidth}\centering
\scriptsize{\qquad $\mathcal{D}_{\alpha}(\mathbf{u}_{\T},\mathsf{p}_{\T};\T)$}
\includegraphics[trim={0 0 0 0},clip,width=3.1cm,height=3.3cm,scale=0.5]{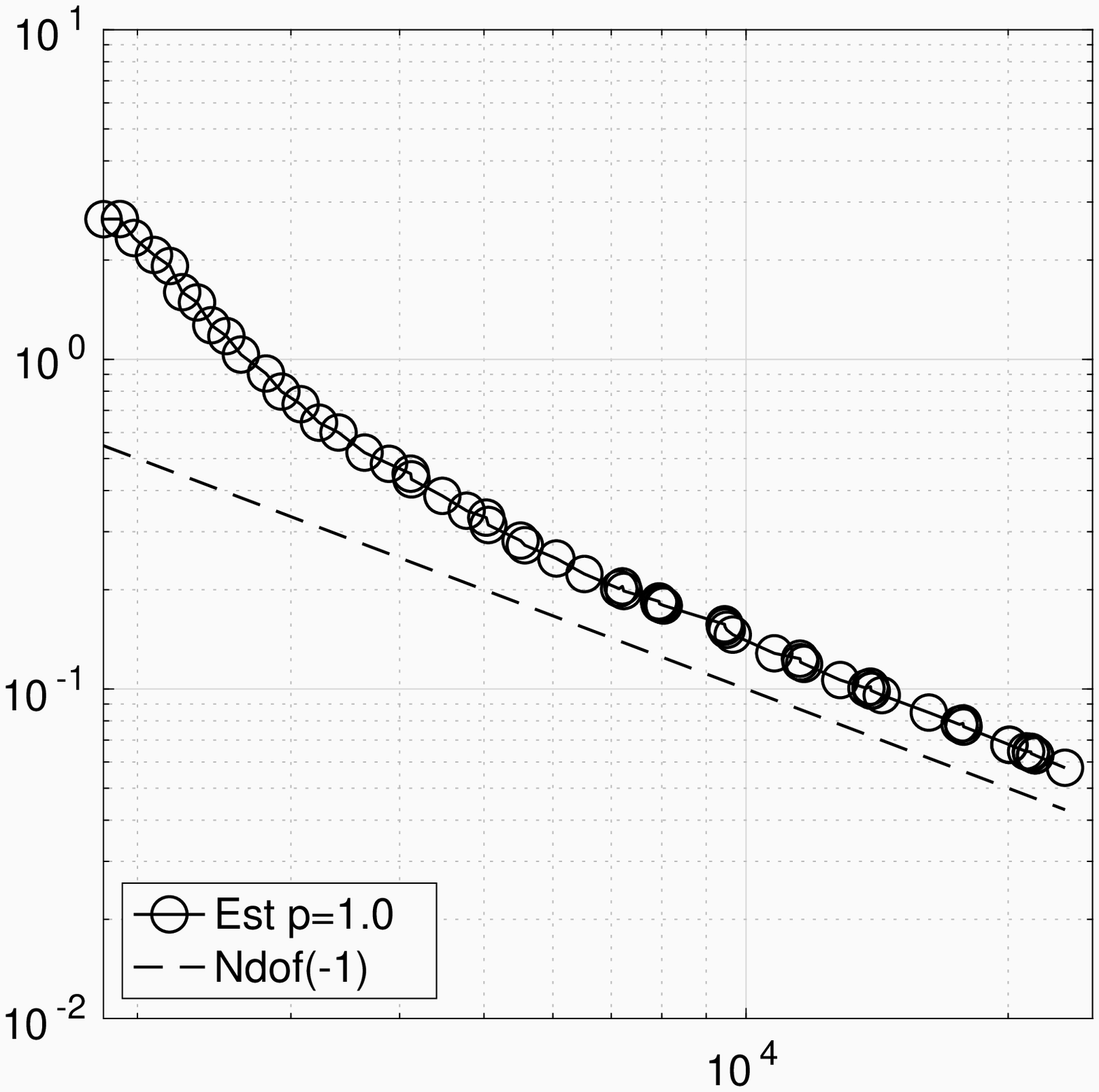} \\
\qquad \tiny{(C.1)}
\end{minipage}
\begin{minipage}[b]{0.4\textwidth}\centering
\includegraphics[trim={25cm 0 25cm 0},clip,width=3.5cm,height=3.5cm,scale=0.5]{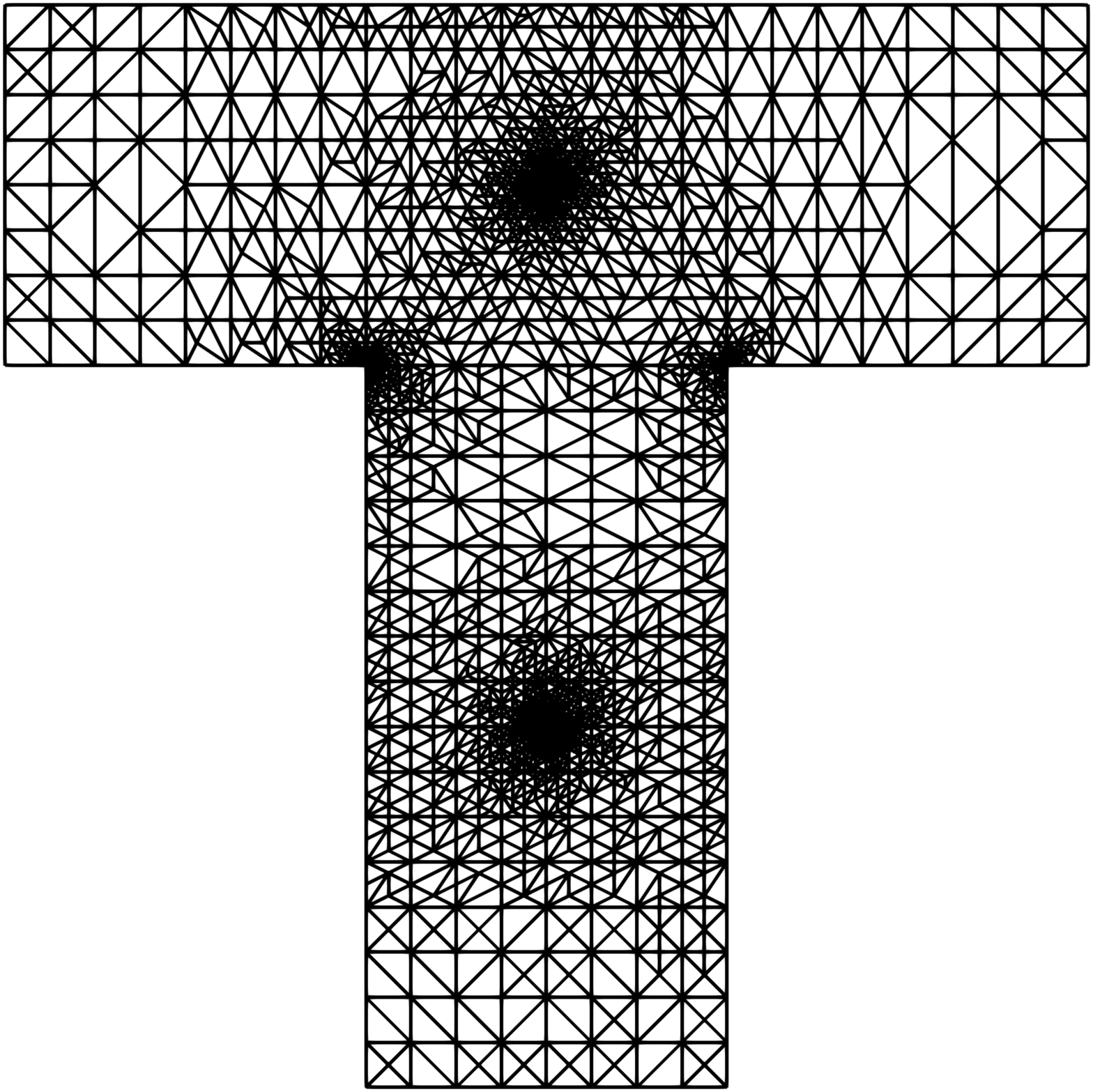} \\
\qquad \tiny{(C.2)}
\end{minipage}
\begin{minipage}[b]{0.3\textwidth}\centering
\includegraphics[trim={20cm 0 20cm 0},clip,width=3.5cm,height=3.5cm,scale=0.5]{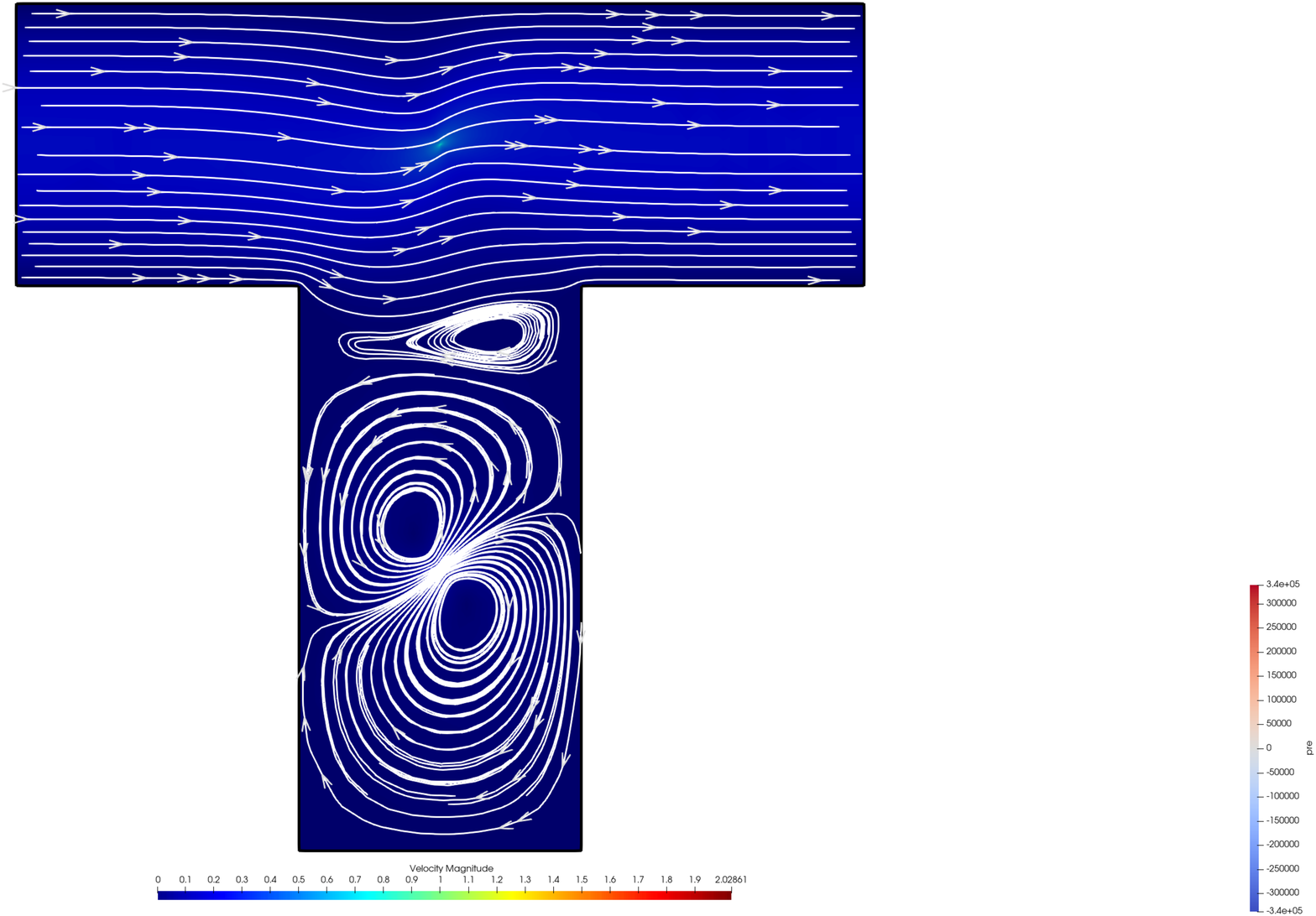} \\
\qquad \tiny{(C.3)}
\end{minipage}
\\
\begin{minipage}[b]{0.25\textwidth}\centering
\includegraphics[trim={5cm 0 5cm 0},clip,width=3.2cm,height=3.6cm,scale=0.5]{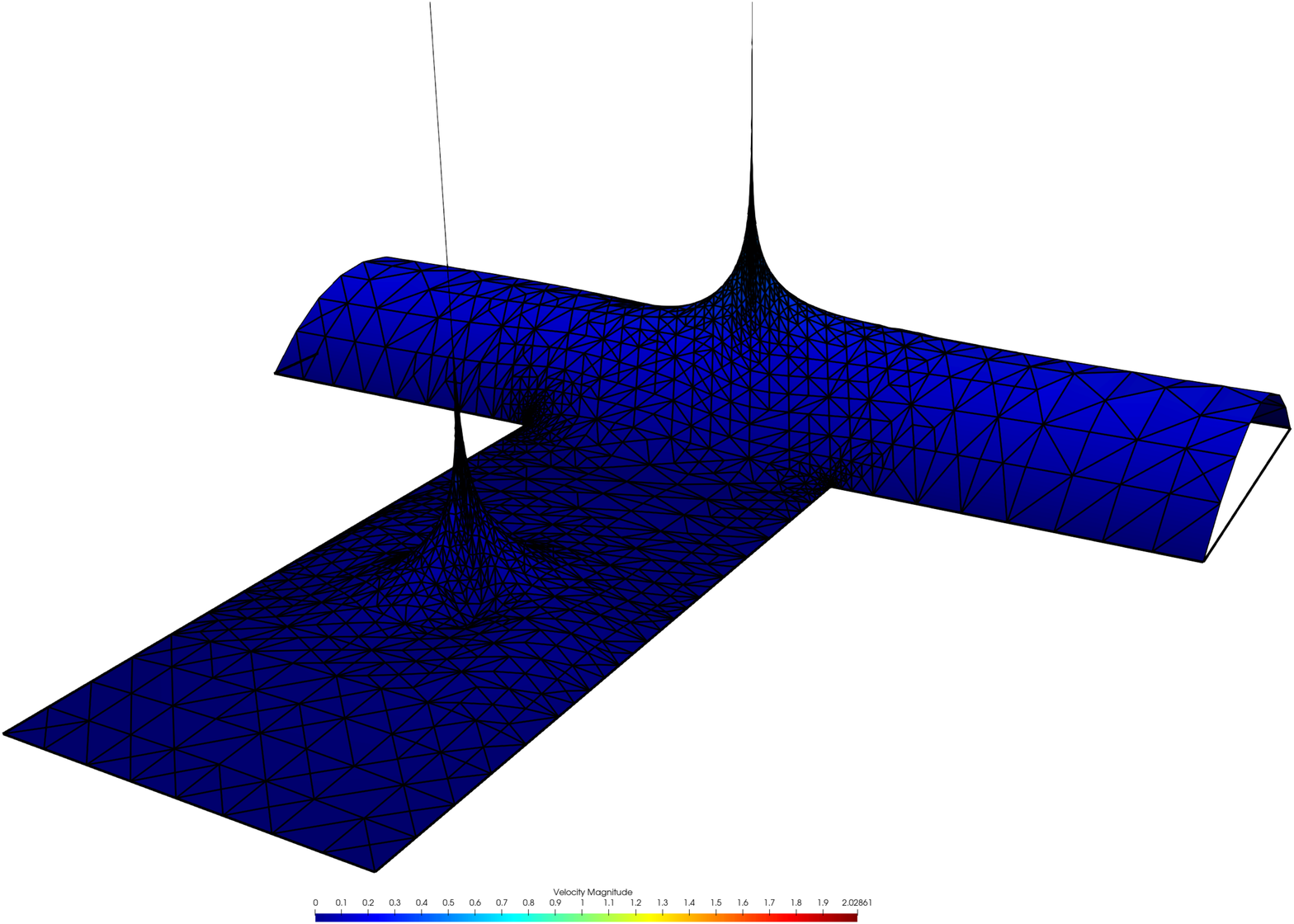} \\
\qquad \tiny{(C.4)}
\end{minipage}
\begin{minipage}[b]{0.4\textwidth}\centering
\includegraphics[trim={20cm 0 20cm 0},clip,width=3.5cm,height=3.6cm,scale=0.5]{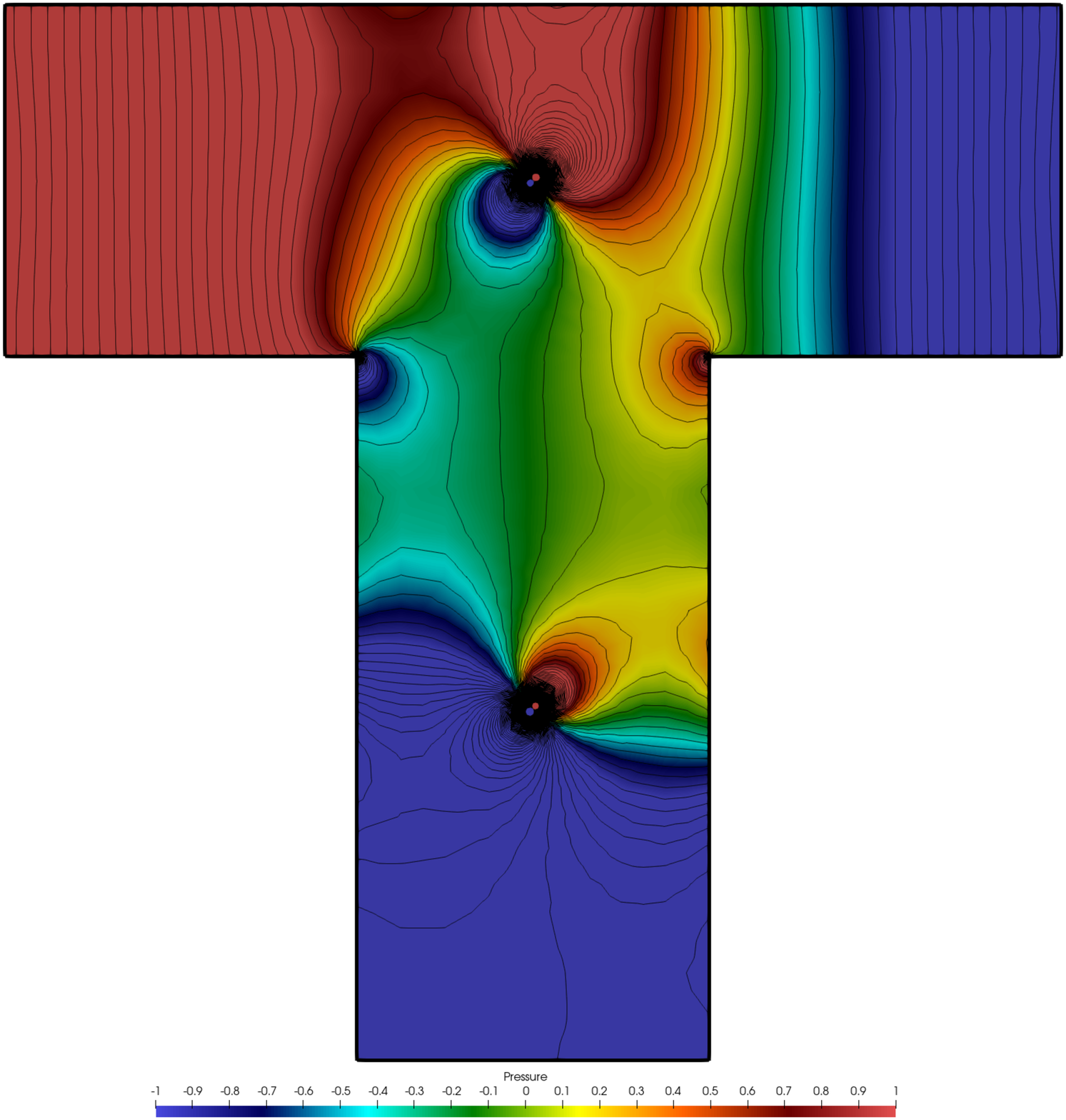} \\
\qquad \tiny{(C.5)}
\end{minipage}
\begin{minipage}[b]{0.3\textwidth}\centering
\includegraphics[trim={5cm 0 5cm 0},clip,width=3.6cm,height=3.6cm,scale=0.5]{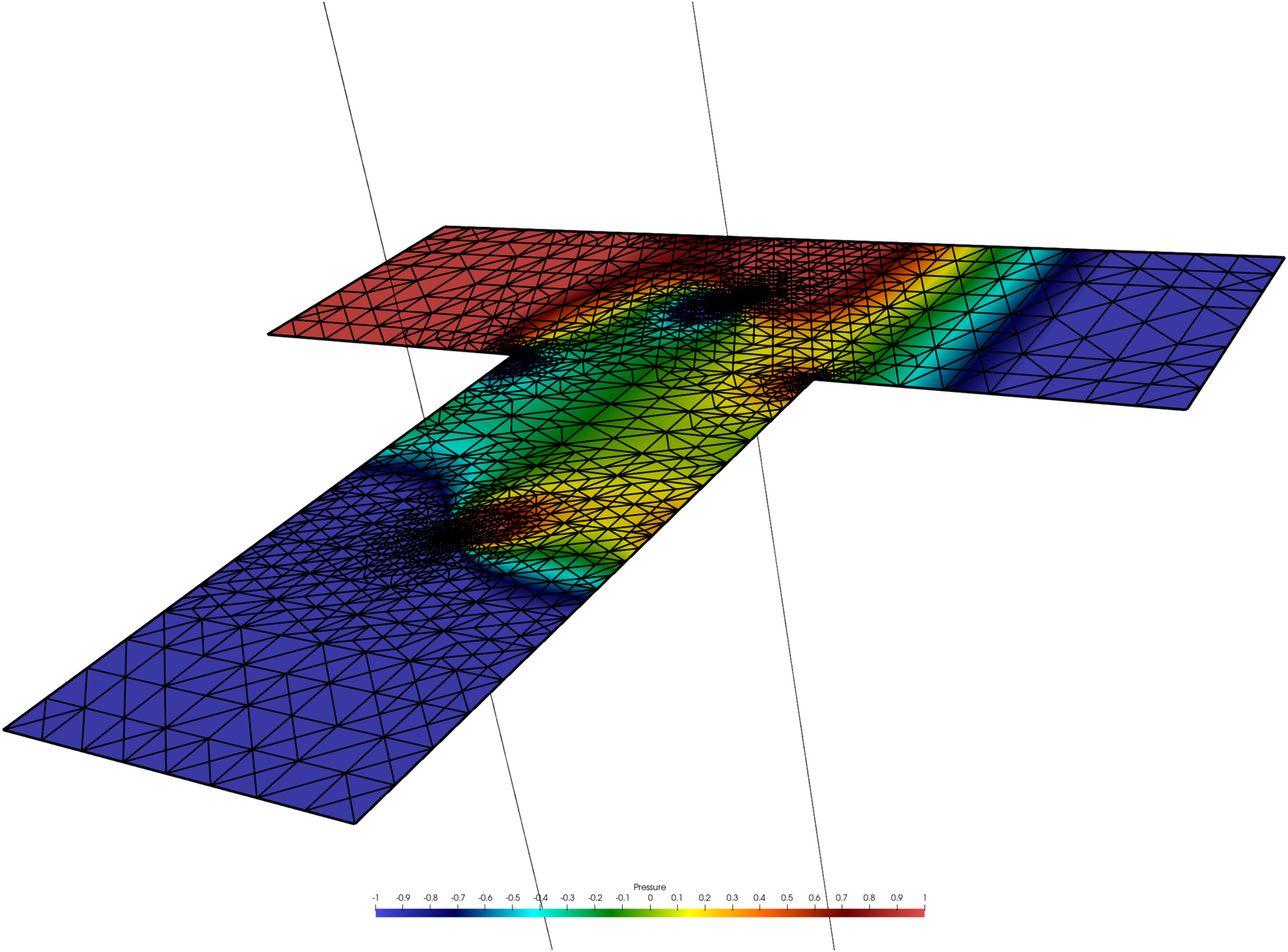} \\
\qquad \tiny{(C.6)}
\end{minipage}
\caption{Example 3: Computational rate of convergence for $\mathcal{D}_{1.0}(\mathbf{u}_{\T},\mathsf{p}_{\T};\T)$ (C.1); the mesh obtained after 60 adaptive refinements (4378 elements and 2263 vertices) (C.2); streamlines for $|\mathbf{u}_{\T}|$ (C.3); elevation for $|\mathbf{u}_{\T}|$ (C.4); pressure contour (C.5); and elevation for $\mathsf{p}_{\T}$ (C.6).}
\label{fig:test_03}
\end{figure}
\section{Conflict of interest}
The authors have not disclosed any competing interests.

\section{Data availability}
The datasets generated during and/or analyzed during the current study are available from the corresponding author on reasonable request.
\bibliographystyle{siamplain}
\bibliography{biblio}

\end{document}